\documentclass[a4paper,11pt,reqno]{article}

\usepackage[body=17cm,top=3cm,bottom=4cm]{geometry}
\usepackage[english]{babel}
 \usepackage{amsrefs}
\usepackage{amsmath,amsfonts,amsthm,amscd,amssymb,mathrsfs,accents}
\usepackage[colorlinks=true]{hyperref}
\usepackage{lipsum}
\usepackage{mathtools}
\usepackage{cases}
\usepackage{bbm}
\usepackage{graphicx}
\usepackage{subfig}
\usepackage{float}
\usepackage{empheq}
\usepackage{enumerate}
\usepackage{appendix}
\usepackage{times}
\usepackage{microtype}
\usepackage{tikz}
\usepackage{geometry}
\usepackage{hyperref}
\usepackage{shuffle}
\usepackage{ stmaryrd }
\usepackage{authblk}

\usepackage{tikz}
\usepackage{geometry}
\usetikzlibrary{arrows,chains,matrix,positioning,scopes}

\usetikzlibrary{shapes.misc}
\usetikzlibrary{shapes.symbols}
\tikzset{
	xi/.style={circle,fill=blue!10,draw=black,inner sep=0pt,minimum size=1.2mm},
	xib/.style={circle,fill=blue!10,draw=black,inner sep=0pt,minimum size=1.6mm},
	not/.style={circle,fill=black,draw=black,inner sep=0pt,minimum size=0.5mm},
	>=stealth,
	}
\makeatletter
\def\DeclareSymbol#1#2#3{\expandafter\gdef\csname MH@symb@#1\endcsname{\tikz[baseline=#2,scale=0.15]{#3}}}
\def\<#1>{\csname MH@symb@#1\endcsname}

\theoremstyle{plain}
\newtheorem{lemma}{Lemma}[section]
\newtheorem{theorem}[lemma]{Theorem}
\newtheorem{corollary}[lemma]{Corollary}
\newtheorem{proposition}[lemma]{Proposition}

\theoremstyle{definition}
\newtheorem*{notation*}{Notation}
\newtheorem*{plan*}{Plan of the paper}
\newtheorem*{ackno*}{Acknowledgements}
\newtheorem{definition}[lemma]{Definition}
\theoremstyle{remark}
\newtheorem{remark}[lemma]{Remark}

\makeatletter
\newcommand{\VERT}{\vert\kern-0.25ex\vert\kern-0.25ex\vert}

\newcommand{\R}{\mathbb{R}}
\newcommand{\Z}{\mathbb{Z}}
\newcommand{\N}{\mathbb{N}}

\newcommand{\dd}{\mathrm{d}}
\newcommand{\D}{\mathcal{D}}
\newcommand{\T}{\mathbb{T}}
\newcommand{\B}{\mathcal{B}}
\newcommand{\X}{\mathcal{X}}
\newcommand{\CC}{\mathcal{C}}
\newcommand{\TS}{\mathcal{T}}
\newcommand{\I}{\mathcal{I}}
\newcommand{\J}{\mathcal{J}}

\newcommand{\U}{\mathcal{U}}
\newcommand{\W}{\mathcal{W}}
\newcommand{\SC}{\mathscr{S}}
\newcommand{\M}{\mathscr{M}}
\newcommand{\eps}{\varepsilon}
\newcommand{\1}{\mathbbm{1}}
\newcommand{\F}{\mathcal{F}}
\newcommand{\E}{\mathbb{E}}
\newcommand{\PR}{\mathbb{P}}
\newcommand{\G}{\mathcal{G}}

\newcommand{\y}{\tilde{y}}

\newcommand{\RS}{\mathcal{R}}
\newcommand{\PK}{\mathcal{P}}
\newcommand{\Sol}{\mathcal{S}}
\newcommand{\Hi}{\mathcal{H}}
\newcommand{\DD}{\mathbb{D}}
\newcommand{\loc}{\mathrm{loc}}

\makeatletter
\DeclareRobustCommand{\cev}[1]{%
  \mathpalette\do@cev{#1}%
}
\newcommand{\do@cev}[2]{%
  \fix@cev{#1}{+}%
  \reflectbox{$\m@th#1\vec{\reflectbox{$\fix@cev{#1}{-}\m@th#1#2\fix@cev{#1}{+}$}}$}%
  \fix@cev{#1}{-}%
}
\newcommand{\fix@cev}[2]{%
  \ifx#1\displaystyle
    \mkern#23mu
  \else
    \ifx#1\textstyle
      \mkern#23mu
    \else
      \ifx#1\scriptstyle
        \mkern#22mu
      \else
        \mkern#22mu
      \fi
    \fi
  \fi
}

\makeatother

\numberwithin{equation}{section}

\begin{document}
\title{Malliavin calculus for regularity structures: the case of gPAM}

\date{}
%
%
\author[1]{G. Cannizzaro}
\author[1,2]{P.K. Friz}
\author[3]{P. Gassiat}
{\tiny
\affil[1]{{Institut f\"ur Mathematik, Technische Universit\"at Berlin}}
\affil[2]{{Weierstra\ss --Institut f\"ur Angewandte Analysis und Stochastik, Berlin}}
\affil[3]{{CEREMADE, Universit\'e Paris-Dauphine}}
}
\maketitle

\begin{abstract}
Malliavin calculus is implemented in the context of [M. Hairer, A theory of regularity structures, Invent. Math. 2014]. This involves some constructions of independent interest, notably an extension of the structure
which accomodates a robust, and purely deterministic, translation operator, in $L^2$-directions, between ``models". In the concrete context of the generalized parabolic Anderson model in 2D - one of the singular SPDEs
discussed in the afore-mentioned article - we establish existence of a density at positive times.
\end{abstract}

\tableofcontents

\section{Introduction}

Malliavin calculus \cite{Ma97} is a classical tool for the analysis of stochastic (partial) differential equations, e.g. \cite{Nua, SaSo} and the references therein. The aim of this paper is to explore Malliavin calculus in the context of Hairer's {\it regularity structures} \cite{Hai}, a theory designed to provide a solution theory for certain ill-posed stochastic partial differential equations (SPDEs) typically driven by Gaussian (white) noise. By now, there is an impressive list of such equations that can be handled in this framework, many well-known from the (non-rigorous) physics literature: KPZ, parabolic Anderson model, stochastic quantization equation, stochastic Navier--Stokes, ...

At this moment, and despite a body of general results and a general d\'emarche, each equation still needs some tailor-made analysis, especially when it comes to renormalization \cite[Sec. 8,9]{Hai} and convergence of renormalized approximations \cite[Sec.10]{Hai}, in the context of Gaussian white noise. For this reason, we focus on one standard example of the theory - the generalized parabolic Anderson model (gPAM) - although an effort is made throughout, with regard to future adaptions to other equations, to emphasize the main governing principles of our results. To be specific, recall that gPAM is given (formally!) by the following non-linear SPDE 
\begin{equation}\label{eq:gPAM}
(\partial_t - \Delta) u = g(u) \xi,\qquad u(0,\cdot)=u_0(\cdot).
\end{equation}
for $t\geq 0$, $g$ sufficiently smooth, spatial white noise $\xi=\xi(x,\omega)$ and fixed initial data  $u_0$. Assuming periodic boundary conditions, write $x \in \T^d$, the $d$-dimensional torus. Now a.s. the noise is a Gaussian random distribution, of H\"older regularity $\alpha<-d/2$. Standard reasoning suggests that $u$ (and hence $g(u)$) has regularity $\alpha+2$, due to the regularization of the heat-flow. But the product of two such 
H\"older distributions is only well-defined, if the sum of the regularities is strictly positive - which is the case in dimension $d=1$ but not when $d=2$. Hence we focus on gPAM in dimension $d=2$, along \cite{Hai} and also Gubinelli et al. \cite{gip} in a different (paracontrolled) framework.

A necessary first step in employing Malliavin calculus in this context is an understanding of the perturbed equation, formally given by
\begin{equation}\label{eq:gPAMh}
(\partial_t - \Delta) u^h = g(u^h) (\xi+h),\qquad u(0,\cdot)=u_0^h(\cdot)
\end{equation}
where $h \in \mathcal{H}$, the Cameron--Martin space, nothing but $L^2$ in the Gaussian (white) noise case. Proceeding on this formal level, setting $v^h = \tfrac{\partial}{\partial \eps} \{ u^{\eps h}  \} | _{\eps = 0}$ leads us to the following {\it tangent equation}  

\begin{equation}\label{eq:DergPAMInt}
(\partial_t-\Delta) v^h=g(u)h+v^hg'(u)\xi,\qquad v^h_0(\cdot)=0.
\end{equation}

Readers familiar with Malliavin calculus will suspect (correctly) that $v^h = \langle Du,h \rangle_{\mathcal{H}}$, where $Du$ is the Malliavin derivative (better: $\mathcal{H}$-derivative) of  $u$, solution to gPAM as given in (\ref{eq:gPAM}). Once in possession of a Malliavin differentiable random variable, such as $u=u(t,x; \omega)$ for a fixed $(t,x)$, non-degeneracy of $\langle Du, Du \rangle_{\mathcal{H}}$ will guarantee existence of a density. This paper is devoted to implementing all this rigorously  in the 
context of regularity structures. We have, loosely stated,

\begin{theorem} In spatial dimension $d=2$, equations (\ref{eq:gPAM}),(\ref{eq:gPAMh}),(\ref{eq:DergPAMInt}) can be solved in a consistent, renormalized sense (as reconstruction of modelled distributions, on a suitably extended regularity structure). If the solution $u$ to (\ref{eq:gPAM}) exists on $[0,T)$, for some explosion time $T=T(u_0; \omega)$, then so does then $v^h$, for any $h \in L^2$, and $v^h$ is indeed the $\mathcal{H}$-derivative of $u$ in direction $h$. At last, for suitably non-degenerate $g(.)$, conditional on $0<t<T$, and for fixed $x \in \T^2$, the solution $u=u(t,x;\omega)$ to gPAM admits a density.
\end{theorem} 

Let us highlight some of the technical difficulties and key aspects of this work.

\newpage

\begin{itemize}

\item All equations under consideration are ill-posed. Solutions $u,v^h$ to (\ref{eq:gPAM}), (\ref{eq:DergPAMInt}) can be understood as limit of mollified, renormalized equations, based on, for suitable (divergent) constants $C_\eps$, 
\begin{equation}\label{eq:rengPAMintro}
\partial_t \tilde{u}_\eps=\Delta \tilde{u}_\eps+g(\tilde{u}_\eps)(\xi_\eps-C_{\eps}g'(\tilde{u}_\eps)), 
\end{equation}
and 
\begin{equation} 
\partial_t \tilde{v}_\eps^{h}=\Delta \tilde{v}_\eps^{h}+g(\tilde u_\eps)h+\tilde{v}_\eps^{h}\Big(g'(\tilde u_\eps)\xi_\eps-C_{\eps}\Big( (g'(\tilde u_\eps))^2+g''(\tilde u_\eps)g(\tilde u_\eps)\Big)\Big), 
\end{equation}
respectively.\footnote{Throughout the text, upper tilde ($\sim$) indicates renormalization.} That said, following \cite{Hai} solutions are really constructed as fixed points to abstract equations.

\item While one may expect that $u(\omega+h) = u^h (\omega)$, our analysis relies on the ability to perform this translation in a purely analytical manner. In particular, writing $K\xi \in C^{\alpha+2}$ (think: $C^{1^-}$) for the solution of the linearized problem ($g\equiv1$), one clearly has to handle products such as $(K\xi)h$, where $h \in L^2 \subset C^{-1}$. Unfortunately, as product of H\"older distributions this is not well-defined, there is no easy way out, for Hairer's theory is very much written in a H\"older
 setting.\footnote{To wit, a model on the polynomial regularity structure represents precisely a H\"older function; a model on the tensor (Hopf) algebra represents precisely a H\"older rough path, cf. \cite{FH14}.} On the other hand, classical harmonic analysis tells us that the product  
 \begin{equation}  \label{equ:BesovProduct}
        \CC^{\alpha+2} \times H^\beta \to \CC^\gamma
\end{equation}
is well-defined provided that $\alpha+\beta+2>0$ and $\gamma=\min\{\beta-\frac{d}{2},\alpha+2+\beta-\frac{d}{2}\}$ (see Theorems 2.82, 2.85 and Proposition 2.71 in~\cite{BCD}), but one has to step outside the Besov-$\infty$ (i.e. H\"older) scale. A key technical aspect of our work is to develop the necessary estimates for H\"older models in gPAM, when paired with $h \in L^2 \equiv H^0$, which in turn requires some delicat wavelet analysis. (Remark that we could have considered perturbation $h \in H^\beta$ for some $\beta <0$, 
which en passant shows that the effective tangent space to gPAM is larger than the Cameron-Martin space.\footnote{A similar remark for SDEs is due to Kusuoka \cite{Ku93}, revisited by rough path methods in  \cite{FV06}.}) 

\item In order to provide an abstract formulation of (\ref{eq:gPAMh}),(\ref{eq:DergPAMInt}) in the spirit of Hairer, one cannot use the standard gPAM structure as given in \cite{Hai}. Indeed, the very presence of a perturbation $h \in L^2$ forces us to introduce 
a new symbol $\color{blue}H$, which in turn induces several more symbols, such as ${\color{blue}\I(\Xi)H}$, corresponding to $(K\xi)h$. Key notions such as structure group and renormalization group have to be revisited for the enlarged structure. In particular, it is seen that renormalization commutes with (abstract) translation ${\color{blue}\Xi} \mapsto{\color{blue}\Xi+H}.$

\item Non-degeneracy of $\langle Du, Du \rangle_{\mathcal{H}}$ is established by a novel {\it strong maximum principle} for solutions to linear equations -- on the level of modelled distributions -- which may be of independent interest.
Indeed, the argument (of Theorem \ref{thm:Positive}), despite written in the context of gPAM, adapts immediately to other situations, such as the linear multiplicative stochastic heat equation in dimension $d=1$ (cf. \cite{HP}) where we recover Mueller's work, \cite{Mueller}, and to the linear PAM equation in dimensions $d=2,3$ for which the result appears to be new.  Remark that maximum principles have played no role so far in the study of singular SPDE \`a la Hairer (or Gubinelli et al.) - presumably for the simple reason that a maximum principle hings on the second order nature of a PDE, whereas the local solution theory of singular SPDEs is mainly concerned with the regularization properties of convolution with singular kernels (or Fourier multipliers) making no second order assumptions whatsoever.

\item  We have to deal with the fact \cite{Hai} that solutions are only constructed locally in time. This entails  a number of technical localization arguments such as Lemma  \ref{lem:Xn}, written in a way that is amenable to adaptions to equations other then gPAM. In specific case of (non-linear) gPAM, however, explosion can only happen in $L^\infty$ (cf. Proposition \ref{prop:solgPAM} with $\eta=0$, based on \cite{Hai,HP}, first observed in \cite{gip}.)  Appealing again to a maximum principle, we observe non-explosion for a rich class of non-linear $g$ {\it with sufficiently large zeros}. In particular, the reader may safely assume explosion time $T(\omega) = + \infty$ at first reading. 

\end{itemize}

We briefly comment on previous related works, alternative approaches and possible extensions.
\newline

{\bf From rough paths to regularity structures.} As is well-known, the theory of regularity structures was inspired by rough path theory, with many precise correspondences:  rough path $\to$ {\it model}, controlled rough path $\to$ {\it modelled distribution}, rough integration $\to$ {\it reconstruction map} etc. In the same spirit, our investigation of Malliavin calculus within regularity structures builds on previously obtained insights in the context of rough differential equations (RDEs) driven by Gaussian rough paths  \cite{CFV09,CF10,CF11,CHLT15, FGGR}. In this context, the natural tangent space of $p$-rough paths consists of paths of finite $q$-variation and (\ref{equ:BesovProduct}) may be regarded as a form of Young's inequality, valid provided $1/p+1/q>1$. (For general Gaussian rough paths it then remains to understand when $\mathcal{H} \subset C^{q-\text{var}}$, cf. \cite{FGGR}.) In a sense, in the SPDE setting, Besov-$\infty$ (resp. -$2$) spaces provide a reasonable replacement for $p$ (resp. $q$) variation spaces. A point of departure between rough paths and regularity structure concerns $\langle Du,h \rangle_{\mathcal{H}}$, where the explicit representation in terms of the Jacobian of the flow, much used in the SDE/RDE context, has no good correspondence and different arguments are needed.

{\bf Paracontrolled distributions versus regularity structures and beyond gPAM.}  The renormalized solution to gPAM was also obtained in the paracontrolled framework of Gubinelli et al. \cite{gip}. There is no doubt that Malliavin differentiabiliy of $u$ could have also been obtained in this setting. However, it is widely agreed that regularity structures are ultimately more powerful (think: multiplicative stochastic heat equation \cite{HP} or the generalized KPZ equation,  forthcoming work by Bruned et al.) so that setting up Malliavin calculus in this framework, for a first time, seems to be an important task in its own right. Moreover, we insist that many of the concepts introduced in this work (incl. extended structure and models, translation operators, $\mathcal{H}$-regularity of solutions ...), and in fact our general d\'emarche, will provide a roadmap for dealing with (singular, subcritical) SPDEs other than gPAM. 

\begin{plan*} 
Section~\ref{sec:Mall} and~\ref{sec:Frame} are somewhat introductory. In the first, we briefly recall the tools of Malliavin Calculus we will need. In the second, we will initially build a common ground in which to define the previous equations, constructing a regularity structure that contains all the processes that appear in the description of $u$, $u^h$ and $v^h$ as well as the ones needed to solve~\eqref{eq:gPAM}, ~\eqref{eq:gPAMh} and~\eqref{eq:DergPAMInt}. Then we will review the main ingredients of the theory of Regularity Structures and the solution theory for the aforementioned equations. In Section~\ref{sec:MallDiff}, we prove that $v^h$ is indeed the Malliavin derivative of $u$ evaluated at $h$ and list its main properties. Finally, in Section~\ref{sec:Dens}, after providing the strong maximum principle \textit{\`a la} Mueller announced before, we show that $u$ admits a density with respect to the Lebesgue measure conditioned to non-exploding. 
\end{plan*}

\begin{ackno*}
PKF acknowledges partial funding from the European Research Council through StG-258237 / CoG-683164 and DFG research unit FOR2402. 
PG, affiliated to TU Berlin while this work was commenced, received partial funding through StG-258237. GC acknowledges a stipend from DFG research training group 1845.
\end{ackno*}

\begin{notation*}

We collect here some definitions and notations we will use throughout the paper. 
Let $\alpha\in(0,1]$. We say that a function $f$ belongs to $\CC^\alpha(\R^d)$ if for any compact subset $D$ of $\R^d$
\[
\|f\|_{\CC^\alpha}:=\sup_{x\in D}|f(x)| +\sup_{\substack{x\neq y\in D\\ |x-y|\leq 1}} \frac{|f(y)-f(x)|}{|x-y|^\alpha}<\infty\,.
\]
In case $\alpha=0$, we identify $\CC^0$ with the space of $L^\infty_{\loc}$ functions. 
Since we will be working on parabolic equations, in which time and space play different roles and, consequently, have different scaling, we introduce the parabolic norm on $\R^{1+d}$, which, abusing notation, is given by
\[
|z-w|:=\sqrt{|t-s|}+|x-y|
\]
for  $z=(t,x)$ and $w=(s,y)\in\R^{1+d}$. Now, for $\alpha\in(0,1)$ and $\eta\geq 0$, we define the weighted space of (locally) H\"older continuous functions $\CC^{\alpha,\eta}(\R^+\times\R^d)$ as the set of those $f$ such that for any compact $D\subset (0,\infty)\times\R^d$
\begin{equation}\label{def:HN}
\|f\|_{\CC^{\alpha,\eta}}:=\sup_{z=(t,x)\in D} 
|f(z)|+\sup_{\substack{z=(t,x),w=(s,y)\in D\\|z-w|\leq 1}}(t\wedge s)^{\frac{\alpha-\eta}{2}\vee0 }\frac{|f(z)-f(w)|}{|z-w|^\alpha}<\infty
\end{equation}
where, again, the generic points $z,w\in\R^+ \times \R^2$ have to be understood as $z=(t,x)$ and $(s,y)$.
\newline

Let $\SC'(\R^d)$ be the space of Schwartz distributions, $\alpha<0$ and $r:=-\lfloor \alpha\rfloor$. Then, we say that  $f\in\SC'(\R^d)$ belongs to $\CC^\alpha(\R^d)$ (locally) if it belongs to the dual of $\CC^r$ and for every compact set $D\subset\R^d$
\[
\|f\|_{\alpha,D}:=\sup_{z\in D}\sup_{\varphi\in\B^r_1}\sup_{\lambda\in(0,1]} \lambda^{-\alpha}|\langle f,\varphi_z^\lambda\rangle|<\infty
\]
where $\B_1^r$ is the space of $\CC^r$ functions compactly supported in the unit ball whose $\CC^r$-norm is less or equal to 1 and $\varphi^\lambda_x$ is the rescaled version of $\varphi\in\B_1^r$ centered at $x\in\R^d$, i.e. $\varphi^\lambda_x(y)=\lambda^{-d}\varphi(\lambda^{-1}(y-x))$. 

The parabolic version of $\CC^\alpha$ on $\R^{1+d}$ is obtained by simply replacing the rescaled function in the following way 
\begin{equation}\label{eq:ParScaling}
\varphi_z^\lambda(w):=\lambda^{-2-d}\varphi(\lambda^{-2}(t-s),\lambda^{-1}(x-y)).
\end{equation}
where $z=(t,x)$ and $w=(s,y)\in\R^{1+d}$. 
\newline

We will denote by $\N$ the set of non-negative integers.

\end{notation*}

\section{Malliavin Calculus in a nutshell}\label{sec:Mall}

In this section, we recall tools and notations of Malliavin calculus that we will need in the rest of the paper (for a thourogh introduction see for example~\cite{Nua}). 

Let $(\Omega,\Hi,\PR)$ be an abstract Wiener space, i.e. $\Omega$ is a separable Banach space, $\PR$ a zero-mean Gaussian probability measure with full support on $\Omega$ and $\Hi$ the associated Cameron-Martin space, and $\F$ the completion of the Borel $\sigma$-algebra on $\Omega$ with respect to $\PR$. We know that each element $y\in\Omega^\star$ defines a Gaussian random variable that can be denoted by $W(y)$. Then, $W$ can be extended to $\Hi$ and is an isonormal Gaussian process on $(\Omega, \F,\PR)$ according to Definition 1.1.1 in~\cite{Nua}.

Now, we say that a random variable $F$ on $\Omega$ is smooth if it can be written as $F=f(W(h_1),\dots,W(h_n))$, where $h_1,\dots,h_n\in\Hi$ and $f$ is an infinitely continuously differentiable function such that $f$ and all its partial derivatives have (at most) polynomial growth. For a smooth random variable, we can define its Malliavin derivative as (Definition 1.2.1 in~\cite{Nua})
\[
DF:=\sum_{i=1}^n\partial_if(W(h_1),\dots,W(h_n))h_i
\]
and, since $D$ is closable from $L^p(\Omega)$ to $L^p(\Omega;\Hi)$ for all $p\geq 1$, we can set  $\DD^{1,p}$ to be the closure of the set of smooth random variables under the norm
\[
\|F\|_{1,p}:=\left(\E[|F|^p]+\E[\|DF\|_\Hi^p]\right)^{\frac{1}{p}}
\]
where $\|\cdot\|_\Hi$ is the norm on $\Hi$. Thanks to the local properties of the Derivative operator, we can now localize the definition of $\DD^{1,p}$. A random variable $F\in\DD^{1,p}_{\loc}$ if there exists a sequence $(\Omega_n,F_n)_{n\geq 1}\subset \F\times\DD^{1,p}$ such that $\Omega_n\uparrow\Omega$, and $F_n=F$ almost surely on $\Omega_n$. 

We aim at proving that the solution to~\eqref{eq:gPAM} admits a density with respect to the Lebesgue measure, a classical tool in this context is provided by a criterion due to Bouleau and Hirsch~\cite{BH} that we now recall (the formulation below is borrowed by~\cite{Nua} Theorems 2.1.2 and 2.1.3)

\begin{theorem}\label{thm:BH}
Let $F$ be a real-valued random variable on $(\Omega, \F,\PR)$. If $F\in\DD^{1,p}_{\loc}$ for $p\geq 1$, then $F$, conditioned by the set $\|DF\|_\Hi>0$, is absolutely continuous with respect to the Lebesgue measure. 
\end{theorem}

The analysis of the above mentioned equation we will perform in the upcoming sections, is~\textit{pathwise}, hence we need a notion of differentiability better adapted to this construction. 
Moreover we lack a global well-posedness result for~\eqref{eq:gPAM} in the sense that we cannot prevent a priori an explosion of the $L^\infty$ norm of its solution, hence the definition of $\Hi$-differentiability given by Definition 4.1.1 in~\cite{Nua} is too strong in our context and we instead resolve to use Definition 3.3.1 (c) in~\cite{UZ}. 
 
\begin{definition}[Local $\Hi$-Differentiability]\label{def:MallDiff}
Let $(\Omega,\Hi,\PR)$ be an abstract Wiener space. We will say that a random variable $F$ is locally $\Hi$-differentiable if there exists an almost surely positive random variable $q$ such that $h\mapsto F(\omega+h)$ is continuously differentiable on $\{h\in\Hi:\|h\|_\Hi<q(\omega)\}$. The set of full measure $Q=\{q>0\}$ will be called the set of differentiability of $F$. Finally we will indicate by $\CC^1_{\Hi-\loc}$ the set of all locally $\Hi$-differentiable functions. 
\end{definition}

\begin{remark}~\label{rmk:MallDiff}
Of course, if for $F\in\CC^1_{\Hi-\loc}$, $q$ can be taken to be equal to $+\infty$ almost surely, then $F\in\CC^1_\Hi$ according to Definition 4.1.1 in~\cite{Nua}.  
\end{remark}

As pointed out before, the definition above represents the local version of the usual notion of $\Hi$-differentiability. We need to verify that it is not \textit{too} weak, in the sense that the Bouleau and Hirsch's criterion (Theorem~\ref{thm:BH}) can still be applied. This is indeed the case as we will see in the following proposition, whose proof is completely analogous to Proposition 4.1.3 in~\cite{Nua} (or to the one of Proposition 3.4.1 in~\cite{UZ}) and therefore we limit ourselves to indicate the differences. 

\begin{proposition}\label{prop:locH-MallDiff}
We have $\CC^1_{\Hi-\loc}\subset\DD^{1,2}_{\loc}$. 
\end{proposition}
\begin{proof}
Let $F\in\CC^1_{\Hi-\loc}$ and $q$ the strictly positive random variable introduced in Definition~\ref{def:MallDiff}. For $n\in\N$, let $A_n\subset\Omega$ be given by
\[
A_n=\left\{\omega\in\Omega: q(\omega)\geq \frac{4}{n},\,\,\sup_{\|h\|_\Hi<\frac{2}{n}}|F(\omega+h)|\leq n,\,\,\sup_{\|h\|_\Hi<\frac{2}{n}}\|DF(\omega+h)\|_\Hi\leq n\right\}
\]
then, since $F$ is locally $\Hi$-differentiable, $\Omega=\bigcup_n A_n$ a.s.. Moreover, thanks to Corollary 8.3.10 in~\cite{Str},  for all $n\in\N$ there exists a $\sigma$-compact set $G_n\subset A_n$ such that $\PR(G_n)=\PR(A_n)$. For $A\in\F$, set 
\[
\rho_A(\omega):=\inf\{\|h\|_{\Hi}:\omega+h\in A\}
\]
and let $\phi$ be a non-negative, $\CC^\infty(\R)$ compactly supported function such that $|\phi(t)|\leq 1$, $|\phi'(t)|\leq 4$ for all $t$ and $\phi(t)=1$ for $|t|\leq \frac{1}{3}$ and $0$ for $|t|>\frac{2}{3}$.  

Then, it is easy to show (exploiting essentially the properties of $\rho$, see pg. 230 in~\cite{Nua} or Lemma 3.4.2 in~\cite{UZ}, and Exercise 1.2.9 in~\cite{Nua}) that $F_n:=\phi(n\rho_{G_n})F$ is the localizing sequence required in the definition of $\DD^{1,2}_\loc$. 
\end{proof}

In the rest of this paper we  will always consider the case of white noise on the two-dimensional torus, i.e. $\mathcal{H}=L^2(\T^2)$, and we will want to apply the above results to $F=u(t,x)$ (actually, to slightly different random variables because of technicalities due to the explosion time).

\section{The framework}\label{sec:Frame}

\subsection{The Regularity Structure for gPAM}\label{section:RS}

Recall that a regularity structure is a triplet $\TS=(A,T,G)$ in which, $A\subset\R$ is a locally finite bounded from below set of homogeneities,  $T=\bigoplus_{\alpha\in A}T_\alpha$, the \textit{model space}, is a graded normed vector space, and $G$, the \textit{structure group}, is a set of linear transformations on $T$ such that for every $\Gamma\in G$ and $\tau\in T_\alpha$, $\Gamma \tau-\tau\in\bigoplus_{\beta<\alpha}T_\beta$. 

The construction of the regularity structure $\TS_g=(A_g,T_g, G_g)$ needed to solve~\eqref{eq:gPAM} was already carried out in several papers (see \cite{Hai, HP, HL2}), hence we limit ourselves to recalling its main ingredients. At first, one defines two sets, $\U$, containing all the symbols appearing in the description of the solution to (gPAM), and $\mathcal{W}$, containing the ones needed to make sense of the right-hand side of the equation, as the smallest collections such that {\color{blue} $1$}, {\color{blue} $X_i^k$}, for $k\in\N$ and $i=1,2$, belong to $\U$ and, whenever ${\color{blue}\tau}\in\U$ then ${\color{blue}\tau\Xi}\in\W$, while  for every ${\color{blue}\tau}\in\W$, ${\color{blue}\I(\tau)}\in\U$. In the previous, {\color{blue} $X^k$} are the symbols corresponding to the polynomials, {\color{blue} $\Xi$} to the noise and $\I$ is the abstract integration map. The model space $T$ is then given by the set of finite linear combinations of elements in $\W\cup\U$ and can be nicely decomposed as the direct sum of $\langle\W\rangle$ and  $\langle \U\rangle$.

To each of the symbols so constructed we can then associate a homogeneity, $|{\color{blue} 1}|=0$, $|{\color{blue} X_j}|=1$, $|{\color{blue}\Xi}|=\alpha_{\min}$, where 
$$\alpha_{\min}=-1-\kappa$$
 for $\kappa>0$ small enough, and then recursively, for every $\tau$, $\bar{\tau}$
\[
|\tau\bar{\tau}|=|\tau|+|\bar{\tau}|,\qquad |\I(\tau)|=|\tau|+2
\]
so that the graded structure on $T$ is simply obtained by defining $T_\beta$ as the set of finite linear combinations of those symbols in $T$ with homogeneity equal to $\beta$. 

It turns out that, to solve the equation~\eqref{eq:gPAM}, we will not need the whole model space as previously constructed but it will be sufficient to consider those elements in $\U$ and $\W$ with homogeneity respectively less than a fixed threshold $\gamma$, slightly greater than $-\alpha_{\min}$, and $\gamma+\alpha_{\min}$. We will indicate the union of this restricted sets of symbols by $\F$. Hence, the model space we will use from now on is 
\[
T_g=\langle \F\rangle=T_\W\oplus T_\U:=\langle{\color{blue}\Xi},{\color{blue}\I(\Xi)\Xi},{\color{blue}X_i\Xi}:\,i=1,2\rangle\oplus\langle {\color{blue}1},{\color{blue}\I(\Xi)},{\color{blue}X_i}:\,i=1,2\rangle
\]
and its associated set of homogeneities is $$A_g=\{\alpha_{\min},2\alpha_{\min}+2,\alpha_{\min}+1,0,\alpha_{\min}+2,1\}.$$ 

In order to describe the structure group we first introduce the vector space $T_g^+$, whose basis vectors are symbols of the form
\[
X^k \prod \J_l(\tau_l),\qquad \tau_l\in T_g
\]
where $k\in\N^2$ and factors $\J_l(\tau_l)$ are omitted whenever  $|\tau_l|+2-|l|\leq 0$. Analogously to what done before, we can assign to each of these symbols a homogeneity, $|1|=0$, $|X_i|=1$ and then recursively, for $\tau\in T_g$ and $\tau_1,\,\tau_2\in T_g^+$, $|\J_l(\tau)|=|\tau|+2-|l|$ and $|\tau_1\tau_2|=|\tau_1|+|\tau_2|$.  We then define the linear map $\Delta:T_g\to T_g\otimes T_g^+$ via 
\[
\Delta {\color{blue}1}={\color{blue}1}\otimes 1,\quad\Delta {\color{blue}X_i}={\color{blue}X_i}\otimes 1+{\color{blue}1}\otimes X_i,\quad \Delta{\color{blue}\Xi}={\color{blue}\Xi}\otimes 1
\]
and then recursively, for $\tau$, $\bar{\tau}\in T_g$
\begin{equation}\label{eq:Delta}
\Delta\tau\bar{\tau}=(\Delta\tau)(\Delta\bar{\tau}),\quad \Delta \I(\tau)=(\I\otimes \text{Id})\Delta\tau+\sum_{l,k}\frac{1}{k!\,l!}{\color{blue}X^k}\otimes X^l\J_{k+l}(\tau)
\end{equation}
where Id is the identity and the sum runs over finitely many $k,l$ since $\J_{k+l}(\tau)=0$ if $|\tau|+2-|l|-|k|\leq 0$. Now, let $\mathcal{G}_+$ be the family of linear functionals $f: T_g^+\to \R$, such that, for any $\tau$, $\bar{\tau}\in T_g^+$, $f(\tau\bar{\tau})=f(\tau)f(\bar{\tau})$, the structure group $G_g$ will be then composed by the maps $\Gamma_f$ given by 
\begin{equation}\label{eq:Gamma}
\Gamma_f\tau=(\text{Id}\otimes f)\Delta \tau
\end{equation}
According to Proposition 8.21 in~\cite{Hai}, $G_g$ is a group and, thanks to Theorem 8.24, it satisfies the structure group relation $\Gamma_f\tau-\tau\in\bigoplus_{\beta<\alpha} T_\beta$ for every $\tau\in T_\alpha$. If we now specialize to the case at study, it is immediate to verify that, for a given $f\in\G_+$,  the action of $\Gamma_f$ on the elements of $T_g$, can be represented as the matrix
\begin{equation}\label{eq:matrix}
\Gamma_f=
 \left(
    \begin{array}{r@{}c|c@{}l}
  &    \begin{smallmatrix}
        1 & f(\J(\Xi)) & f(X_1) & f(X_2) \\
          0 & 1 & 0 &0\\
	  0 & 0 & 1&0\\
        0 & 0 & 0&1\rule[-1ex]{0pt}{2ex}
      \end{smallmatrix} & \mbox{\huge$\mathbb{O}$} \\\hline
  &    \mbox{\huge $\mathbb{O}$} &  
       \begin{smallmatrix}\rule{0pt}{2ex}
        1 &f(\J(\Xi)) &f(X_1) &f(X_2) \\
        0 &1 &0 &0\\
        0 &0 &1 &0\\
	0 &0 &0 &1
      \end{smallmatrix}    
    \end{array} 
\right)
\end{equation}
Thanks to the previous, the group structure of $G_g$ becomes explicit. More specifically, for $f_1$, $f_2\in\G_+$, one can directly show that $\Gamma_{f_1}\Gamma_{f_2}=\Gamma_{f_1+f_2}$, $\Gamma_{f_1}^{-1}=\Gamma_{-f_1}$ and that the identity matrix is the identity of the group.   
Moreover, $T_\W$ and $T_\U$ are invariant under the action of the structure group and therefore, according to Definition 2.5 in~\cite{Hai}, are sectors of $T$ of regularity $\alpha_{\min}$ and $0$, respectively.  

\subsection{Enlarging $\TS_g$}\label{section:RSh}

In order to be able to formulate~\eqref{eq:gPAMh} at the abstract level we need to suitably enlarge the regularity structure just constructed without altering its characterizing features. To do so, we will add a symbol {\color{blue} $H$} that will play the same algebraic role as the symbol for the noise, ${\color{blue}\Xi}$, but whose \textit{realization} through the model will possess further properties that we will specify later on. Let us then define two sets $\U^H$ and $\W^H$ such that $\U\subset\U^H$, $\W\subset\W^H$ and, as before,
\[
{\color{blue}\tau}\in\U^H\Longleftrightarrow{\color{blue}\tau\Xi},\,{\color{blue}\tau H}\in\W^H\qquad{\color{blue}\tau}\in\W^H\Longrightarrow{\color{blue}\I(\tau)}\in\U^H
\]
The vector space $T^H$ will be simply given by the set of finite linear combinations of elements in $\U^H\cup\W^H$ and, also in this case, it can be conveniently decomposed as the direct sum of $\langle\W^H\rangle$ and  $\langle \U^H\rangle$. We can assign a homogeneity to each element following the same rules prescribed in the previous section, imposing though $|{\color{blue}H}|=\alpha_{\min}$ (this is consistent with $\mathcal{H}=L^2(\T^2) \subset C^{-1}$, and no better, and recall $\alpha_{\min}=-1-\kappa$, $\kappa>0$). 

Once again, we define $\F^H$ as the set of symbols in $\W^H$ and $\U^H$ of homogeneity less than $\gamma+\alpha_{\min}$ and $\gamma$, respectively, and set 
\begin{multline*}
T_g^H=\langle \F^H\rangle=T_{\W^H}\oplus T_{\U^H}:=\\
\langle{\color{blue}\Xi},{\color{blue}H},{\color{blue}\I(\Xi)\Xi},{\color{blue}\I(\Xi)H},{\color{blue}\I(H)\Xi},{\color{blue}\I(H)H},{\color{blue}X_i\Xi},{\color{blue}X_iH}:\,i=1,2\rangle\oplus\langle {\color{blue}1},{\color{blue}\I(\Xi)},{\color{blue}\I(H)},{\color{blue}X_i}:\,i=1,2\rangle
\end{multline*}
whose associated set of homogeneities is the same as before, i.e. $A^H_g=\{\alpha_{\min},2\alpha_{\min}+2,\alpha_{\min}+1,0,\alpha_{\min}+2,1\}$. 

Concerning the definition of the structure group, we impose the symbol ${\color{blue} H}$ (and all the ones containing it) to behave, at the algebraic level, analogously to ${\color{blue}\Xi}$. More specifically, let $(T_g^H)^+$ be the vector space, whose basis vectors are  
\[
\Big\{X^k \prod \J_l(\tau_l):  \tau_l\in T_g^H\,\,\text{and}\,\, |\tau_l|+2-|l|>0\Big\}
\]
and assign to each of its elements a homogeneity according to the same rules as for the elements in $T_g^+$. Then, we extend $\Delta$ to a map $\Delta^H:T_g^H\to T_g^H \otimes (T_g^H)^+$ in such a way that 
\[
\Delta^H {\color{blue}H}={\color{blue}H}\otimes 1,\qquad \Delta^H\tau=\Delta\tau\,\,\text{for all $\tau\in T_g$}
\]
and the relations in~\eqref{eq:Delta} still hold. Finally, defining $\mathcal{G}_+^H$ as the set of linear functionals $f: (T_g^H)^+\to \R$ such that $f(\tau\bar{\tau})=f(\tau)f(\bar{\tau})$ for all $\tau$, $\bar{\tau}\in (T_g^H)^+$, the structure group will be given by those $\Gamma_f$, $f\in\mathcal{G}_+^H$, acting on $T_g^H$ as in~\eqref{eq:Gamma} but with $\Delta$ substituted by $\Delta^H$. 

At last, we notice that, when restricted to $T_g^H$, the action of $\Gamma_f$ can be expressed as
\begin{equation}\label{eq:matrixh}
\Gamma_f=
 \left(
    \begin{array}{r@{}c|c@{}l}
  &    \begin{smallmatrix}
        1 & 0 &f(\J(\Xi)) & 0 &f(\J(H)) &0& f(X_1) & 0&f(X_2) & 0 \\
        0 & 1 & 0 & f(\J(\Xi)) & 0 &f(\J(H))&0 &f(X_1)&0& f(X_2)\\
        0 & 0 & 1 & 0 & 0 & 0&0 &0&0 &0\\
	0 & 0 & 0 & 1& 0& 0&0 &0&0 &0\\
	0 & 0 & 0 & 0& 1& 0 &0 &0&0 &0\\
	0 & 0 & 0 & 0& 0& 1 &0 &0&0 &0\\
	0 & 0 & 0 & 0& 0& 0 &1 &0&0 &0\\
	0 & 0 & 0 & 0& 0& 0 &0 &1&0 &0\\
	0 & 0 & 0 & 0& 0& 0 &0 &0&1 &0\\
	0 & 0 & 0 & 0& 0& 0 &0 &0 &0&1\rule[-1ex]{0pt}{2ex}
      \end{smallmatrix} & \mbox{\huge$\mathbb{O}$} \\\hline
  &    \mbox{\huge $\mathbb{O}$} &  
       \begin{smallmatrix}\rule{0pt}{2ex}
        1 &f(\J(\Xi)) &f(\J(H))&f(X_1) &f(X_2) \\
        0 &1 &0 &0 &0\\
        0 &0 &1 &0 &0\\
	0 &0 &0 &1 &0\\
	0 &0 &0 &0 &1
      \end{smallmatrix}    
    \end{array} 
\right)
\end{equation}
Once more, the same conclusions established before hold in this case as well, i.e. the group structure of $G^H_g$ is explicit and the subspaces $T_{\W^H}$ and $T_{\U^H}$ are sectors of regularity $\alpha_{\min}$ and $0$, respectively. 

Finally, we point out that, by construction, $\TS_g\subset\TS_g^H$ in the sense of Section 2.1 of~\cite{Hai}, where $\TS_g^H=(A^H_g,T^H_g,G^H_g)$.

\begin{remark}
The construction just carried out is a straight-forward generalization of the one in Section 8.1 in~\cite{Hai} and corresponds to the case in which we have more than one symbol for the noise (see for example~\cite{ZZ}, where this is done in the context of the Navier-Stokes equation). Nevertheless, we underline once more that the symbol ${\color{blue} H}$ has a completely different meaning and has been introduced at the sole purpose of keeping track of the translation of the noise at the abstract level. In particular, the homogeneity of these symbols is somewhat artificial as we will see in the following section. 
\end{remark}

\subsection{Admissible Models}\label{subsection:Amodel}

The objects constructed in the previous section are, for the moment, pure symbols to which we want to associate a suitable family of distributions. To this aim, Hairer introduces the notion of model for a \textit{general} regularity structure $\TS=(A,T,G)$, i.e. a pair of maps $Z=(\Pi,\Gamma)$, where $\Pi:\R^3\to\mathcal{L}(T,\SC'(\R^3))$ (with $\mathcal{L}(X,Y)$ the set of linear functionals from $X$ to $Y$) and $\Gamma:\R^3\times\R^3\to G$, such that, for every $x,y,z\in\R^3$, $\Pi_x\Gamma_{xy}=\Pi_y$ and $\Gamma_{xy}\Gamma_{yz}=\Gamma_{xz}$. Moreover, for every $\gamma>0$ and compact set $D$ there exists a constant $C_{\gamma,D}$ such that
\begin{equation}\label{eq:ModelAna}
|\langle \Pi_x\tau,\varphi_x^\lambda\rangle|\leq C_{\gamma,D}\lambda^{|\tau|}\qquad\text{and}\qquad\|\Gamma_{xy}\tau\|_{m}\leq C_{\gamma,D}|x-y|^{|\tau|-m}
\end{equation}
uniformly over symbols $\tau\in T$\footnote{Following \cite{Hai}, $\| . \|_\alpha$ denotes the norm on $T_\alpha$, freely generated by symbols of homogeneity $\alpha$ which in turn are assumed to have unit norm.} 
  with homogeneity strictly smaller than $\gamma$,  $\varphi\in\B_1^r$ for $r$ the smallest integer strictly greater that $\min A$, $x,y$ in a compact and $m\in A$ less than $|\tau|$ (here and later, the rescaling of $\varphi$ has to be understood in parabolic sense, see~\eqref{eq:ParScaling}, unless otherwise stated). We endow the family of models with the system of seminorms $\VERT Z\VERT_{\gamma;D}:=\|\Pi\|_{\gamma;D}+\|\Gamma\|_{\gamma;D}$ defined as the smallest constant $C_{\gamma,D}$ such that the bounds in~\eqref{eq:ModelAna} hold. 

Among the set of models, we need to identify a suitable subset embracing the main constituents of our equations. To begin with, following Section 5 in~\cite{Hai}, we rewrite the heat kernel $\bar{K}$ in spatial dimension $d=2$ as the sum of two terms, a ``singular" part $K$ (represented in our regularity structure by the symbol $\I$) and a smooth remainder $R$ (that, being smooth, can simply be ``lifted" through the polynomials), in such a way that 
\begin{enumerate}
\item $K$ is compactly supported in $\{|x|^2+t\leq 1\}$, is symmetric in the spatial variable $x$ and is $0$ for $t\leq 0$,
\item for $(t,x)$ such that $|x|^2+t<\frac{1}{2}$ and $t>0$, one has 
\[
K(t,x)=\frac{1}{|4\pi t|}e^{-\frac{|x|^2}{4t}}
\]
and it is smooth on $\{|x|^2+t\geq \frac{1}{4}\}$,
\item $K$ annihilates every polynomial of parabolic degree less than $3$. 
\end{enumerate} 
\begin{remark}
We are allowed to split the heat kernel thanks to Lemma 5.5 in~\cite{Hai}. Indeed, it precisely satisfies the scaling condition there required. 
\end{remark}

At this point all the elements are in place and we can define the family of so called \textit{admissible models} for either of the regularity structures constructed before. 

\begin{definition}\label{def:Amodel}
A model $(\Pi,\Gamma)=:Z$ on $\TS_g$ (resp. $\TS_g^H$) is said to be \textit{admissible} if, for every $x,y\in\R^3$, the following conditions hold
\begin{enumerate}
\item $\Pi_x {\color{blue}1}(y)=1$ and for every multiindex $k\in\N^2$ and $\tau\in \W$ (resp. $\W^H$) such that ${\color{blue}\tau X^k}\in T_g$ (resp. $T_g^H$)  
\begin{subequations}\label{eq:Amodel}
\begin{equation}\label{eq:AmodelPol}
\Pi_x {\color{blue}\tau X^k}(\cdot)=(\cdot-x)^k\Pi_x\tau 
\end{equation}
and for every $\tau\in \W$ (resp. $\W^H$) such that $\I(\tau)\in T_g$ (resp. $T_g^H$),
\begin{equation}\label{eq:AmodelInt1}
\Pi_x {\color{blue}\I(\tau)}(y)=\langle \Pi_x {\color{blue}\tau},K(y-\cdot)\rangle +\sum_{0\leq|l|<|\tau|+2}\frac{(y-x)^l}{l!}f_x(\J_l(\tau))
\end{equation}
\end{subequations}
where $f_x$ is an element in $\mathcal{G}_+$ (resp. $\mathcal{G}^H_+$), characterized by\footnote{Due to multiplicativity of such elements, this determines $f_x$ on $T_g^+$ (resp. 
$T_{g+}^H$).}
\begin{equation}\label{eq:AmodelPol1}
f_x(X_i)=-x_i
\end{equation}
and, again for every $\tau\in \W$ (resp. $\W^H$) such that $\I(\tau)\in T_g$ (resp. $T_g^H$),
\begin{equation}\label{eq:AmodelInt}
f_x(\J_l(\tau))=-\langle \Pi_x\tau,D^{(l)}K(x-\cdot)\rangle,\qquad |l|<|\tau|+2
\end{equation}
\item The map $\Gamma$ is given via the relation
\begin{equation}\label{eq:AmodelAlg1}
\Gamma_{xy}=(\Gamma_{f_x})^{-1}\Gamma_{f_y}.
\end{equation}
\end{enumerate}
\end{definition}
The existence of such admissible models is non-trivial. Nevertheless, it turns out that it is always possible to lift a smooth function (say a mollification of the noise) $\xi_\eps$ to an admissible model imposing~\eqref{eq:Amodel},~\eqref{eq:AmodelAlg1}, in addition to
\begin{equation}\label{eq:CanModel}
\Pi^\eps_x{\color{blue}\Xi}(y)=\xi_\eps(y)\qquad\text{and}\qquad \Pi^\eps_x\tau\bar{\tau}(y)=\Pi^\eps_x\tau(y)\Pi^\eps_x\bar{\tau}(y)
\end{equation}
as it was shown in~\cite{Hai}, Proposition 8.27. We will refer to such a model as the \textit{canonical model}. Moreover, for $T=T_g$ or $T_g^H$ (and similar for $\TS$), we say that an admissible model is \textit{smooth} if $\Pi:\R^3\to\mathcal{L}(T,\CC^\infty(\R^3))$ is a smooth function (cf. \cite{HW}; note that every canonical model is smooth).
We then define $\M(\TS)$, a {\it separable} space of admissible models, as the closure of the set of smooth admissible models under the family of semidistances  
\begin{equation}\label{def:Distance}
\VERT \Pi-\bar{\Pi}\VERT_D:=\sup_{x\in D}\sup_{\varphi\in\B^r_1}\sup_{\lambda\in(0,1]}\sup_{\tau\in T} \lambda^{-|\tau|}|\langle \Pi_x\tau-\bar{\Pi}_x\tau,\varphi_x^\lambda\rangle|
\end{equation} 
where $D$ is a compact subset of $\R^3$ and $r:=-\lfloor \alpha_{\min}\rfloor$. 

\begin{remark} 
Let $(\Pi,\Gamma)$ be an admissible model. The algebraic properties of the model and~\eqref{eq:AmodelAlg1} in the previous definition convey an extremely important fact, i.e. the quantity $\Pi_x(\Gamma_{f_x})^{-1}$ is independent of the base point $x$, or in other words, for every $x,y$ we have 
\begin{equation}\label{eq:AmodelAlg}
\Pi_x(\Gamma_{f_x})^{-1}=\Pi_y(\Gamma_{f_y})^{-1}
\end{equation}
The reason why we are stating it explicitly is that a partial inverse hold. Indeed, if for a pair of maps $(\Pi, \Gamma)$,~\eqref{eq:AmodelAlg1} and~\eqref{eq:AmodelAlg} hold, then the algebraic relations characterizing a model are immediately satisfied, and this will be crucial in what follows.  
\end{remark}
\begin{remark}\label{remark:1}
It might appear weird that in the definition of the semidistance~\eqref{def:Distance} there is no reference to the map $\Gamma$. As already pointed out in Remark 2.4 of~\cite{HW} and Remark 3.5 of~\cite{HP},  if $(\Pi,\Gamma)$ is a pair of maps for which the first analytical bound in~\eqref{eq:ModelAna} and the equalities~\eqref{eq:Amodel} and~\eqref{eq:AmodelAlg1} hold then also the second analytical bound in~\eqref{eq:ModelAna} is automatically satisfied thanks to Theorem 5.14 in~\cite{Hai}. For the reader's convenience, we give a short and self-contained proof of this in our setting in Appendix~\ref{remark}.
\end{remark}
\begin{remark}
We point out that as a consequence of Definition~\ref{def:Amodel}, the relation~\eqref{eq:Gamma} and the previous remark, for an admissible model $(\Pi,\Gamma)$, the action of the map $\Gamma$ is fully determined by the map
$$   f: x \mapsto f_x $$
introduced in Definition \ref{def:Amodel} above, hence we will denote an admissible model either by $(\Pi,\Gamma)$ or by $(\Pi,f)$ without further notice. 
\end{remark}
\begin{remark}\label{remark:periodicM}
As a final remark, we will consider only models \textit{adapted} to the action of a subgroup of the translations, according to Definition 3.33 in~\cite{Hai}. More specifically, this means that, if $e_1$ and $e_2$ are the canonical basis vectors on $\R^2$ and $T_i:\R^3\to\R^3$ is given by $T_i(t,x)=(t,x+2\pi e_i)$, we require that for all $z\in\R^3$ and $\varphi\in\SC$, $\langle\Pi_{T_iz}\tau,\varphi(T_i^{-1}\cdot)\rangle=\langle\Pi_{x}\tau,\varphi\rangle$ and $\Gamma_{f_{T_iz}}=\Gamma_{f_x}$. In this way, for $I\subset\R$ an interval, the domain $D$ appearing in~\eqref{def:Distance} can be simply taken to be $I\times\T^2$, and we will simply omit it.
\end{remark}

Because of the stringent conditions imposed in the previous definition, we would like to have a way to check if a model is indeed admissible given the minimal possible amount of information. To this purpose, following what done in Section 2.4 of~\cite{HW}, we introduce the notion of \textit{minimal admissible model}. 
\begin{definition}\label{def:minModel}
Let $\TS=(A,T,G)$ be either $\TS_g$ or $\TS_g^H$, and $T_-$ the subspace of $T$ generated by the symbols with negative homogeneity. A pair of maps $\Pi: \R^3\to\mathcal{L}(T_-,\SC'(\R^d))$, $\Gamma:\R^3\times\R^3\to G$ is said to form a \textit{minimal admissible model} for $\TS$ if for all $\tau\in T_-$, $\Pi$ satisfies the first bound in~\eqref{eq:ModelAna}, as well as the relations~\eqref{eq:AmodelPol},~\eqref{eq:AmodelInt},~\eqref{eq:AmodelAlg1} and~\eqref{eq:AmodelAlg}. We indicate by $\M_{\_}(\TS)$ the closure of the family of all such smooth pairs under the semidistance given in~\eqref{def:Distance}, but where the last supremum is taken only over the elements $\tau\in T_{-}$. 
\end{definition}

The previous definition is, as a matter of fact, meaningful, since it gives just enough information to define the action of $\Gamma_{xy}$ on all the terms of negative homogeneity of either $T_g$ or $T_g^H$. Indeed, it is sufficient to verify that, for any $x$, this is true for $\Gamma_{f_x}$. We have an explicit expression for the latter,~\eqref{eq:matrix} and~\eqref{eq:matrixh} respectively, out of which we deduce that we only need to check if the expressions $f_x(\J(\Xi))$ and $f_x(\J(H))$ can be obtained and this is guaranteed by~\eqref{eq:AmodelInt} and the fact that $\Pi_x{\color{blue}\Xi}$ and $\Pi_x{\color{blue}H}$ are, by assumption, given.

The importance of the the space $\M_{\_}(\TS)$ is clarified by the following theorem (see Theorem 2.10 in~\cite{HW} for the analogous statement in the context of the stochastic quantization equation).  

\begin{theorem}\label{thm:minModel}
Let $\TS=(A,T,G)$ be either $\TS_g$ or $\TS_g^H$. For every $(\Pi,\Gamma)\in\M_{\_}(\TS)$ there exists a unique admissible model $(\tilde{\Pi},\tilde{\Gamma})\in\M(\TS)$ such that, for every element $\tau\in\TS$ with negative homogeneity and $x\in\R^3$, $\Pi_x\tau=\tilde{\Pi}_x\tau$. Moreover, the map that assigns $\M_{\_}(\TS)\ni(\Pi,\Gamma)\mapsto(\tilde{\Pi},\tilde{\Gamma})\in\M(\TS)$ is locally Lipschitz continuous. 
\end{theorem}
\begin{proof}
As mentioned in~\cite{HW}, the proof is a straightforward concequence of Proposition 3.31 and Theorem 5.14 in~\cite{Hai}. Nevertheless we point out that, since we require the extended model  $(\tilde{\Pi},\tilde{\Gamma})$ to be admissible, we have no choice. Indeed, it is already specified by $(\Pi,\Gamma)$ on the elements of negative homogeneity, and on the others (${\color{blue} \I(\tau)}$, $\tau\in\{{\color{blue} \Xi}, {\color{blue} H}\}$, and ${\color{blue} X_i}$, $i=1,2$), relations~\eqref{eq:AmodelPol} and~\eqref{eq:AmodelInt1} leave no alternatives. At this point one would have to show that the algebraic relations are indeed satisfied and that the analytical bounds hold. While the latter follow from Lemmas 5.19 and 5.21 in~\cite{Hai}, the first is an easy computation.  
\end{proof}

\subsection{Extension and Translation of Admissible Models}

So far we have completely ignored the specific role the symbol {\color{blue}$H$} is supposed to play. Indeed, such symbol should represent the abstract counterpart of an element in the Cameron-Martin space and, therefore, we would at least need to impose that $\Pi_x {\color{blue}H}$ corresponds to an $L^2$ function. To incorporate this condition, instead of modifying Definition~\ref{def:Amodel}, we will show that, given an admissible model for $\TS_g$, one can always uniquely extend it to a suitable admissible model for the whole of $\TS_g^H$. 

\begin{proposition} \label{prop:ContShiftModel}
Let $Z=(\Pi,f)$ be an admissible model for $\TS_g$. Given $h\in L^2(\T^2)$, there exists a unique admissible model $Z^{e_h}=(\Pi^{e_h},\Gamma^{e_h})$ on $\TS_g^H$ such that
\begin{enumerate}
\item for all $\tau\in T_g$, $\bar{\tau}\in T_+$ and $x\in\R^3$, $\Pi_x^{e_h}\tau=\Pi_x\tau$ and $f_x^{e_h}\bar{\tau}=f_x\bar{\tau}$, 
\item $\Pi^{e_h}_x{\color{blue} H}=h$ for all $x$ and for every $\tau\in\F^H\setminus\F$, $|\tau|<0$, that can be written as $\tau_1\tau_2$ for $\tau_1\in\U^H$ and $\tau_2\in\{{\color{blue}\Xi},{\color{blue} H}\}$, 
\begin{equation}\label{eq:defTildePi}
\Pi^{e_h}_x \tau = \big(\Pi^{e_h}_x \tau_1\big)\,\big(\Pi^{e_h}_x\tau_2\big)
\end{equation}
\item for all $x\in\R^3$, $Z^{e_h}$ satisfies~\eqref{eq:Amodel} on $\TS_g^H$.
\end{enumerate}
Moreover, the map $E$ that assigns to $(h,Z)\in L^2(\T^2)\times\M(\TS_g)$, $E_hZ:=Z^{e_h}\in\M(\TS_g^H)$ is jointly locally Lipschitz continuous. For a given model $Z$ and $L^2$-function $h$ we call $E_h Z=Z^{e_h}$, the extension of $Z$ in the $h$-direction. 
\end{proposition}

\begin{proof} Let $(\Pi,f)$ be an admissible model for $\TS_g$. We now construct $(\Pi^{e_h},f^{e_h})$ as follows. At first, we set $\Pi_x^{e_h}\tau=\Pi_x\tau$ and $f_x^{e_h}\bar{\tau}=f_x\bar{\tau}$, for all $\tau\in T_g$, $\bar{\tau}\in T_g^+$ and $x\in\R^3$ (condition 1 in the statement). We then extend it \textit{recursively} on the rest of $\F^H$ setting $\Pi_x^{e_h}{\color{blue}H}(y):=h(y)$, defining $f_x^{e_h}$ by~\eqref{eq:AmodelPol1},~\eqref{eq:AmodelInt} and the requirement of being multiplicative (notice that for the elements in $T_+$ this is already the case since, on those, $f^{e_h}\equiv f$), and finally imposing~\eqref{eq:defTildePi} and~\eqref{eq:Amodel}, noting that the products in \eqref{eq:defTildePi} are well-defined due to \eqref{equ:BesovProduct} (analytic bounds necessary for the sequel are established in Appendix A). To be fully explicit, from (\ref{eq:AmodelInt}) knowledge of $\Pi_x^{e_h}{\color{blue}H}$ implies that $f_x^{e_h} \J_l (H)$ is determined. This in turn gives us $\Pi_x^{e_h}{\color{blue} \I (H)}$, thanks to (\ref{eq:AmodelInt1}). The realization $\Pi_x^{e_h}$ on all other symbols in $\F^H\setminus\F$, $|\tau|<0$ is then obtained from (\ref{eq:defTildePi}). \footnote{Strictly speaking, given that ${\color{blue} \I (H)}$ is the only symbol in $\F^H\setminus\F$ of positive homogeneity, the realization map is already defined in all of $\F^H$. The (here unnecessary) use of the minimal admissible model is justified with regard to adaptions of our argument to more complicated regularity structures}

At this point all we need to show is that $(\Pi^{e_h},f^{e_h})$ is a minimal admissible model according to Definition~\ref{def:minModel}, so that Theorem~\ref{thm:minModel} will directly lead to the conclusion. By construction, the image through $\Pi_x^{e_h}$ of the elements of negative homogeneity is fully determined, hence the first bound in~\eqref{eq:ModelAna} follows by Lemma \ref{lem:newlemApp} in Appendix~\ref{app:comp}. 

We then define, for every $x,y$, $\Gamma_{xy}^{e_h}=(\Gamma_{f_x^{e_h}})^{-1} \Gamma_{f_y^{e_h}}$, so that it only remains to verify the validity of~\eqref{eq:AmodelAlg}. It is definitely true for ${\color{blue}H}$ since $\Gamma_{f_x^{e_h}}{\color{blue}H}={\color{blue}H}$ and $\Pi_x^{e_h}{\color{blue}H}$ is independent of $x$. All the other terms in $\tau\in\F^H\setminus\F$, $|\tau|<0$ can be rewritten as $\tau_1\tau_2$ for $\tau_1\in\U^H$ and $\tau_2\in\{{\color{blue}\Xi},{\color{blue} H}\}$. Since, by construction $\Gamma_{f_x^h}$ is multiplicative, we have
\[
\Pi_x^{e_h}(\Gamma_{f_x^{e_h}})^{-1}\tau= \Big(\Pi_x^{e_h}(\Gamma_{f_x^{e_h}})^{-1}\tau_1\Big)\Big(\Pi_x^{e_h}(\Gamma_{f_x^{e_h}})^{-1}\tau_2\Big)
\]
and we already pointed out that the second factor is independent of the base point $x$, for $\tau_2\in  \{{\color{blue}\Xi},{\color{blue} H}\}$. For the first, it suffices to consider $\tau_1={\color{blue}\I(H)}$ since for the other elements it follows by the fact that $(\Pi^{e_h},f^{e_h})$ coincides with $(\Pi,f)$ on $\TS_g$ and the latter is admissible. Now, the matrix in~\eqref{eq:matrixh} conveys that
\[
(\Gamma_{f_x^{e_h}})^{-1}{\color{blue}\I(H)}={\color{blue}\I(H)}-f_x^{e_h}(\J(H)) {\color{blue}1}
\]
hence, applying $\Pi_x^{e_h}$ to both sides and recalling that $\Pi_x^{e_h}$ satisfies~\eqref{eq:AmodelInt1}, we get
\[
\Pi_x^{e_h}(\Gamma_{f_x^{e_h}})^{-1}{\color{blue}\I(H)}(y)=\langle \Pi_x^{e_h} {\color{blue}H},K(y-\cdot)\rangle +f_x^{e_h}(\J(H))-f_x^{e_h}(\J(H))=\langle \Pi_x^{e_h} {\color{blue}H},K(y-\cdot)\rangle
\]
and the last member of the previous chain of equalities does not depend on $x$ since $\Pi_x^{e_h}{\color{blue}H}$ does not. 
\newline

Concerning the local Lipschitz continuity, let $M>0$. Then, the same bounds obtained in Lemma~\ref{lem:newlemApp} immediately imply
\[
\VERT{\Pi}^{e_h} - \tilde{\Pi}^{e_{\tilde h}} \VERT \lesssim \|h-\tilde h\|_{L^2}+\VERT \Pi-\tilde{\Pi}\VERT
\] 
uniformly over $h$, $\tilde h\in L^2(\T^2)$ and $Z=(\Pi,f)$, $\tilde{Z}=(\tilde \Pi,\tilde f)\in\M(\TS_g)$ such that $\|h\|_{L^2}$, $\|\tilde h\|_{L^2}$, $\VERT \Pi\VERT$, $\VERT \tilde \Pi\VERT\leq M$, where $E_hZ=(\Pi^{e_h},f^{e_h})$ and $E_{\tilde h}\tilde Z=(\tilde{\Pi}^{e_{\tilde h}},\tilde f^{e_{\tilde h}})$ are defined as above, and the implicit constant in the previous inequality depends only on $M$. 
\end{proof}

The previous proposition gives a canonical way to \textit{extend} an admissible model for $\TS_g$ to an admissible model for a bigger regularity structure, $\TS_g^H$, and uniquely specifies the action of such an extended model on the new symbols. However, it is important to remember that, in a way, we aim at \textit{translating} the model in the Cameron-Martin directions. To do so, we propose an abstract procedure that allows to encode such an operation on the space of admissible models. 

Let $\TS_g=(A_g,T_g,G_g)$ and $\TS_g^H=(A_g^H,T_g^H,G_g^H)$ be the regularity structures constructed in sections~\ref{section:RS} and~\ref{section:RSh}. We introduce two linear maps $\tau_H:T_g\to T_g^H$, the \textit{abstract translation map}, and $\tau_H^+: T_g^+\to (T_g^H)^+$, where $T_g^+$ and $(T_g^H)^+$ are nothing but the sets of ``coefficients" introduced in the above mentioned sections, and define them recursively by
\[
\tau_H({\color{blue}\Xi})={\color{blue}\Xi}+{\color{blue}H},\qquad \tau_H({\color{blue}X^k})={\color{blue}X^k}\,\,\forall k\,\,\text{multi-index}
\]
and further imposing $\tau_H$ to be multiplicative and to commute with the abstract integration map $\I$. Concerning $\tau_H^+$, we again require it to leave the polynomials invariant, to be multiplicative and to satisfy the following relation
\begin{equation}\label{eq:tauI}
\tau_H^+(\J_l(\tau))=\J_l(\tau_H(\tau))
\end{equation}
for all $\tau\in T_g$ such that $|\tau|+2-|l|>0$. 
\begin{remark}\label{rem:Invariance}
Since the homogeneity of ${\color{blue}\Xi}$ and the one of ${\color{blue}H}$ are the same by construction, a straightforward induction argument shows that if $\tau\in (T_g)_\alpha$, $\alpha\in A_g$, then $\tau_H(\tau)\in(T_g^H)_\alpha$. Indeed, it is trivially true for ${\color{blue}\Xi}$ and the polynomials. Given the fact that $\tau_H$ is multiplicative, if it holds for $\tau$ and $\bar{\tau}$ then it holds for $\tau\bar{\tau}$, and since it commutes with the abstract integration map, if it is satisfied by $\tau$, it is also satisfied by $\I(\tau)$.  
\end{remark}
Thanks to the two maps $\tau_H$ and $\tau_H^+$, we are ready to clarify what it means to \textit{translate} an admissible model. Let $Z=(\Pi, f)\in\M(\TS_g)$ and $h\in L^2$, we then set $T_h Z=(\Pi^h,f^h):=(\Pi^{e_h}\tau_H,f^{e_h}\tau_H^+)$, where $(\Pi^{e_h},f^{e_h})$ is the extended model defined in Proposition~\ref{prop:ContShiftModel}. The purpose of the next proposition is then to show that the \textit{translated} model $T_hZ$ is again an admissible model and prove some continuity properties of the map $T$. 

\begin{proposition}\label{prop:TransModel}
Let $\TS_g$ and $\TS_g^H$ be the regularity structures constructed in Sections~\ref{section:RS} and~\ref{section:RSh} respectively, and $\tau_H$ and $\tau_H^+$ be the maps defined above. Given $Z\in\M(\TS_g)$ and $h\in L^2(\T^2)$, $T_hZ=(\Pi^h,f^h):=(\Pi^{e_h}\tau_H,f^{e_h}\tau_H^+)$ is still an admissible model on $\TS_g$ and the map $T$ that assigns to $(h,Z)\in L^2(\T^2)\times\M(\TS_g)$, $T_hZ=Z^h\in\M(\TS_g)$ is jointly locally Lipschitz continuous. Finally, for a given model $Z$ and $L^2$-function $h$ we call $Z^h$, the translation of $Z$ in the $h$-direction.
\end{proposition}

\noindent In the proof of the proposition we will need the following lemma

\begin{lemma}\label{lemma:RelCon}
In the same context as in Proposition~\ref{prop:TransModel}, for every $\tau\in T_g$ the following relation holds
\begin{equation}\label{eq:RelCon}
(\tau_H\otimes\tau_H^+)\Delta\tau=\Delta^H\tau_H(\tau)
\end{equation} 
\end{lemma}
\begin{proof}
The proof of this lemma proceeds by induction. It is definitely true for ${\color{blue}\Xi}$ and the polynomials. Assume it holds for $\tau$ and $\bar{\tau}$, then it also holds for $\tau\bar{\tau}$ since all the maps involved are multiplicative. Concerning $\I(\tau)$, by~\eqref{eq:Delta}, we have
\begin{align*}
(\tau_H\otimes\tau_H^+)\Delta\I(\tau)&=(\tau_H\otimes\tau_H^+)(\I\otimes \text{Id})\Delta\tau+\sum_{l,k}\frac{1}{k!\,l!}(\tau_H\otimes\tau_H^+){\color{blue}X^k}\otimes X^l\J_{k+l}(\tau)\\
&=(\I\otimes \text{Id})(\tau_H\otimes\tau_H^+)\Delta\tau +\sum_{l,k}\frac{1}{k!\,l!}{\color{blue}X^k}\otimes X^l\J_{k+l}(\tau_H(\tau))\\
&=(\I\otimes \text{Id})\Delta^H\tau_H(\tau) +\sum_{l,k}\frac{1}{k!\,l!}{\color{blue}X^k}\otimes X^l\J_{k+l}(\tau_H(\tau))=\Delta^H\I(\tau_H(\tau))=\Delta^H\tau_H(\I(\tau))
\end{align*}
where the second equality is due to the facts that by construction, $\tau_H$ and $\I$ commute, both $\tau_H$ and $\tau_H^+$ leave the polynomials invariant and are multiplicative, and relation~\eqref{eq:tauI}, while the third follows by the induction hypothesis. 
\end{proof}
\begin{proof}[Proof of Proposition~\ref{prop:TransModel}]
Let us begin by showing that, given a model $Z\in\M(\TS_g)$ and a function $h\in L^2$, $Z^h$ is still an admissible model on $\TS_g$. To do so, we will first prove that the algebraic relations~\eqref{eq:Amodel},~\eqref{eq:AmodelAlg1} and~\eqref{eq:AmodelAlg} are matched, and then verify that the analytical bounds are satisfied. Recall that, thanks to Proposition~\ref{prop:ContShiftModel}, $E_hZ=Z^{e_h}=(\Pi^{e_h},f^{e_h})$ belongs to $\M(\TS_g^H)$, hence~\eqref{eq:Amodel},~\eqref{eq:AmodelAlg1} hold for it. Since moreover the maps $\tau_H$ and $\tau_H^+$ are multiplicative, leave the polynomials invariant and satisfy the relation~\eqref{eq:tauI}, a straightforward computation shows that~\eqref{eq:Amodel} also holds for $Z^h$. While the definition of the maps $\Gamma^h_{xy}$ are implied by the definition of $f^h$ so that~\eqref{eq:AmodelAlg1} is trivially satisfied, the proof of~\eqref{eq:AmodelAlg} is more subtle. By definition we have
\begin{align*}
\Pi_x^h (\Gamma^h_{f_x^{-1}})&=\Pi_x^{e_h}\tau_H(\Gamma_{f_x^{e_h}\tau_H^+})^{-1}=\Pi_x^{e_h}\tau_H \left((\text{Id}\otimes (f_x^{e_h})^{-1}\tau_H^+)\Delta\right)\\
&=\Pi_x^{e_h}\left(\tau_H\otimes(f_x^{e_h})^{-1}\tau_H^+\right)\Delta=\Pi_x^{e_h}(\text{Id}\otimes(f_x^{e_h})^{-1})(\tau_H\otimes\tau_H^+)\Delta
\end{align*}
Now, by Lemma~\ref{lemma:RelCon}, it follows that the right-hand side of the previous equals
\[
\Pi_x^{e_h}(\text{Id}\otimes(f_x^{e_h})^{-1})\Delta^H\tau_H=\Pi_x^{e_h}(\Gamma_{f_x^{e_h}})^{-1}\tau_H= \Pi_y^{e_h}(\Gamma_{f_y^{e_h}})^{-1}\tau_H=\Pi_y^h (\Gamma^h_{f_y^{-1}})
\]
where the second equality is due to the fact that $Z^{e_h}$ is an admissible model for $\TS_g^H$. 

At this point we can focus on the analytical bounds. By Remark~\ref{remark:1} we only have to verify that~\eqref{eq:ModelAna} is satisfied for $\Pi^h=\Pi^{e_h}\tau_H$ but this is immediate since, by Remark~\ref{rem:Invariance}, $\tau_H$ leaves the homogeneities invariant and we already know~\eqref{eq:ModelAna} holds for $\Pi^{e_h}$, since it is a model. The same argument, joint with the results in Proposition~\ref{prop:ContShiftModel}, guarantees the local Lipschitz continuity of the map $T$ in its arguments, so that the proof is concluded.
\end{proof}

\begin{remark}\label{rem:TransCanMod}
Let $\eps>0$, $\xi_\eps=\xi\ast\varrho_\eps$ and $h_\eps=h\ast\varrho_\eps$, where $\xi$ is a distribution, $h$ and $L^2$-function and $\varrho_\eps$ a rescaled mollifier. As we saw in Section~\ref{subsection:Amodel}, we can lift $\xi_\eps$ to the canonical model $Z_\eps$ on $T_g$ by imposing~\eqref{eq:Amodel},~\eqref{eq:AmodelAlg1} and~\eqref{eq:CanModel}. Following the same procedure, but setting also $\tilde{\Pi}_x{\color{blue}H}(y)=h_\eps(y)$, we can construct the canonical model $\bar{Z}_\eps$ on $T_g^H$. Then it is straightforward to prove that $\bar{Z}_\eps=E_{h_\eps}Z_\eps$. 

One might wonder why instead of the construction carried out above, in order to define the translated model on $T_g$, we did not simply follow once more the same procedure (the one to construct the canonical model), requiring though in~\eqref{eq:CanModel},
\[
\tilde{\Pi}_x^{\eps}{\color{blue}\Xi}(y)=\xi_\eps(y) +h_\eps(y)
\]
and obtaining $\tilde{Z}_\eps=(\tilde{\Pi}^\eps,\tilde{\Gamma}^\eps)$. (It is immediate to show that $\tilde{Z}_\eps=T_{h_\eps}Z_\eps$.) Admissibility of $\tilde{Z}$ would then follow from Proposition 8.27 in \cite{Hai}. The problem with this approach is that it gives no estimates in terms of $h \in L^2$, which will be crucial in the sequel. Furthermore, any direct probabilistic construction of (renormalized) model associated to $\xi_\eps +h_\eps$ would lead to $h$-dependent null-sets, opposing any chance to establish $\Hi$-regularity of the Wiener functionals at hand (solutions to gPAM in our case). Finally, another advantage of the construction above is that it applies in a systematic way to any admissible model, without any need for (a converging sequence of) smooth models.
\end{remark}
\begin{remark}
In the context of rough paths, the extension operator defined in this section is reminiscent of the Young pairing or $(p,q)$-Lyons lift introduced in Definition 9.25 of~\cite{FrizVictoir}. 
\end{remark}

\subsection{Extending the Renormalization Group}

The parabolic Anderson equation~\eqref{eq:gPAM} is ill-posed, since the product between the expected solution and the noise cannot be classically defined. One of the main advantages of the theory of Regularity Structures is that such an issue can be overcome thanks to a suitable renormalization procedure. In general, one would like to define a family of maps $M$ (the so called Renormalization Group, $\Re$) acting on $\M$ such that for every $Z\in\M$, $M Z\in\M$ and there exists a sequence $M_\eps\subset\Re$ with the property $\lim_{\eps\to0}M_\eps Z_\eps$ exists, where $Z_\eps$ is the canonical model defined above. We will give only a sketch of the procedure, addressing the reader interested in the general construction of the renormalization group associated to a given regularity structure to Section 8.3 in~\cite{Hai}. 

In the specific context of~\eqref{eq:gPAM}, it turns out that we only need to deal with a one-dimensional subgroup of $\Re$, $\Re_0$ isomorphic to $\R$, that can be explicitly described as follows. Let $M$ be a map acting on the subspace $T_0$ of $T_g$, given by $T_\W\oplus\langle {\color{blue}1}\rangle$, as $M({\color{blue}\I(\Xi)\Xi})={\color{blue}\I(\Xi)\Xi}-C\,{\color{blue}1}$, where $C\in\R$, and $M(\tau)=\tau$ for all the others $\tau\in T_0$. Notice that $M = M(C)$ can be represented by the matrix
\begin{equation}\label{def:M}
M=
\begin{pmatrix}
1 & 0 & 0 & 0 &0\\
0 & 1 & 0 & 0&0\\
0 & 0 & 1 & 0&0\\
0 & -C & 0 & 1&0\\
0 &0&0&0&1
\end{pmatrix}
\end{equation}
where $C$ is a real number. 
\begin{remark}
It is immediate to see that the set $\{M(C): C\in\R\}$ forms a one-dimensional group with respect to the usual matrix product. 
\end{remark}
At this point we want to use these $M$'s to characterize the elements of $\Re_0$. More specifically, for $Z=(\Pi,\Gamma)\in\M(\TS_g)$, with a slight abuse of notation, we define the action of $M$ on $Z$ as $M(Z):=Z^M$, where $Z^M=(\Pi^M,\Gamma^M)$ is defined on $T_0$ by
\begin{equation}\label{eq:RenModel}
\Pi_x^M\tau=\Pi_x M\tau,\qquad \Gamma_{xy}^M\tau=\Gamma_{xy}\tau
\end{equation}
For $M$ to be an element of the renormalization group one has to verify that $Z^M$ can be extended to an admissible model and that the family of $M$'s forms indeed a group under composition. Even if such a result is a consequence of the abstract construction carried out in the Sections 8.3 and 9.1 of~\cite{Hai}, we can exploit our explicit definitions to give a more direct, but more specific, proof. 
\begin{proposition}\label{prop:RenGroup}
Let $\TS_g$ be the regularity structure defined above and $\M(\TS_g)$ the family of admissible models associated to it. For any $Z\in\M(\TS_g)$ and $M$ defined as in~\eqref{def:M}, $M(Z)=Z^M$ given by~\eqref{eq:RenModel} can be uniquely extended to an element of $\M(\TS_g)$ and the family of maps $\Re_0:=\{M: C\in\R\}$ forms a group under composition. 
\end{proposition}
\begin{proof}
While, due to~\eqref{def:M} and~\eqref{eq:RenModel}, it is immediate to show that $\Re_0$ forms a group under composition, in order to verify that $Z^M$ can be uniquely extended to an element of $\M(\TS_g)$, thanks to Theorem~\ref{thm:minModel}, it suffices to prove  that $Z^M$ is a minimal admissible model according to Definition~\ref{def:minModel}. By~\eqref{def:M} and since, by assumption, $Z$ is an admissible model, the analytical bounds straightforwardly hold for every $\tau\in\W$ different from ${\color{blue}\I(\Xi)\Xi}$. For the latter, notice that the action of $M$ consists of adding a counterterm of strictly greater homogeneity, and, by linearity and~\eqref{eq:ModelAna}, we have
\[
|\langle \Pi_x^M{\color{blue}\I(\Xi)\Xi},\varphi_x^\lambda\rangle|=|\langle \Pi_x M{\color{blue}\I(\Xi)\Xi},\varphi_x^\lambda\rangle|\leq|\langle \Pi_x{\color{blue}\I(\Xi)\Xi},\varphi_x^\lambda\rangle|+|C\langle\Pi_x{\color{blue}1},\varphi_x^\lambda\rangle|\lesssim \lambda^{2\alpha_{\min}+2}+|C| \lesssim \lambda^{2\alpha_{\min}+2}
\] 
where the latter holds since $2\alpha_{\min}+2<0$. Concerning~\eqref{eq:AmodelPol} ,~\eqref{eq:AmodelInt},~\eqref{eq:AmodelAlg1} and~\eqref{eq:AmodelAlg}, since $\Gamma_{xy}^M=\Gamma_{xy}$, the only one whose validity is not obvious is the latter. In other words, we have to show that $\Pi_x^M(\Gamma_{f_x}^M)^{-1}=\Pi_y^M(\Gamma_{f_y}^M)^{-1}$. Due to~\eqref{eq:RenModel}, if $M\Gamma_{f_x}=\Gamma_{f_x}M$ for all $x$ we are done, indeed
\[
\Pi_x^M(\Gamma_{f_x}^M)^{-1}\tau=\Pi_x M(\Gamma_{f_x})^{-1}\tau=\Pi_x(\Gamma_{f_x})^{-1}M\tau=\Pi_y(\Gamma_{f_y})^{-1}M\tau=\Pi_y M(\Gamma_{f_y})^{-1}\tau=\Pi_y^M(\Gamma_{f_y}^M)^{-1}\tau
\]
But, the explicit expressions for $M$ and $\Gamma_{f_x}$, guarantee that the necessary commutation equality can be proved through a direct computation consisting in multiplying the corresponding matrices. 
\end{proof}

At this point we can turn our attention to~\eqref{eq:gPAMh} and see what changes have to be performed in order to to be able to renormalize this equation. To this aim, we would like to suitably extend the maps $M$ to the admissible models on $\TS_g^H$. Notice that, thanks to Proposition~\ref{prop:ContShiftModel}, all the terms belonging to $T_g^H\setminus T_g$ are well-defined independently of the specific realization of the noise. Therefore, it is natural to impose that the renormalization procedure leaves those terms invariant. 

We define $T_0^H:=T_{\W^H}\oplus\langle {\color{blue}1}\rangle$ and, given $M$ as in~\eqref{def:M}, we set $M^H:T_0^H\to T_0^H$ as $\left.M^H\right|_{T_0}=M$ and the identity on the orthogonal complement of $T_0$ in $T_0^H$. As before, $M^H$ admits an obvious matrix representation, and we can prescribe the action of $M^H$ on $Z=(\Pi,\Gamma)\in\M(\TS_g^H)$ as $M^H(Z)=Z^{M^H}$ where
\begin{equation}\label{eq:RenModelh}
\Pi_x^{M^H}\tau=\Pi_x M^H\tau,\qquad \Gamma_{xy}^{M^H}\tau=\Gamma_{xy}\tau
\end{equation}
At this point, not only a result analogous to Proposition~\ref{prop:RenGroup} holds, but more is true. 
\begin{proposition}\label{prop:RenGrouph}
Let $h\in L^2(\T^2)$, $\TS_g$ and $\TS_g^H$ be the regularity structures defined above. For any $Z\in\M(\TS_g^H)$ and $M^H$ defined as stated above, $M^H(Z)=Z^{M^H}$ given by~\eqref{eq:RenModelh} can be uniquely extended to an element of $\M(\TS_g^H)$ and the family of maps $\Re_0^H:=\{M^H: C\in\R\}$ forms a group under composition. Moreover, given $M\in\Re_0$ and $Z=(\Pi,\Gamma)\in\M(\TS_g)$, $M^H(E_hZ)$ is an admissible model on $\TS_g^H$ and the following equality holds
\begin{equation}\label{eq:RenExt}
M^H(E_hZ)=E_hM(Z).
\end{equation}
In other words, the operations of extension and renormalization commute. 
\end{proposition}
\begin{proof}
The proof of the first part of the statement proceeds along the same lines as the proof of Proposition~\ref{prop:RenGroup}, hence we will focus on the equality~\eqref{eq:RenExt}. 

By the aforementioned proposition, we know that, given $Z=(\Pi,\Gamma)\in\M(\TS_g)$, $M(Z)=Z^M$ is still an admissible model. Thanks to Proposition~\ref{prop:ContShiftModel}, there exists a unique extension of $Z$ to $T_g^H$, $E_hZ=Z^{e_h}$, and a unique extension of $Z^M$ to $T_g^H$, given by $E_h Z^M$, such that both $E_hZ$ and $E_hZ^M$ satisfy properties 1,2 and 3 there stated. Since by the first part of the statement we are proving, $M^H(E_hZ)$ is again admissible, we only need to show that $M^H(E_hZ)$ enjoys the same properties. The third, that is, validity of (\ref{eq:Amodel}), is obvious, since it is an admissible model on $T_g^H$. For the first, let $\tau\in T_g$. Notice that
\[
(\Pi_x^{e_h})^{M^H}\tau= \Pi_x^{e_h} M^H\tau= \Pi_x^{e_h} M\tau=\Pi_x M\tau=\Pi^M_x\tau=  (\Pi^M_x)^h\tau
\]
where the second equality follows by the fact that $M^H\tau=M\tau$ for every $\tau\in T_g$ and the last by the fact that $(\Pi^M_x)^h$ satisfies condition 1 in Proposition~\ref{prop:ContShiftModel}.

Finally, take $\F^H\setminus\F\ni\tau=\tau_1\tau_2$, $\tau_1\in\U^H$ and $\tau_2\in\{{\color{blue}\Xi},{\color{blue}H}\}$, then
\[
(\Pi_x^{e_h})^{M^H}\tau= \Pi_x^{e_h} M^H\tau=\Pi_x^{e_h}\tau=\Pi_x^{e_h}\tau_1\,\Pi_x^{e_h}\tau_2=(\Pi_x^{e_h})^{M^H}\tau_1\,(\Pi_x^{e_h})^{M^H}\tau_2
\]
where the first equality holds since $\tau \notin \F$ implies $\tau\neq \Xi \I(\Xi)$, hence $M^H \tau = \tau$, and the last equality holds since both the models $Z^{e_h}$ and $(Z^{e_h})^{M^H}$ are admissible (hence their action on $\U^H$ is the same and determined by~\eqref{eq:Amodel}) and the map $M^H$ leaves ${\color{blue}\Xi}$ and ${\color{blue}H}$ invariant.
\end{proof}
As a corollary of the previous result we can also show that a relation analogous to~\eqref{eq:RenExt} holds when substituting the translation operator to the extension one. 

\begin{corollary}\label{cor:RenGroupTh}
Let $h\in L^2(\T^2)$ and $\TS_g$ be the regularity structure defined above. Given $M\in\Re_0$ and $Z=(\Pi,\Gamma)\in\M(\TS_g)$, we have
\begin{equation}\label{eq:RenTrans}
M(T_hZ)=T_hM(Z).
\end{equation}
In other words, the operations of translation and renormalization commute. 
\end{corollary}
\begin{proof}
Since we know that, in our context, the renormalization map does not affect the map $\Gamma$, we only have to show that $T_h(\Pi^M)=(T_h\Pi)^M$. It is immediate to verify, by a direct computation, that, for $M\in\Re_0$ and $M^H$ defined as above, $\tau_H(M\tau)=M^H\tau_H(\tau)$ therefore, recalling the definition of $T_h$ given in Proposition~\ref{prop:TransModel}, for $\tau\in T_0$, we have
\[
(T_h\Pi)_x^M\tau=\Pi_x^hM\tau=\Pi_x^{e_h}\tau_H(M\tau)=\Pi_x^{e_h}M^H\tau_H(\tau)= (\Pi^M)_x^{e_h}\tau_H(\tau)=T_h(\Pi^M)_x\tau
\]
where the fourth equality follows by~\eqref{eq:RenTrans}. Now, since $T_h(\Pi^M)$ and $(T_h\Pi)^M$ coincide on $T_0$, hence in particular on the elements of $T_g$ of negative homogeneity, the uniqueness part of Theorem~\ref{thm:minModel} implies the result. 
\end{proof}

\subsection{Convergence of the Renormalized Models}

Let $(\Omega,\mathscr{F},\PR)$ be a probability space and $\xi$ a spatial white noise on the two dimensional torus, i.e. a Gaussian process taking values in the space of distributions $\SC'(\T^2)$ whose covariance function is given by
\[
\E[\langle\xi,\varphi\rangle\langle\xi,\psi\rangle]=\langle\varphi,\psi\rangle
\]
for any $\varphi,\psi\in L^2(\T^2)$. As is well-known (e.g. Lemma 10.2 of~\cite{Hai}) as a distribution, $\xi$ belongs almost surely to $\CC^{\alpha}$ for every $\alpha<-1$. We want to understand, on one side, how to consistently lift the white noise to an admissible model for $\TS_g$ and $\TS^H_g$ and, on the other, what is the relation between the two. To do so, we begin by mollifying the noise via setting $\xi_\eps:=\xi\ast \varrho_\eps$, where $\varrho$ is a  compactly supported smooth function integrating to 1 and $\varrho_\eps$ its rescaled version. 
Starting with $\xi_\eps$ we define the canonical model $Z_\eps\in\M(\TS_g)$ with the procedure outlined in Section~\ref{subsection:Amodel} and, for a given $h\in L^2(\T^2)$, we then ``extend" it through the map $E_h$ given in Proposition~\ref{prop:ContShiftModel} and ``translate" it through the map $T_h$ given in Proposition~\ref{prop:TransModel}, obtaining $E_hZ_\eps$ and $T_hZ_\eps$ respectively. The problem is that, since the model $Z_\eps$ does not converge there is simply no hope that neither $E_hZ_\eps$ nor $T_hZ_\eps$ do. This is precisely the point in which we need to exploit the renormalization techniques introduced above. 
Thanks to Theorem 10.19 in~\cite{Hai}, we already know that there exists a choice of $M_\eps\in\Re_0$ such that the sequence $M_\eps Z_\eps$ converges in probability, hence passing to a subsequence, almost surely. Since Propositions~\ref{prop:ContShiftModel} and~\ref{prop:TransModel} guarantee the joint local Lipschitz continuity of $E$ and $T$ with respect to both $h$ and the model and Proposition~\ref{prop:RenGrouph} and Corollary~\ref{cor:RenGroupTh} ensure the these maps commute with $M^H$ and $M$ respectively, we immediately deduce that also $M_\eps^H E_hZ_\eps$ and $M_\eps T_hZ_\eps$ converge almost surely along this subsequence. 

\begin{lemma}\label{lemma:TransModel}
Let $\hat{Z}\in\M(\TS_g)$ be the Gaussian model constructed in Theorem 10.19 in~\cite{Hai}. Then there exists a set of measure zero, $N$, such that for every $\omega\in N^c$ and every $h\in L^2(\T^2)$
\begin{equation}\label{eq:Trans}
\hat{Z}(\omega+h)=T_h\hat{Z}(\omega)
\end{equation}
\end{lemma}
\begin{proof}
Let $Z_\eps$ be the canonical model on $\TS_g$ and $M_\eps$ a sequence of renormalization maps such that $\hat{Z}_\eps:=M_\eps Z_\eps$ converges to $\hat{Z}\in\M(\TS_g)$ almost surely. Fix a null set $N_1$ so that, for every $\omega\in N_1^c$, $\hat{Z}_\eps(\omega)\to\hat{Z}(\omega)$. As a consequence of Remark~\ref{rem:TransCanMod}, it is immediate to convince oneself that for all $h\in L^2(\T^2)$
\[
Z_\eps(\omega+h)=T_{h_\eps}Z_\eps(\omega)
\]
outside of some null set $N_2$, where $h_\eps=h\ast\varrho_\eps$. Then, by the local Lipschitz continuity of the map $T$ and Corollary~\ref{cor:RenGroupTh}, for $\omega\in N_1^c\cap N_2^c$, we have
\[
\hat{Z}(\omega+h)=\lim_{\eps\to 0}\hat{Z}_\eps(\omega+h)=\lim_{\eps\to 0}M_\eps T_{h_\eps}Z_\eps(\omega)=\lim_{\eps\to 0} T_{h_\eps}M_\eps Z_\eps(\omega)=T_h\hat{Z}(\omega)
\]
which concludes the proof. 
\end{proof}

\subsection{Modelled Distributions and Fixed Point argument}\label{sec:FPA}

In the previous sections we achieved two goals. On the one side, we built a family of objects that represent the building blocks we need in order to ``lift" the equations~\eqref{eq:gPAM} and~\eqref{eq:gPAMh}. On the other, we gave to each of these objects a precise sense and showed how to coherently construct them starting from a Gaussian noise. It remains to define the spaces in which our equations will be solved at the ``abstract" level and how to concretely interpret them. To this purpose, Hairer defines the space of \textit{modelled distributions}, the model dependent counterpart of the space of H\"older functions. Given a regularity structure $\TS$ and a model $Z=(\Pi,\Gamma)$ on it, we say that $U:\R^+ \times \R^2\to\oplus_{\beta<\gamma}T_{\beta}$ belongs to $\D^{\gamma,\eta}(\Gamma)$ if for every compact domain $D\subset\R^+ \times \R^2 $ 
\begin{equation}\label{eq:ModDistNorm}
\VERT U\VERT_{\gamma,\eta;D}:=\sup_{z\in D}\sup_{\beta<\gamma}|t|^{\frac{\beta-\eta}{2}\vee0}\|U(z)\|_\beta+\sup_{\substack{z,w\in D\\|z-\bar{z}|\leq 1}}\sup_{\beta<\gamma}(|t|\wedge|\bar{t}|)^{\frac{\gamma-\eta}{2}\vee 0}\frac{\|U(z)-\Gamma_{z\bar{z}}U(\bar{z})\|_{\beta}}{|z-\bar{z}|^{\gamma-\beta}}
\end{equation}
is finite, where the generic points $z,\bar{z}\in\R^+ \times \R^2$ have to be understood as $z=(t,x)$ and $\bar{z}=(\bar{t},\bar{x})$, and recall that, by $\|\tau\|_\beta$ we mean the norm of the projection of $\tau$ on $T_\beta$. In order to study the continuity of the solution map with respect to the underlying model, we will need to compare modelled distributions belonging to the space $\D^{\gamma,\eta}$, but based on different models. Let $Z=(\Pi,\Gamma)$, $\bar{Z}=(\bar{\Pi},\bar{\Gamma})$ be two models on $\TS$, and $U\in\D^{\gamma,\eta}(\Gamma)$, $\bar{U}\in\D^{\gamma,\eta}(\bar{\Gamma})$ two modelled distributions, then a natural notion of distance between them can be obtained by~\eqref{eq:ModDistNorm}, via replacing $U(z)-\bar{U}(z)$ by $U(z)$ in the first summand and 
\[
U(z)-\bar{U}(z)-\Gamma_{z\bar{z}}U(\bar{z})+\bar{\Gamma}_{z\bar{z}}\bar{U}(\bar{z})
\]
to $U(z)-\Gamma_{z\bar{z}}U(\bar{z})$ in the second. We indicate the result by $\VERT U;\bar{U}\VERT_{\gamma,\eta;D}$, this notation being due to the fact that, as a distance, $\VERT \cdot;\cdot\VERT_{\gamma,\eta;D}$ is not a function of $U-\bar{U}$. 

\begin{remark}\label{remark:periodicMD}
Since we aim at solving our equations with periodic boundary conditions, we will only consider \textit{symmetric} modelled distributions according to Definition 3.33 in~\cite{Hai}. In other words, let $e_1$ and $e_2$ be as in Remark~\ref{remark:periodicM}, then $U\in\D^{\gamma,\eta}$ is said to be symmetric if for any $(t,x)\in\R^+\times\R^2$, $U(t,x+2\pi e_i)=U(t,x)$. Hence, for any $T>0$, the domain $D$ appearing in~\eqref{eq:ModDistNorm} can be simply taken to be $(0,T]\times\T^2$, and will therefore be omitted.  
\end{remark}

If the model has the role of assigning to each abstract symbol a specific distribution, we also need to understand how to attribute to a modelled distribution a concrete meaning. This is precisely what the \textit{reconstruction operator}, $\RS$, does. In general, $\RS$ is a map from $\D^{\gamma,\eta}(\Gamma)$ to $\SC'(\R^3)$, but in the case in which the model is composed of smooth functions (think, for example, of the canonical model) then $\RS U$ is a continuous function, explicitly given by
\[
\RS U(z)=\left(\Pi_zU(z)\right)(z)
\]
Thanks to Theorem 3.10 in~\cite{Hai} we know much more, indeed the latter states that, as soon as $\gamma>0$ then the map $(Z, U)\mapsto\RS U\in\SC'$ is jointly locally Lipschitz continuous, allowing to define $\RS U$ also in the case in which the previous relation is nonsensical. 

For reasons that will be clarified in what follows, we will abstractly solve equations~\eqref{eq:gPAM} and~\eqref{eq:gPAMh} in the spaces $\D_\U^{\gamma,\eta}$ and $\D_{\U^H}^{\gamma,\eta}$ respectively, consisting of those modelled distributions taking values in $\TS_\U$ and $\TS_{\U^H}$. An element $U\in\D_{\U}^{\gamma,\eta}$ (resp. $\D_{\U^H}^{\gamma,\eta}$), for $\gamma>1$  can be conveniently decomposed as
\begin{equation}\label{eq:dec}
U(z)=\varphi_1(z){\color{blue}1}+\varphi_{\I(\Xi)}(z){\color{blue}\I(\Xi)}+\varphi_X(z){\color{blue}X}
\end{equation}
then Proposition 3.28 in~\cite{Hai} implies that $\RS U=\varphi_1$ and belongs to $\CC^{\alpha_{\min}+2,\eta}$. 
\newline

Let us consider a smooth function $\xi_\eps$ and $h\in L^2(\T^2)$. We rewrite~\eqref{eq:gPAM} and~\eqref{eq:gPAMh} in their mild formulation, i.e.
\begin{equation}\label{eq:gPAMmild}
u=\bar{K}\ast(g(u)\xi_\eps)+\bar{K}u_0,\qquad u^h=\bar{K}\ast(g(u^h)(\xi_\eps+h))+\bar{K}u_0  
\end{equation}
where $\bar{K}$ denotes the heat kernel, $\ast$ the space-time convolution and $\bar{K}u_0$ the solution to the heat equation with $u_0$ as initial condition. We want to transpose such a representation and rephrase it in terms of modelled distributions. To do so, we need to understand how to compose a smooth function with an element of $\D_{\U}^{\gamma,\eta}$ (resp. $\D_{\U^H}^{\gamma,\eta}$), how to define the product of two modelled distributions and what is the abstract counterpart of the convolution with a suitably defined abstract heat kernel. 

Let $U\in\D_{\U}^{\gamma,\eta}$ (resp. $\D_{\U^H}^{\gamma,\eta}$), $\gamma>1$ and $g:\R\to\R$ be a smooth function (actually, for later purposes, $g\in\CC^\chi$ with $\chi\geq \frac{10}{3}$ would be sufficient). Thanks to the fact that $U$ admits the decomposition~\eqref{eq:dec}, we can follow the recipe described in Section 4.2 in~\cite{Hai} and write
\begin{equation} \label{equ:G}
(G_\gamma(U))(z)= g(\varphi_1(z)){\color{blue}1}+g'(\varphi_1(z))\varphi_{\I(\Xi)}(z){\color{blue}\I(\Xi)}+g'(\varphi_1(z))\varphi_X(z){\color{blue}X}
\end{equation}
then Proposition 6.13 in~\cite{Hai} guarantees that $G_\gamma$ as a function from $\D_{\U}^{\gamma,\eta}$ (resp. $\D_{\U^H}^{\gamma,\eta}$) to itself, is locally Lipschitz continuous provided that $\gamma>0$ and $\eta\in[0,\gamma]$. Moreover, in~\cite{HP} a stronger result is shown, namely Proposition 3.11 allows us to compare $G_\gamma$ when evaluated at modelled distributions based at different models, yielding the local Lipschitz continuity of $G_\gamma$ also with respect to the models. 

Concerning the convolution with the heat kernel, it is possible to summarize the content of Theorem 5.12, Proposition 6.16 and Theorem 7.1 in~\cite{Hai} simply saying that, provided that ${\gamma}<\bar{\gamma}+2$, $\eta<\alpha_{\min}\wedge\bar{\eta}+2$ and $\bar{\eta}>-2$, there exists a linear operator $\PK:\D^{\bar{\gamma},\bar{\eta}}\to\D_{\U}^{\gamma,\eta}$ such that
\begin{enumerate}
\item one has the identity $\RS \PK U=\bar{K}\RS U$,
\item $\PK U=\I U + \tilde{\PK} U$, where $\tilde{\PK} U$ takes value in the polynomial structure and depends on the model and the reconstruction operator associated to it,
\item there exists $\theta>0$ such that 
\begin{equation}\label{bound:Schauder}
\VERT \PK U\VERT_{\gamma,\eta}\lesssim T^\theta\VERT U \VERT_{\bar{\gamma},\bar{\eta}}
\end{equation}
where the norms are taken over $[0,T]\times\R^2$ (or equivalently $[0,T]\times\T^2$, by periodicity).
\end{enumerate}
Before writing the abstract version of~\eqref{eq:gPAMmild}, we collect in the following Lemma a number of trivial consistency relations between modelled distributions based at an admissible model and its extended and translated counterpart. In particular, it explains how translation and extension behave with respect to the operations just described. 

\begin{lemma}\label{lemma:consistency} 
Let $\TS_g$ and $\TS_g^H$ be the regularity structures introduced in sections~\ref{section:RS} and~\ref{section:RSh}, $\tau_H$ the abstract translation operator and $h\in L^2(\T^2)$. Let $Z=(\Pi,\Gamma)\in\M(\TS_g)$, $Z^h=(\Pi^h,\Gamma^h)\in\M(\TS_g)$ its translated version, $Z^{e_h}=(\Pi^{e_h},\Gamma^{e_h})\in\M(\TS_g^H)$ its extended one, and $\RS$, $\RS^h$ and $\RS^{e_h}$ their respective reconstruction operators. Then, for $\gamma>0$ and $\eta\in[0,\gamma]$, and every  $U\in\D^{\gamma,\eta}(\Gamma)$ and $U^h\in\D^{\gamma,\eta}(\Gamma^h)$, we have
\begin{enumerate}
\item $\tau_H(U^h)\in\D^{\gamma,\eta}(\Gamma^{e_h})$ and $\tau_H$ commutes with the operations of composition with smooth functions (for $U^h$ taking values in $\TS_\U$) and product between modelled distributions;
\item $\RS U=\RS^{e_h} U$ and $\RS^{e_h} \tau_H(U^h)=\RS^h U^h$, where, a priori, the previous equalities have to be understood in the sense of distributions;
\item $\PK^{e_h}(U)=\PK(U)$ and $\PK^{e_h}\tau_H(U^h)=\tau_H(\PK^{h}(U^h))$, where $\PK$, $\PK^{h}$ and $\PK^{e_h}$ are the abstract convolution kernels associated to $Z$, $Z^h$ and $Z^{e_h}$ respectively. 
\end{enumerate}
\end{lemma}

\begin{proof}
See Appendix~\ref{remark}.
\end{proof}

The last ingredient we need in order to be able to rewrite the equations in~\eqref{eq:gPAMmild} in our abstract context, is the initial condition. 
Given $u_0 \in\CC^\eta (\T^2)$, $\eta \ge 0$ (recall $\CC^0 \equiv L^\infty$) it is well-known that $\bar{K} u_0 \in \CC^{\gamma,\eta}$, the (parabolic) H\"older space whose norm was defined in~\eqref{def:HN}), for any $\gamma>\eta \ge 0$, where $\eta$ accounts for the behaviour at time zero. In particular, then
the (parabolic) jet of order $\gamma$,
\[
(\mathscr{T}_\gamma \bar{K} u_0) (z)=\sum_{|k|<\gamma}\frac{{\color{blue}X^k}}{k!}(D^k (\bar{K} u_0))(z).
\]
is well-defined and yields an element in $\D_{\U}^{\gamma,\eta}$ on $[0,T]$ for every fixed $T>0$, and hence in $\D_{\U}^{\gamma}$ on $(0,T]$. Then, we can write
\begin{gather}
U=\PK^{e_h} (G_\gamma(U){\color{blue}\Xi})+\mathscr{T}_\gamma\bar{K}u_0,\label{eq:gPAMabs}\\
 U^H=\PK^{e_h} (G_\gamma(U^H)({\color{blue}\Xi}+{\color{blue}H}))+\mathscr{T}_\gamma\bar{K}u_0\label{eq:gPAMhabs}
\end{gather}
where we are indicating with the same symbol the two abstract convolution kernels since, thanks to Lemma~\ref{lemma:consistency}, there is no possibility of confusion. 

In the next proposition we recall the solution theory for the previous equations, essentially given in Corollary 9.3 and Proposition 9.4 in~\cite{Hai}. 

\begin{proposition}\label{prop:solgPAM}
Let $\alpha_{\min}\in(-\frac{4}{3},-1)$, $\gamma\in(|\alpha_{\min}|,\frac{4}{3})$ and $\eta\in[0,\alpha_{\min}+2)$. Then for every admissible model $Z\in\M(\TS_g)$ and initial condition $u_0\in\CC^\eta$, the equation~\eqref{eq:gPAMabs} admits a unique solution in $\D_{\U}^{\gamma}$ on $(0,T)$ for $T>0$ small enough. Setting $T_\infty:=T_\infty(u_0,Z)$ to be the supremum of the times $T$ such that~\eqref{eq:gPAMabs} admits a unique fixed point, one has either $T_\infty=\infty$ or $\lim_{t\to T_\infty}\|\RS U(t,\cdot)\|_\eta=\infty$. Furthermore, the map $\Sol$ that assignes to $(u_0,Z)\in\CC^\eta\times\M(\TS_g)$ the solution $U=\Sol(u_0,Z)$, is jointly locally Lipschitz continuous and, as a consequence, $T_\infty$ is lower-semicontinuous as a function of $(u_0,Z)$. 

Let $Z_\eps\in\M(\TS_g)$ be the canonical model, associated to smooth $\xi_\eps$, then $u_\eps=\RS \Sol(u_0,Z_\eps)$ solves
\[
\partial_t u_\eps =  \Delta u_\eps+  g(u_\eps) \xi_\eps,\qquad u_\eps(0,\cdot)=u_0(\cdot).
\]
On the other hand, for $M=M(C) \in\Re_0$,  $\tilde{u}_\eps=\RS^M \Sol(u_0,MZ_\eps)$ solves
\begin{equation}\label{eq:rengPAM}
\partial_t \tilde{u}_\eps=\Delta \tilde{u}_\eps+g(\tilde{u}_\eps)(\xi_\eps-Cg'(\tilde{u}_\eps)),\qquad \tilde u_\eps(0,\cdot)=u_0(\cdot).
\end{equation}
\end{proposition}
\begin{proof}
As already pointed out, the statement and its proof were already given in Corollary 9.3 of~\cite{Hai}. The only details we added are the local Lipschitz continuity of the solution map, which is implied by Proposition 3.11 in~\cite{HP}, the lower-semicontinuity of the existence time $T_\infty$, whose proof coincides \textit{mutatis-mutandis} with the one given by Hairer in Proposition 1.5 of~\cite{hairer_solving_2013} and the fact that we can take the initial condition to be in $\CC^0 \equiv L^\infty$, which comes from the proof of Theorem 3.10 in~\cite{HP}. 
\end{proof}

As a consequence of the previous and Theorem 10.19 in~\cite{Hai}, Theorem 1.11 in~\cite{Hai} follows at once. Below, we recall this latter statement. 

\begin{theorem}\label{thm:solgPAM}
In the same setting as above, let furthermore $\xi$ be a spatial white noise, $\xi_\eps=\varrho_\eps\ast\xi$ its mollification and $Z_\eps$ the canonical model associated to it. Let $M_\eps$ be the sequence of renormalization maps determined in Theorem 10.19 of~\cite{Hai}, i.e. such that $M_\eps Z_\eps$ converges in probability to $\hat{Z}\in\M(\TS_g)$. 
Then, $\tilde u_\eps =\RS^{M_\eps} \Sol(u_0,M_\eps Z_\eps)$ converges locally uniformly, i.e. on compacts in $\R^+ \times \T^2$,  to a limit $u=\RS\Sol(u_0,\hat{Z})$, in probability.  
\end{theorem}

Since the homogeneities of the symbols ${\color{blue}H}$ and ${\color{blue}\Xi}$ are the same by construction, Proposition~\ref{prop:solgPAM} and Theorem~\ref{thm:solgPAM} hold for~\eqref{eq:gPAMhabs} as well. Nevertheless, in this case, we will not be interested in \textit{general} admissible models on $\TS_g^H$ but on those coming from an element in $\M(\TS_g)$ and consequently mapped to $\M(\TS_g^H)$ through $E_h$ defined in Proposition~\ref{prop:ContShiftModel}. The purpose of the following statement is indeed to clarify what is the relation between~\eqref{eq:gPAMabs} and~\eqref{eq:gPAMhabs}, and to understand how the solution map is affected by the operations of translation and extension. 

\begin{proposition}\label{prop:solgPAMh}
In the same setting as Proposition~\ref{prop:solgPAM}, let $\Sol$ be the map that assignes to $(u_0,Z)\in\CC^\eta\times\M(\TS_g)$ the solution $U=\Sol(u_0,Z)$ to~\eqref{eq:gPAMabs}, and $\Sol^H$ be the one that assigns to $(u_0,Z^H)\in\CC^\eta\times\M(\TS_g^H)$ the solution $U^H\in\D^{\gamma,\eta}_{\U^H}(\Gamma^H)$ to~\eqref{eq:gPAMhabs}. For $h\in L^2(\T^2)$ and $Z\in\M(\TS_g)$, let $\Sol^H_{\text{Ex}}(u_0,h,Z):=\Sol^H(u_0,E_h Z)$ and $\Sol_{\text{Tr}}(u_0,h,Z):=\Sol(u_0,T_h Z)$. Then $\Sol^H_{\text{Ex}}$ and $\Sol_{\text{Tr}}$ are jointly locally Lipschitz continuous and $\Sol^H_{\text{Ex}}(u_0,h,Z)=\tau_H(\Sol_{\text{Tr}}(u_0,h,Z))$. 

Furthermore, let $Z_\eps$ be the canonical model on $\TS_g$ associated to a smooth function $\xi_\eps$, and take also $h_\eps$ smooth (and hence $L^2$) on the $\T^2$. Then $u^{h_\eps}_\eps=\RS \Sol_{\text{Tr}}(u_0,h_\eps, Z_\eps)=\RS \Sol^H_{\text{Ex}}(u_0,h_\eps, Z_\eps)$ solves 
\[
\partial_t {u}^{h_\eps}_\eps=\Delta{u}^{h_\eps}_\eps+g({u}^{h_\eps}_\eps)(\xi_\eps+h_\eps),\qquad u_\eps(0,\cdot)=u_0(\cdot).
\]
On the other hand, for $M=M(C)\in\Re_0$, $\tilde{u}^{h_\eps}_\eps=\RS^M \Sol_{\text{Tr}}(u_0,h_\eps, MZ_\eps)=\RS^{M^H} \Sol^H_{\text{Ex}}(u_0,h_\eps, MZ_\eps)$ solves
\begin{equation}\label{eq:rengPAMh}
\partial_t \tilde{u}^{h_\eps}_\eps=\Delta\tilde{u}^{h_\eps}_\eps+g(\tilde{u}^{h_\eps}_\eps)(\xi_\eps+h_\eps-Cg'(\tilde{u}^{h_\eps}_\eps)),\qquad \tilde u_\eps(0,\cdot)=u_0(\cdot).
\end{equation}
\end{proposition}
\begin{proof}
The local Lipschitz continuity of $\Sol^H_{\text{Ex}}$ and $\Sol_{\text{Tr}}$ is a direct consequence of the local Lipschitz continuity of the extension map $E$, the translation map $T$ as well as the one of $\Sol$ and $\Sol^H$. 

Fix $u_0\in\CC^\eta$, $Z\in\M(\TS_g)$ and $h\in L^2(\T^2)$. In order to prove that $\Sol^H_{\text{Ex}}(u_0,h,Z)=\tau_H(\Sol_{\text{Tr}}(u_0,h,Z))$, name the left-hand side  $U^H$ and the right-hand side $\tau_H(U^h)$, where $U^h=\Sol_{\text{Tr}}(u_0,h,Z)$. At this point, thanks to Lemma~\ref{lemma:consistency}, on one side we know that $\tau_H(U^h)\in\D^{\gamma,\eta}_{\U^H}(\Gamma^{e_h})$, while on the other hand
\[
\PK^{e_h} (G_\gamma(\tau_H(U^h))({\color{blue}\Xi}+{\color{blue}H}))=\PK^{e_h} (\tau_H(G_\gamma(U^h))\tau_H({\color{blue}\Xi}))=\PK^{e_h} (\tau_H(G_\gamma(U^h){\color{blue}\Xi}))=\tau_H(\PK^{e_h}(G_\gamma(U^h){\color{blue}\Xi}))
\]
and since, by assumption, $U^h$ solves~\eqref{eq:gPAMabs} with respect to $Z^h$, we have
\[
\PK^{e_h} (G_\gamma(\tau_H(U^h))({\color{blue}\Xi}+{\color{blue}H}))+\mathscr{T}_\gamma\bar{K}u_0=\tau_H(\PK^{e_h}(G_\gamma(U^h){\color{blue}\Xi})+\mathscr{T}_\gamma\bar{K}u_0)=\tau_H(U^h)
\]
in other words $\tau_H(U^h)\in\D^{\gamma,\eta}_{\U^H}(\Gamma^{e_h})$ solves~\eqref{eq:gPAMhabs} and by uniqueness it coincides with $U^H$.

As a consequence of Lemma~\ref{lemma:consistency} and following the same argument as in the proof of Proposition 9.4 of~\cite{Hai}, the last part of the statement can be shown. 
\end{proof}

The following theorem is now straightforward.

\begin{theorem}\label{thm:solgPAMh}
In the same setting as Proposition~\ref{prop:solgPAMh} and Theorem~\ref{thm:solgPAM}, in particular with $h\in L^2(\T^2)$, set $h_\eps:=h\ast\varrho_\eps$. Then $\tilde u_\eps^{h_\eps}=\RS^{M_\eps} \Sol_{\text{Tr}}(u_0,h_\eps, M_\eps Z_\eps)=\RS^{M^H_\eps} \Sol^H_{\text{Ex}}(u_0,h_\eps, M_\eps Z_\eps)$ converges locally uniformly to a limit $u^h=\RS\Sol_{\text{Tr}}(u_0,h,\hat{Z})=\RS \Sol^H_{\text{Ex}}(u_0,h, \hat{Z})$, in probability. 
\end{theorem}

\begin{proof}
The result is a straightforward application of the previous proposition and the fact that both the extension and translation operators are locally Lipschitz continuous, note $h_\eps \to h$ in $L^2(\T^2)$, and commute with the renormalization maps.  
\end{proof}

To conclude this section we want to show that we can solve the afore mentioned equations up to the same time, uniformly in $h$ belonging to a small ball, which by now is a simple corollary of Propositions~\ref{prop:solgPAM} and~\ref{prop:solgPAMh}.

\begin{corollary}\label{cor:time}
In the same setting as Lemma~\ref{lemma:consistency} and Proposition~\ref{prop:solgPAM}, let $U\in\D^{\gamma,\eta}_{\U}(\Gamma)$ be the unique solution to~\eqref{eq:gPAMabs} and $U^H\in\D^{\gamma,\eta}_{\U^H}(\Gamma^{e_h})$ be the unique solution to~\eqref{eq:gPAMhabs}. Then, for every $T<T_\infty(u_0,Z)$ there exists $\delta>0$ such that $U$ and $U^H$ exist up to $T$, uniformly over $h\in L^2(\T^2)$ with $\|h\|_{L^2}<\delta$.   
\end{corollary}
\begin{proof}
Let $Z=(\Pi,\Gamma)\in\M(\TS_g)$, $U\in\D_\U^{\gamma,\eta}(\Gamma)$ be the unique maximal solution to~\eqref{eq:gPAMabs} and $T_\infty(u_0,Z)$ its explosion time. Let $h\in L^2(\T^2)$,  $T_hZ=Z^h\in\M(\TS_g)$ be the translation of $Z$ in the $h$-direction, $U^h\in\D_\U^{\gamma,\eta}(\Gamma^h)$ be the unique maximal solution to~\eqref{eq:gPAMabs} and $T_\infty(u_0,Z^h)$ its explosion time. Notice that, trivially, $Z=T_0 Z$, hence, by the local Lipschitz continuity of the map $T$ in $h$ we know that we can control the norm of the difference between $Z$ and $Z^h$ in terms of the $L^2$-norm of $h$. Since $T_\infty$ is lower-semicontinuous, by definition we have that for every $\eps>0$ there exists $\delta>0$ such that $T_\infty(u_0,\tilde{Z})>T_\infty(u_0,Z)-\eps$ for every $\tilde{Z}\in\M(\TS_g)$ such that $\VERT \tilde{Z}-Z\VERT<\delta$. Hence, upon choosing a smaller $\delta$, for every $h$ with $\|h\|_{L^2}<\delta$, $U^h$ and $U$ live at least up to $T_\infty(u_0,Z)-\eps$.
But now, thanks to Proposition~\ref{prop:solgPAMh} we know that $\tau_H(U^h)=U^H$ and the proof is concluded.
\end{proof}

\subsection{Weak maximum principles and gPAM}


\subsubsection{Global existence for a class of non-linear $g$}\label{subsec:GEgPAM}

As is well-known and summarized in Propositions~\ref{prop:solgPAM} and~\ref{thm:solgPAM}, one has uniqueness and \textit{local} existence for (renormalized) solutions to gPAM. (Throughout $g$ is assumed to sufficiently smooth in order to be in the framework \cite{Hai}.) When $g=g(u)$ is (affine) linear, then global existence holds. For a generic non-linearity $g$, however, global existence may fail, especially if no further growth assumptions on $g$ are 
made.\footnote{Think of the well-studied blowup of semilinear equations such as $\left( \partial _{t}-\Delta \right) u=u^{p}$ for $p>1$.} Essentially, this is due to the fact that the left of $g$, that is $U \mapsto G_\gamma(U)$, as defined in 
(\ref{equ:G}), is locally but {\it not} globally Lipschitz. As a consequence, there have been no global existence results for non-linear gPAM in the literature, even in the example of compactly supported $g$ or $g(\cdot)=\sin(\cdot)$. 

We observe in this section that a weak maximum principle, in the form of a {\it comparison argument}, provides uniform bounds which guarantee global existence. (This comes at the price of a structural assumption on $g$, satisfied in the afore-mentioned examples.)

\begin{proposition}\label{prop:GEgPAM}
\bigskip For fixed initial data $u_{0}^{\varepsilon }=u_{0} \in \CC^\eta$, with $\eta \ge0$, assume
\[
\exists\,\, C\geq \left\Vert u_{0}\right\Vert _{\infty }:g\left( C\right) =g\left( -C\right)=0.
\]
Then solutions $\tilde u_\eps$ to (\ref{eq:rengPAM}) are uniformly bounded on $\R^+\times \T^2$. As a consequence, the (renormalized) gPAM solution exists at all positive times.
\end{proposition}

\begin{proof}
We focus on 
\[
\sup_{\varepsilon \in (0,1]}\sup_{[0,T]\times \mathbb{T}^{2}}u^{\varepsilon
}\left( t,x\right) <\infty ,
\]
leaving a similar lower bound to the reader. Throughout $\varepsilon \in
(0,1]$ is fixed. Our assumption implies that $v\left( t,x\right) :=C$ is (trivially) a solution to the equation given in , that is
\[
\partial_t \tilde{u}_\eps=\Delta \tilde{u}_\eps+g(\tilde{u}_\eps)(\xi_\eps-Cg'(\tilde{u}_\eps)),
\]
and in fact a super-solution to the Cauchy problem with initial data $u_0$, since $u_{0}\leq C\equiv v\left( 0,\cdot \right)$. Hence, by comparison,\footnote{Knowing that we have classical solutions to the (\ref{eq:rengPAM}), this is in fact a simple consequence of calculus ...}
$\tilde{u}_\eps \left( t,x\right) \leq C$,
which is the desired uniform estimate for $C$ does not depend on $\varepsilon $.
\end{proof} 

%

\subsubsection{Weak maximum principle for the renormalized tangent equation}\label{subsec:WMP}

Formally differentiating the noise of  \eqref{eq:rengPAM} in $h_\eps$ direction, one is lead to a linear, inhomogenous (``renormalized tangent equation" ; cf. also ~\eqref{eq:renDergPAM} below) of the form
\begin{equation}\label{eq:renDergPAM2}
\partial_t \tilde{v}_\eps^{h_\eps}=\Delta \tilde{v}_\eps^{h_\eps}+g(\tilde u_\eps)h_\eps+\tilde{v}_\eps^{h_\eps}\Big(g'(\tilde u_\eps)\xi_\eps-C\Big( (g'(\tilde u_\eps))^2+g''(\tilde u_\eps)g(\tilde u_\eps)\Big)\Big)
\end{equation}
(This, as well as the convergence of the renormalized tangent equation is discussed in the following section.) By Duhamel's principle, it is usually enough to study the homogenous problem, that is,
\begin{equation}\label{eq:renDergPAMhom2}
\partial_t \tilde{v}_\eps^{\hom}=\Delta \tilde{v}_\eps^{\hom}+\tilde{v}_\eps^{\hom}\Big(g'(\tilde u_\eps)\xi_\eps-C\Big( (g'(\tilde u_\eps))^2+g''(\tilde u_\eps)g(\tilde u_\eps)\Big)\Big)
\end{equation}
with given initial data $\tilde{v}_\eps^{\hom} (0, \cdot) \equiv v_0 ^{\hom}  \in \CC^\eta$, $\eta \ge 0$. A uniform (in $\eps$) weak maximum principles holds.

\begin{proposition}\label{prop:WMA}
Assume $ v_0 ^{\hom} \ge 0$. Then $ \tilde{v}_\eps^{\hom} \ge 0$ on $\R^+ \times \T^2$, for all $\eps \in (0,1]$. 
\end{proposition} 
\begin{proof}  There is little to say. The unique (classical) solution to \eqref{eq:renDergPAMhom2} is given by the Feynman-Kac formula, which trivially implies non-negativity. Alternatively, 
use the fact that comparison holds for \eqref{eq:renDergPAMhom2}, which allows to conclude since $0$ is a subsolution.
\end{proof}

\section{Differentiating the solution map}\label{sec:MallDiff}

The purpose of this section is to show that we are allowed to differentiate the solution map $\Sol_{\text{Ex}}^H$ (defined in Proposition \ref{prop:solgPAMh}) in the direction of $h\in L^2(\T^2)$. Heuristically, for a smooth $\xi_\eps$ and $\delta>0$, let $u_\eps$ be the solution to 
\begin{equation*} 
(\partial_t - \Delta) u_\eps = g(u_\eps) \xi_\eps,\qquad u_\eps(0,\cdot)=u_0(\cdot)
\end{equation*}
and $u_\eps^\delta$ the solution of the same equation with $\xi_\eps$ replaced by $\xi_\eps + \delta h$. 
Then, assuming $\lim_{\delta\to 0} \delta^{-1}(u^\delta_\eps-u_\eps)$ exists and is given by a function $v^h_\eps$, one can guess that the latter should satisfy
\begin{equation}\label{eq:DergPAM}
(\partial_t - \Delta) v^h_\eps=  g'(u_\eps)v^h_\eps\xi_\eps + g(u_\eps)h,\qquad v_0(\cdot)=0.
\end{equation}
We will refer to the previous as the \textit{tangent equation}. Not only is this computation formal, but we know that neither $u_\eps$ nor $v^h_\eps$ can be expected to converge as $\eps \to 0$. The remedy is to work with $\tilde v^h_\eps$, a {\it renormalization} of $v^h_\eps$ so
that 
$$
   \tilde v^h_\eps \to v^h
$$
similar to convergence of $\tilde u_\eps \to u$ previously given in Theorem \ref{thm:solgPAM}. 
We will begin by giving a consistent solution theory for~\eqref{eq:DergPAM}, prove its linearity and continuity with respect to $h$, and conclude by showing that it is indeed the Malliavin derivative of the solution to~\eqref{eq:gPAM}.

\subsection{The Malliavin Derivative}

Let $h\in L^2(\T^2)$, $Z\in\M(\TS_g)$ and $E_hZ=(\Pi^{e_h},\Gamma^{e_h})\in\M(\TS_g^H)$ be the extension of $Z$ in the $h$ direction as defined in Proposition~\ref{prop:ContShiftModel}. Thanks to the results in the previous section, we know how to lift~\eqref{eq:DergPAM} to the space of modelled distributions based at $E_hZ$, and we have
\begin{equation}\label{eq:DergPAMabs}
V^h = \PK^{e_h} (G_\gamma(U) {\color{blue} H} + G_\gamma'(U)  {\color{blue}\Xi}  V^h)
\end{equation}
where $U\in\D_\U^{\gamma,\eta}(\Gamma)\subset\D_\U^{\gamma,\eta}(\Gamma^{e_h})$ is the unique solution to~\eqref{eq:gPAMabs} according to Proposition~\ref{prop:solgPAM}, which we will refer to as the \textit{abstract tangent equation}.
(Here $G_\gamma'(U)$ is defined as in (\ref{equ:G}) but with $g'$ instead of $g$.) 

\begin{proposition}\label{prop:solDergPAM}
Let $\alpha_{\min}\in(-\frac{4}{3},-1)$, $\gamma\in(|\alpha_{\min}|,\frac{4}{3})$ and $\eta\in[0,\alpha_{\min}+2)$. Let $Z=(\Pi,\Gamma)\in\M(\TS_g)$ and $U\in\D_\U^{\gamma,\eta}(\Gamma)$ be the unique maximal solution to~\eqref{eq:gPAMabs} on $(0,T_\infty)$. Let $h\in L^2(\T^2)$ and $E_hZ=(\Pi^{e_h},\Gamma^{e_h})\in\M(\TS_g^H)$, then equation~\eqref{eq:DergPAMabs} admits a unique solution $V^h\in\D^{\gamma,\eta}_{\U^H}(\Gamma^{e_h})$ on $(0,T_\infty)$. Moreover, the map $\Sol^D$ that assigns to $(h,Z)\in L^2(\T^2)\times \M(\TS_g)$ the solution $V^h=\Sol^D(h,Z)$ is jointly locally Lipschitz continuous.

Furthermore, let $V^{h_\eps}=\Sol^D(h_\eps,Z_\eps)$ be the solution of~\eqref{eq:DergPAMabs} with respect to the canonical model,$Z_\eps\in\M(\TS_g)$, associated to a smooth function $\xi_\eps$, 
and take also $h_\eps$ smooth (and hence $L^2$) on the $\T^2$, 
then $v^{h_\eps}_\eps=\RS V^{h_\eps}$ solves~\eqref{eq:DergPAM}. On the other hand, for $M=M(C)\in\Re_0$, 
$\tilde{v}_\eps^{h_\eps}=\RS^{M^H} \Sol^D(h_\eps,MZ_\eps)$ solves
\begin{equation}\label{eq:renDergPAM}
\partial_t \tilde{v}_\eps^{h_\eps}=\Delta \tilde{v}_\eps^{h_\eps}+g(\tilde u_\eps)h_\eps+\tilde{v}_\eps^{h_\eps}\Big(g'(\tilde u_\eps)\xi_\eps-C\Big( (g'(\tilde u_\eps))^2+g''(\tilde u_\eps)g(\tilde u_\eps)\Big)\Big),\qquad \tilde{v}_\eps^{h_\eps}(0,\cdot)=0
\end{equation}
where $\tilde u_\eps=\RS^M\Sol(u_0, MZ_\eps)$ is given according to Proposition~\ref{prop:solgPAM}. We will refer to this latter equation as the renormalized tangent equation. 
\end{proposition}

\begin{remark}
In contrast to the equation for $U$, the equation for $V^h$ is a {\it linear}, inhomogenous equation which in fact allows to solve it in the space $\D^{\gamma, \eta'}$ with $\eta'\in(-\alpha_{\min}-2,\alpha_{\min}+2)$, which, for $\eta'\leq 0$, is contained in the space $\D^{\gamma,\eta}$. The reason why we stick with the latter space is that, on the other hand, we cannot solve~\eqref{eq:gPAMabs} in $\D^{\gamma, \eta'}$, for $\eta'<0$, so we decided to formulate the previous proposition in this fashion in order to streamline the presentation. 
\end{remark}

\begin{remark}
As a sanity check, we point out that the renormalized equation~\eqref{eq:renDergPAM} is the same as the one obtained by directly differentiating~\eqref{eq:rengPAMh} in the $h_\eps$ direction. 
\end{remark}

\begin{proof}
For the first part of the statement, we need to check that the assumptions of Theorem 7.8 in~\cite{Hai} are matched, i.e. we have to prove that the map $F_\gamma$ acting on $\D^{\gamma,\eta}_{\U^H}(\Gamma^{e_h})$ and given by
\[
F_\gamma(V^h)=G_\gamma(U) {\color{blue} H} + G_\gamma'(U)  {\color{blue}\Xi}  V^h
\]
is strongly locally Lipschitz in the terminology of Section 7.3 in~\cite{Hai}. Now, the functions $x\mapsto{\color{blue}H}$ and $x\mapsto{\color{blue}\Xi}$ can be viewed as modelled distributions in $\D^{\tilde{\gamma},\tilde{\gamma}}$, for every $\tilde{\gamma}>0$,  taking values in a sector of regularity $\alpha_{\min}$. Moreover, by Proposition 6.13 in~\cite{Hai}, since $U\in\D_\U^{\gamma,\eta}(\Gamma)\subset\D^{\gamma,\eta}_{\U^H}(\Gamma^{e_h})$, then $G_\gamma(U)$ and $G_\gamma'(U)$ live in the same space. Hence, thanks to Proposition 6.12 in~\cite{Hai} both $G_\gamma(U) {\color{blue} H}$ and $G_\gamma'(U)  {\color{blue}\Xi}  V^h$ belong to $\D^{\gamma+\alpha_{\min},\eta+\alpha_{\min}}_{\U^H}(\Gamma^{e_h})$. We can therefore conclude that $F_\gamma$ maps $\D^{\gamma,\eta}_{\U^H}(\Gamma^{e_h})$ into $\D^{\gamma+\alpha_{\min},\eta+\alpha_{\min}}_{\U^H}(\Gamma^{e_h})$ and its strong local Lipschitz continuity follows by the fact that this holds for both the operations of composition with smooth functions and product according to the bounds in Proposition 3.11 in~\cite{HP} and 6.12 in~\cite{Hai} respectively. 

At this point, thanks to Theorem 7.8 in~\cite{Hai}, we know that there exists a small $T>0$ such that~\eqref{eq:DergPAMabs} admits a unique solution $V^h\in\D^{\gamma,\eta}_{\U^H}(\Gamma^{e_h})$ on $(0,T)$ and by Proposition 7.11 in~\cite{Hai} we can build a maximal solution by patching together local solutions. But, since the equation is linear in $V^h$, we immediately see that the time $T$ determined above does not depend on the size of $V^h$ itself, but only on the one of $U$, hence we can iterate the procedure until we reach the explosion time of the latter, i.e. $T_\infty$. Finally, the joint local Lipschitz continuity of the solution map $\Sol^D$ with respect to $h\in L^2(\T^2)$ and $Z\in\M(\TS_g)$ follows by the one of the map $E$ (see Proposition~\ref{prop:ContShiftModel}) and Corollary 7.12 in~\cite{Hai}. 

For the last part of the statement, let $T<T_\infty$, $Z_\eps$ the canonical model, $V^{h_\eps}$ the solution to~\eqref{eq:DergPAMabs} in $\D^{\gamma,\eta}$ with respect to $E_{h_\eps}Z_\eps$ and $v_\eps^{h_\eps}:=\RS V^{h_\eps}$. The fact that $v^{h_\eps}$ solves~\eqref{eq:DergPAM} is straightforward and follows by the properties of the reconstruction map and the abstract integration kernel. 

We will then try to understand what is the equation solved by $\tilde{v}^{h_\eps}_\eps:=\RS\tilde{V}^{h_\eps}$ where $\tilde{V}^{h_\eps}= \Sol^D(h_\eps,MZ_\eps)$, essentially following the proof of Proposition 9.4 in~\cite{Hai}. As a consequence of Remark 7.10 of~\cite{Hai}, which gives a way to understand the structure of the solution to a general equation, one obtains the following representation for $\tilde{V}^{h_\eps}$ up to order $\gamma$ (i.e. up to order $1$ here)
\[
\tilde{V}^{h_\eps}(z)=\tilde{v}^{h_\eps}_\eps(z){\color{blue}1}+g'(\tilde u^\eps(z))\tilde{v}^{h_\eps}_\eps(z){\color{blue}\I(\Xi)}+g(\tilde u^\eps(z)){\color{blue}\I(H)}+\sum_{i=1}^2 \tilde{v}^{h_\eps}_{\eps,X_i}(z){\color{blue}X_i}
\]
where $\tilde u^\eps=\RS(\tilde U^{\eps})$ ($\tilde U^{\eps}$ is the solution to~\eqref{eq:gPAMabs} with respect to $MZ_\eps$), $\tilde{v}^{h_\eps}_{\eps,X_i}$ suitable coefficients and $\tilde{v}^{h_\eps}_{\eps}$ is the coefficient of ${\color{blue}1}$ thanks to Proposition 3.28 in~\cite{Hai}. At this point notice that, 
\[
G_\gamma(\tilde U^{\eps}){\color{blue}H}(z)=g(\tilde u^\eps(z)){\color{blue}H}+g'(\tilde u^\eps(z))g(\tilde u^\eps(z)){\color{blue}\I(\Xi)H}+g'(\tilde u^\eps(z)) \left(\sum_{i=1}^2 \tilde u^\eps_{X_i}(z){\color{blue}X_i H}\right)
\]
and
\begin{align*}
G'_\gamma(\tilde U^{\eps})\tilde{V}^{h_\eps}{\color{blue}\Xi}(z)=& g'(\tilde u^\eps(z))\tilde{v}^{h_\eps}_\eps(z){\color{blue}\Xi}+\left(g'(\tilde u^\eps(z))^2\tilde{v}^{h_\eps}_\eps(z)+g''(\tilde u^\eps(z))g( \tilde u^\eps(z)) \tilde{v}^{h_\eps}_\eps(z)\right){\color{blue}\I(\Xi)\Xi}\\
&+ g'(\tilde u^\eps(z))g(\tilde u^\eps(z)){\color{blue}\I(H)\Xi}+\sum_{i=1}^2 \left( g'(\tilde u^\eps(z))\tilde{v}^{h_\eps}_{\eps,X_i}(z)+g''(\tilde u^\eps(z))\tilde{v}^{h_\eps}_\eps(z) \tilde u^\eps_{X_i}(z)\right){\color{blue}X_i\Xi}
\end{align*}
where in both cases we stopped our expansion up to 0 homogeneity. Now, we want to apply $M^H = M^H(C)$ to both sides of the two previous equalities. Its definition immediately gives 
\begin{gather*}
M^H(G_\gamma(\tilde U^{\eps}){\color{blue}H}(z))=G_\gamma(\tilde U^{\eps}){\color{blue}H}(z),\\ M^H(G'_\gamma(\tilde U^{\eps})\tilde{V}^{h_\eps}{\color{blue}\Xi}(z))=G'_\gamma(\tilde U^{\eps})\tilde{V}^{h_\eps}{\color{blue}\Xi}(z)-C\left(g'(\tilde u^\eps(z))^2\tilde{v}^{h_\eps}_\eps(z)+g''(\tilde u^\eps(z))g(\tilde u^\eps(z)) \tilde{v}^{h_\eps}_\eps(z)\right){\color{blue}1}
\end{gather*}
By Proposition~\ref{prop:RenGrouph}, $E_{h_\eps}MZ_\eps=M^HE_{h_\eps}Z_\eps$, and in~\eqref{eq:RenModelh} we have set 
\[
\Pi_x^{e_{h_\eps},M^H}\tau=\Pi_x^{e_{h_\eps}} M^H \tau
\]
where $\Pi^{e_{h_\eps}}$ is the canonical model on $\TS_g^H$ (see Remark~\ref{rem:TransCanMod}). Hence, since $\RS F(x)=\Pi_x^{e_{h_\eps}}F(x)(x)$ for any modelled distribution $F$ relative to 
the canonical model, we obtain 
\[
\RS\left(G_\gamma(\tilde U^{\eps}){\color{blue}H}+G'_\gamma(\tilde U^{\eps})\tilde{V}^{h_\eps}{\color{blue}\Xi}\right)(z)=g(\tilde u_\eps(z))h_\eps(z)+\tilde{v}_\eps^h(z)\Big(g'(\tilde u_\eps)\xi_\eps-C\Big( (g'(\tilde u_\eps))^2+g''(\tilde u_\eps)g(\tilde u_\eps)\Big)\Big)(z)
\]
which, by the first property of the abstract integration kernel concludes the proof. 

\end{proof}

We are now ready to state and prove the following theorem, in which, on the one hand we establish the convergence of the sequence of solutions to the renormalized tangent equation to a well-defined object and, on the other, we determine the main properties of the latter.

\begin{theorem}\label{thm:solDergPAM}
In the same setting as Theorems~\ref{thm:solgPAM},~\ref{thm:solgPAMh}  and Proposition~\ref{prop:solDergPAM},  $\tilde{v}_\eps^{h_\eps}=\RS^\eps \Sol^D(h_\eps, M_\eps Z_\eps)$ converges locally uniformly to a limit $v^h=\RS\Sol^D(h, \hat{Z})$, in probability. 

Let $(t,x)\in(0,+\infty)\times\T^2$ and $T_\infty=T_\infty(u_0,\hat{Z}(\omega))$ be the explosion time for the solution to~\eqref{eq:gPAMabs} introduced in Proposition~\ref{prop:solgPAM}, then, for almost all $\omega \in \Omega$, for all $t<T_\infty(\omega)$, the map $h\mapsto v^h(t,x;\omega)$ is linear and continuous. 
\end{theorem}

\begin{proof}
The proof of the first part of the statement is analogous to the proof of Theorem~\ref{thm:solgPAMh} therefore we will focus only on the second. Thanks to Theorem 10.19 in~\cite{Hai}, we know that $M_\eps Z_\eps$ converges to $\hat{Z}$ in probability, hence, taking at most a subsequence, almost surely. Let $N$ be the null set in which such convergence fails and $(t,x)\in(0,+\infty)\times\T^2$. We will prove that for every $\omega\in N^c\cap\{t<T_\infty\}$ the map $h\mapsto v^h(x,t)$ is linear and bounded in $h\in L^2(\T^2)$. 

Take $a_1$, $a_2\in\R$ and $h_1$, $h_2\in L^2(\T^2)$. The previous proposition guarantees that $\tilde{v}_\eps^{a_1h_1^\eps+a_2h_2^\eps}=$ $\RS^{M_\eps^H} \Sol^D(a_1h_1^\eps+a_2h_2^\eps, M_\eps Z_\eps(\omega))$, $\tilde{v}_\eps^{h_1^\eps}=\RS^{M_\eps^H} \Sol^D(h^\eps_1, M_\eps Z_\eps(\omega))$ and $\tilde{v}_\eps^{h_2^\eps}=\RS^{M_\eps^H} \Sol^D(h_2^\eps, M_\eps Z_\eps(\omega))$ solve~\eqref{eq:renDergPAM} with $h_\eps$ substituted by $a_1h_1^\eps+a_2h_2^\eps$, $h_1^\eps$ and $h^\eps_2$ respectively. Since, for this latter equation in which all the noise terms are smooth, existence and uniqueness of solutions hold, it is immediate to verify that
\[
\tilde{v}_\eps^{a_1h_1^\eps+a_2h_2^\eps}=a_1\tilde{v}_\eps^{h_1^\eps}+ a_2\tilde{v}_\eps^{h_2^\eps}
\]
Moreover, thanks to the first part of the statement, we know that all of $\tilde{v}_\eps^{a_1h_1^\eps+a_2h_2^\eps}$, $\tilde{v}_\eps^{h_1^\eps}$ and $\tilde{v}_\eps^{h_2^\eps}$ converge as $\eps$ tends to $0$ to well-defined objects, hence the previous equality is preserved in the limit. 

At this point, let $v^h(t,x)=\RS\Sol^D(h,\hat{Z})$ and recall that the reconstruction operator is locally Lipschitz continuous with respect to both the model and the modelled distribution, and, by Proposition~\ref{prop:solDergPAM} $\Sol^D$ is jointly local Lipschitz continuous with respect to the model and $h$. Hence, the map $L^2(\T^2)\ni h\mapsto v^h(t,x)$ is linear and locally Lipschitz continuous, which trivially guarantees its (even global Lipschitz) continuity.  
\end{proof}

\begin{remark}
It is even possible to get an explicit bound for $v^h$ in terms of $h\in L^2$ and the solution to an auxiliary equation. Indeed, $v^h$ is the solution to a linear equation, hence, for $T<T_\infty$, upon looking at $\cev{v}_\eps^{h_\eps}(\tau,y):=\tilde v_\eps^{h_\eps}(T-\tau,y)$\footnote{In this remark we will always indicate the time reversal of a function by an inverted arrow}, we can apply Feynman-Kac formula so that
\[
\cev{v}_\eps^{h_\eps}(\tau,y)=\E_y\left[\int_\tau^T \exp\left(-\int_\tau^r g'(\cev{u}_\eps(s,B_s))\xi_\eps(B_s)-C_\eps(gg')'(\cev{u}_\eps(s,B_s))\dd s\right) g(\cev{u}_\eps(r,B_r))h_\eps(B_r)\dd r\right]
\]
and Cauchy-Schwarz inequality and, again, Feynman-Kac imply
\[
|\tilde v_\eps^{h_\eps}(t,x)|\lesssim \log(T/(T-t)) |\tilde{w}_\eps(t,x)|^{\frac{1}{2}} \|h_\eps\|_{L^2}
\]
where $\tilde{w}_\eps$ is the solution to the auxiliary equation
\[
\partial_t\tilde{w}_\eps=\Delta \tilde{w}_\eps +g(\cev{u}_\eps)^2+2\tilde{w}_\eps\Big(g'(\cev{u}_\eps)\xi_\eps-C\big(gg'\big)'(\cev{u}_\eps) \Big)\Big),\qquad \tilde{w}_\eps(0,\cdot)=0
\]
which is independent of $h$. Following the same line of reasoning of Proposition~\ref{prop:solDergPAM} and Theorem~\ref{thm:solDergPAM} it is possible to show that $\tilde{w}_\eps$ converges to a well-defined object in the limit as $\eps$ tends to $0$ and so does $\tilde{v}_\eps^{h_\eps}$, which allows to conclude
\[
|\tilde v^{h}(t,x)|\lesssim \log(T/(T-t)) |\tilde{w}(t,x)|^{\frac{1}{2}} \|h\|_{L^2}
\]
where the logarithm is due to the fact that we are solving the equation on the two-dimensional torus. 
\end{remark}

\subsection{Malliavin Differentiability}

Let $\hat{Z}$ be the admissible model defined by Theorem 10.19 in~\cite{Hai}, $u_0\in\CC^\eta$ an initial condition and $T_\infty=T_\infty(u_0,\hat{Z})$ the explosion time for $u=\RS(\Sol(u_0,\hat{Z}))$ introduced in Proposition~\ref{prop:solgPAM}. Given $(t,x)\in(0,+\infty)\times\R^2$, the aim of this section is to show that, for almost every $\omega$ and $t<T_\infty(\omega)$, the random variable 
$$u(t,x;\omega)=\RS\Sol(u_0,\hat{Z}(\omega))(t,x)
$$ is Malliavin differentiable in the precise sense of Definition~\ref{def:MallDiff} and that its Malliavin derivative evaluated at $h\in L^2(\T^2)$ is the function $v^h(t,x;\omega)$ defined in Theorem~\ref{thm:solDergPAM}. 
To do so, we set 
\begin{equation}\label{eq:FullSet}
\Upsilon_t :=  \{t<T_\infty\} \cap \{\omega:\hat{Z}(\omega+h)=T_h\hat{Z}(\omega)\,\text{for all}\,h\in L^2(\T^2)\}  
\end{equation}
and note that, thanks to Lemma~\ref{lemma:TransModel}, the second event has full measure, while the first has 
 positive probability at least for $t$ small enough
 as a consequence of $T_\infty > 0$ a.s.
From now on, we fix an $\omega\in\Upsilon_t$ and, before proceeding, we (recall and) introduce some notations. 

\begin{notation*}
For an admissible model $Z=(\Pi,\Gamma)\in\M(\TS_g)$ and an $L^2$ function $h$, we defined in Proposition~\ref{prop:ContShiftModel} the extended model $E_hZ=(\Pi^{e_h},\Gamma^{e_h})\in\M(\TS_g^H)$ and in Proposition~\ref{prop:TransModel} the translated model $T_hZ=(\Pi^h,\Gamma^h)\in\M(\TS_g)$. We will indicate by $\RS$, $\RS^{e_h}$ and $\RS^h$ the reconstruction operators associated to $Z$, $E_hZ$ and $T_hZ$ respectively, and denote by $\Sol$, $\Sol^H$ and $\Sol^D$ the solution maps for~\eqref{eq:gPAMabs},~\eqref{eq:gPAMhabs} and~\eqref{eq:DergPAMabs} as in the corresponding Propositions~\ref{prop:solgPAM},~\ref{prop:solgPAMh} and~\ref{prop:solDergPAM}. 
\end{notation*}

What we have to show is that the map $L^2(\T^2)\ni h\mapsto u(t,x;\omega+h)\in\R$ is Fr\'echet differentiable, which amounts to verify that it is Gateaux differentiable and that the Gateaux differential is continuous. The main technical difficulty one has to overcome is the Gateaux differentiability at $h=0$, and this will be our first focus. 

Let $\delta>0$ and $h\in L^2(\T^2)$, then $|u(t,x;\omega+\delta h)-u(t,x;\omega)-\delta v^h(t,x;\omega)|=o(\delta)$ will follow by a stronger statement on which we will concentrate, namely 
\begin{equation}\label{eq:differentiability}
\|\RS^{\delta h}(U^{\delta h}(\omega))-\RS(U(\omega))-\delta\RS^{e_h}(V^h(\omega))\|_{\CC^{\alpha_{\min}+2,\eta}}=o(\delta)
\end{equation}
where, to simplify the notations, we have set $\Sol(u_0,T_{\delta h}\hat{Z}(\omega))=:U^{\delta h}(\omega)\in\D^{\gamma,\eta}(\Gamma^{\delta h})$, $\Sol(u_0,\hat{Z}(\omega))=:U(\omega)\in\D^{\gamma,\eta}(\Gamma)$ and $\Sol^D(h,\hat{Z}(\omega))=:V^h(\omega)\in\D^{\gamma,\eta}(\Gamma^{e_h})$.

Since we are aiming at reformulating our problem in the abstract space of modelled distributions, the first problem one has to tackle is that, in~\eqref{eq:differentiability}, inside the norm we have three different reconstruction operators. While, thanks to Lemma~\ref{lemma:consistency}, $\RS(\Sol(u_0,\hat{Z}(\omega)))=\RS^{e_h}(\Sol(u_0,\hat{Z}(\omega)))$ since $\Sol(u_0,\hat{Z}(\omega))$ can be viewed as an element of $\D^{\gamma,\eta}(\Gamma^{e_h})$, for the first summand, the following lemma provides the information we need.

\begin{lemma}\label{lemma:consistency2}
Let $f$, $h\in L^2(\T^2)$, $T_f\hat{Z}=(\hat{\Pi}^f,\hat{\Gamma}^f)$ the translation of $\hat{Z}$ in the $f$-direction and $\RS_f$ the reconstruction operator associated to it. Then there exists a null set out of which we have
\begin{equation}\label{eq:consistency2}
\RS_f^{\delta h}(U^{\delta h})=\RS_f^{e_h}(U^H_\delta),\qquad\text{for all}\,\,h\in L^2(\T^2)
\end{equation}
where $U^{\delta h}$ is the solution to~\eqref{eq:gPAMabs} in $\D^{\gamma,\eta}((\hat{\Gamma}^f)^{\delta h})$ and $U^H_\delta$ the one of~\eqref{eq:gPAMhabs} with ${\color{blue}H}$ substituted by $\delta{\color{blue}H}$, in $\D^{\gamma,\eta}((\hat{\Gamma}^f)^{e_h})$.  
\end{lemma}
\begin{proof}
Let $Z_\eps$ be the canonical model associated to $\xi_\eps=\xi\ast\varrho_\eps$ and $M_\eps=M(C_\eps)$ be the sequence of renormalization maps such that $M_\eps Z_\eps$ converges to $\hat{Z}$ almost surely, and we will call $N$  the set in which such a convergence fails. The joint local Lipschitz continuity of $E$ and $T$ proved in Propositions~\ref{prop:ContShiftModel} and~\ref{prop:TransModel}, then guarantee that, for every $h\in L^2(\T^2)$, also $E_{h_\eps}M_\eps Z_\eps$ and $T_{h_\eps}M_\eps Z_\eps$ converge to $E_h\hat{Z}$ and $T_h\hat{Z}$ on $N^c$, where $h_\eps=h\ast\varrho_\eps$ converges to $h$ in $L^2(\T^2)$. Let $\RS^{e_h,\eps}$ and $\RS^{\delta h, \eps}$ be the reconstruction operators associated to $E_{h_\eps}M_\eps Z_\eps$ and $T_{h_\eps}M_\eps Z_\eps$ respectively. 

Thanks to Proposition~\ref{prop:solgPAMh}, we know that both $u^{\delta h,\eps}:=\RS_{f_\eps}^{\delta h,\eps}U^{\delta h}$ and $u^{H,\eps}_\delta:=\RS_{f_\eps}^{e_h,\eps}U^{H}_\delta$ solve
\[
\partial_t w=\Delta w+g(w)(\xi_\eps+f_\eps+\delta h_\eps-C_\eps g'(w))
\]
where $f_\eps=f\ast\varrho_\eps$ converges to $f$ in $L^2(\T^2)$. By the uniqueness of solutions for the previous, it follows that $u^{\delta h,\eps}=u^{H,\eps}_\delta$ and, since both sides converge to a well-defined object,~\eqref{eq:consistency2} holds on $N^c$, which, we stress once more, is independent of $h$. 

\end{proof}

We have now all the tools and the notations in place to state and prove the following Proposition.

\begin{proposition}\label{prop:differentiabilityabs}
For $h\in L^2(\T^2)$ and $Z\in\M(\TS_g)$, let $E_hZ$ be the extension in $\M(\TS_g^H)$. Let $U$, $U^H_\delta$ and $V^h\in\D^{\gamma,\eta}(\Gamma^{e_h})$ be the solutions to~\eqref{eq:gPAMabs},~\eqref{eq:gPAMhabs}, with $\delta {\color{blue}H}$ substituting ${\color{blue}H}$, and~\eqref{eq:DergPAMabs} respectively. Then, uniformly in $h\in L^2(\T^2)$ such that $\|h\|_{L^2}=1$
\begin{equation}\label{eq:GatDiff}
\VERT U^H_\delta-U-\delta V^h\VERT_{\gamma,\eta}=o(\delta)
\end{equation}
As a consequence, given $(t,x)\in(0,+\infty)\times\R^2$, for every $\omega\in \Upsilon_t$, the map $L^2(\T^2)\ni h\mapsto u(t,x;\omega+h)\in\R$ is Gateaux differentiable at $h=0$ and its Gateaux derivative evaluated at $h$ is given by $v^h(t,x;\omega)$. 
\end{proposition}
\begin{proof}
Let $T<T_\infty(u_0,Z)$. Then, as a consequence of Corollary~\ref{cor:time}, we know that there exists $\bar{\delta}>0$ such that $U^H_\delta$ and $U$ exist up to time $T$, for all $\delta\in(-\bar{\delta},\bar{\delta})$. For the rest of the proof we will consider the space $\D^{\gamma,\eta}(\Gamma^{e_h})$ where the norm is taken over $[0,T]\times\R^2$, which, since the model is fixed and our distributions are periodic, is a Banach space. 

Let $M>0$ and $B_{\gamma,\eta}(U,M)$ an open ball of radius $M$ in $\D_{\U^H}^{\gamma,\eta}(\Gamma^{e_h})$ centered at $U$. We now introduce the product space $\X:=(-\bar{\delta},\bar{\delta})\times \B_{\gamma,\eta}(U,M)$, endowed with the euclidian norm $\|(\delta,Y)\|_\X:=(\delta^2+\VERT Y\VERT_{\gamma,\eta}^2)^{\frac{1}{2}}$, and define the map $F_\gamma:\X\to\D_{\U^H}^{\gamma,\eta}(\Gamma^{e_h})$ as 
\[
F_\gamma(\delta, Y)(z)=Y(z)-\PK^{e_h}(G_\gamma(Y)({\color{blue}\Xi}+\delta{\color{blue}H}))(z)-\mathscr{T}_\gamma\bar{K}u_0
\]
The point to prove here is that $F_\gamma$ satisfies the assumptions of the Implicit Function theorem as stated in Theorem 19.28 of~\cite{BD}, around the point $(0,U)$, since, thanks to Proposition~\ref{prop:solgPAM}, $F_\gamma(0,U)=0$. At first we will show the Fr\'echet differentiability of $F_\gamma$. Since $F_\gamma$ is linear in $\delta$, it suffices to verify it at $(0,Y)$ for $Y\in\B_{\gamma,\eta}(U,M)$. Then, let $\delta\in(-\bar{\delta},\bar{\delta})$, $\tilde{Y}\in\B_{\gamma,\eta}(U,M)$ and notice that 
\begin{align}
F_\gamma(\delta,Y+\tilde{Y})-F_\gamma(0,Y)&=\tilde{Y}-\PK^{e_h}\Big(\big(G_\gamma(Y+\tilde{Y})-G_\gamma(Y)\big){\color{blue}\Xi}+\delta G_\gamma(Y+\tilde{Y}){\color{blue}H}  \Big)\notag\\
&=\tilde{Y}-\PK^{e_h}\Big(G_\gamma'(Y)\tilde{Y}{\color{blue}\Xi}\Big)-\delta \PK^{e_h}\Big( G_\gamma(Y){\color{blue}H}\Big)-R_\gamma(\delta,\tilde{Y})\label{eq:FgammaDiff}
\end{align}
where the remainder $R_\gamma$ is given by
\begin{equation}\label{eq:RemDiff}
R_\gamma(\delta, \tilde{Y})=\PK^{e_h}\Big(\big(G_\gamma(Y+\tilde{Y})-G_\gamma(Y)-G_\gamma'(Y)\tilde{Y}\big){\color{blue}\Xi}\Big)+\delta\PK^{e_h}\Big( \big(G_\gamma(Y+\tilde{Y})-G_\gamma(Y)\big){\color{blue}H}  \Big)
\end{equation}
At this point, the proof boils down to show that $R_\gamma(\delta,\tilde{Y})=o(\|(\delta,\tilde{Y})\|_\X)$. To do so, we will treat the two summands separately. Let us begin with the first. Notice that, 
\[
G_\gamma(Y+\tilde{Y})(z)-G_\gamma(Y)(z)-G_\gamma'(Y)\tilde{Y}(z)=\int_0^1(1-\nu)G_\gamma''(Y+\nu \tilde{Y})(z)\tilde{Y}^2(z)\dd \nu
\]
where the equality follows by applying to each of the coefficients of the modelled distribution on the left-hand side the usual Taylor's formula. By Propositions 6.13 and 6.12 in~\cite{Hai}, we know that, on one side, $G_\gamma''(Y+\nu \tilde{Y})$ is a modelled distribution in $\D_{\U^H}^{\gamma,\eta}(\Gamma^{e_h})$ for every $\nu\in[0,1]$, and, on the other, that also $G_\gamma''(Y+\nu \tilde{Y})\tilde{Y}^2\in\D_{\U^H}^{\gamma,\eta}(\Gamma^{e_h})$. Moreover, we get
\[
\VERT G_\gamma''(Y+\nu \tilde{Y})\tilde{Y}^2\VERT_{\gamma,\eta}\lesssim \VERT \tilde{Y}\VERT_{\gamma,\eta}^2
\]
where the proportionality constant depends on the norm of $\Gamma^{e_h}$, the one of $g$ and its derivatives up to the third order, the one of $U$ and $M$, but it is uniform over $\nu\in[0,1]$. Now, as in the proof of Proposition~\ref{prop:solDergPAM}, we point out that the map $x\mapsto{\color{blue}\Xi}$ can be viewed as an element of $\D^{\gamma,\gamma}(\Gamma^{e_h})$ but taking values in a sector of regularity $\alpha_{\min}$. Hence, again by Proposition 6.12 in~\cite{Hai}, it follows that $G_\gamma''(Y+\nu \tilde{Y})\tilde{Y}^2{\color{blue}\Xi}\in\D^{\gamma+\alpha_{\min},\eta+\alpha_{\min}}(\Gamma^{e_h})$. Finally, Proposition 6.16 guarantees that the first summand in~\eqref{eq:RemDiff} is $O(\VERT \tilde{Y}\VERT_{\gamma,\eta}^2)$. 

For the second summand the procedure is identical since ${\color{blue}H}$ has the same homogeneity as ${\color{blue}\Xi}$ and, therefore, can be analogously regarded. Following the same steps as before, one deduces that the second summand is $O(\delta\VERT\tilde{Y}\VERT_{\gamma,\eta})$. Hence, $R_\gamma(\delta, \tilde{Y})$ is $O(\|(\delta,\tilde{Y})\|_\X^2)$, i.e. $o(\|(\delta,\tilde{Y})\|_\X)$, which in turn implies the differentiability of $F_\gamma$. As a byproduct, we can read off equation~\eqref{eq:FgammaDiff} the exact expressions for $D_1F_\gamma(0,Y)$ and $D_2F_\gamma(0,Y)$, where $D_i$ is the directional derivative of $F_\gamma$ in the $i$-th direction, i.e.
\begin{align}
D_1 F_\gamma(0,Y)(\delta)&=- \delta \PK^{e_h}\Big(G_\gamma(Y)) {\color{blue}H}\Big)\label{eq:der1implicit}\\
D_2 F_\gamma(0,Y)(\tilde{Y})(z)&=\tilde{Y}(z)-\PK^{e_h}\Big(G_\gamma'(Y)\tilde{Y} {\color{blue}\Xi}\Big)(z)\label{eq:der2implicit}
\end{align}
where $D_iF_\gamma(\cdot,\cdot)$ are two linear functionals from $\R$ and $\D_{\U^H}^{\gamma,\eta}(\Gamma^{e_h})$, respectively, to $\D_{\U^H}^{\gamma,\eta}(\Gamma^{e_h})$. In order to be able to apply the Implicit Function theorem, the last ingredient we miss is to prove that $D_2 F_\gamma(0,U)$ is a linear and bounded isomorphism. Linearity is obvious and so is boundedness, indeed, thanks to Propositions 6.16, 6.12 and 6.13 in~\cite{Hai}, we have
\[
\VERT D_2 F_\gamma(0,U)(Y)\VERT_{\gamma,\eta}\lesssim \VERT Y\VERT_{\gamma,\eta}+\VERT G_\gamma'(U)Y {\color{blue}\Xi}\VERT_{\gamma-\alpha_{\min},\eta-\alpha_{\min}}\lesssim\VERT Y\VERT_{\gamma,\eta}
\] 
where the neglected constants depend on the same parameters as before. Concerning invertibility, it suffices to show that for every $W\in\D_{\U^H}^{\gamma,\eta}(\Gamma^{e_h})$ there exists a unique $Y\in\D_{\U^H}^{\gamma,\eta}(\Gamma^{e_h})$ such that
\[
Y=W+\PK^{e_h}\Big(G_\gamma'(U)Y {\color{blue}\Xi}\Big)
\]
and this can be achieved by a fixed point argument in the spirit of Proposition~\ref{prop:solDergPAM}. 

At this point, all the assumptions of Theorem 19.28 in~\cite{BD} are matched and we conclude that there exist $\tilde{\delta}<\bar{\delta}$ and a differentiable function $\vartheta:(-\tilde{\delta},\tilde{\delta})\to\D_{\U^H}^{\gamma,\eta}(\Gamma^{e_h})$ such that $\vartheta(0)=U$, $(\delta,\vartheta(\delta))\in\X$ and $F_\gamma(\delta,\vartheta(\delta))=0$ for all $\delta\in(-\tilde{\delta},\tilde{\delta})$. Moreover, we have
\begin{equation}\label{eq:Dertheta}
\vartheta'(\delta)=-\Big(D_2F_\gamma(\delta,\vartheta(\delta))\Big)^{-1}D_1F_\gamma(\delta,\vartheta(\delta)),\qquad \text{for all}\,\,\delta\in(-\tilde{\delta},\tilde{\delta})
\end{equation}
But now, notice that since $F_\gamma(\delta,\vartheta(\delta))=0$, by definition of $F_\gamma$ it follows that $\vartheta(\delta)$ is the, necessarily unique, solution to~\eqref{eq:gPAMhabs}, with $\delta{\color{blue}H}$ substituting ${\color{blue}H}$, in $\D^{\gamma,\eta}(\Gamma^{e_h})$, i.e. $U^H_\delta$. Moreover, thanks to~\eqref{eq:Dertheta}, we can also conclude that $\vartheta'(0)\in\D_{\U^H}^{\gamma,\eta}(\Gamma^{e_h})$ solves
\[
\PK^{e_h}\Big(G_\gamma(U){\color{blue}H}\Big)=\vartheta'(0)-\PK^{e_h}\Big(G_\gamma'(U)\vartheta'(0) {\color{blue}\Xi}\Big)
\]
and, by the uniqueness part of Proposition~\ref{prop:solDergPAM}, it must coincide with $V^h$. Collecting the observations carried out so far, we finally obtain
\[
\VERT U^H_\delta-U-\delta V^h\VERT_{\gamma,\eta}=\VERT\vartheta(\delta)-\vartheta(0)-\delta\vartheta'(0)\VERT_{\gamma,\eta}=o(\delta)
\]
At this point, in the notations introduced before, since $\omega\in\Upsilon_t$, we have 
\begin{align*}
|u(t,x;\omega+\delta h)-&u(t,x;\omega)-\delta v^h(t,x;\omega)|\\
&=|\RS(\Sol(u_0,\hat{Z}(\omega+\delta h)))(t,x)-\RS(\Sol(u_0,\hat{Z}(\omega)))(t,x)-\delta\RS^{e_h}(\Sol^D(h,\hat{Z}(\omega)))(t,x)|\\
&=|\RS^{\delta h}(\Sol(u_0,T_{\delta h}\hat{Z}(\omega)))(t,x)-\RS^{e_h}(\Sol(u_0,\hat{Z}(\omega)))(t,x)-\delta\RS^{e_h}(\Sol^D(h,\hat{Z}(\omega)))(t,x)|\\
&=|\RS^{e_h}U^H_\delta(\omega)(t,x)-\RS^{e_h}U(\omega)(t,x)-\delta\RS^{e_h}V^h(\omega)(t,x)|
\end{align*}
where the third equality follows by Lemma~\ref{lemma:consistency2} choosing $f=0$. Now we can bound the right-hand side of the previous by its $\CC^{\alpha_{\min}+2}$ norm, which, thanks to Propositions 3.28 and 6.9 in~\cite{Hai} satisfies
\[
\|\RS^{\delta h}(U^{\delta h}(\omega))-\RS(U(\omega))-\delta\RS^{e_h}(V^h(\omega))\|_{\CC^{\alpha_{\min}+2,\eta}}\lesssim \VERT U^H_\delta-U-\delta V^h\VERT_{\gamma,\eta}
\]
Thanks to~\eqref{eq:GatDiff} and Theorem~\ref{thm:solDergPAM}, which guarantees the linearity and continuity of $v^h(t,x;\omega)$ in $h$, the conclusion immediately follows. 

\end{proof}

\begin{remark}
The idea of using the Implicit Function theorem in order to prove the differentiability of the solution map is not new. In the context of SDEs, see~\cite{NS}, while for SPDEs, and in particular for the fractional heat equation driven by a fractional Brownian motion with Hurst parameter $H>\frac{1}{2}$, see~\cite{DT}. 
\end{remark}

We are now ready to state the main result of this section and complete the proof of the Malliavin differentiability of the solution map.

\begin{theorem}\label{thm:MallDiff}
Let $\hat{Z}\in\M(\TS_g)$ be the admissible model obtained in Theorem 10.19 of~\cite{Hai}. Then, for fixed $(t,x)\in(0,+\infty)\times\R^2$, the random variable
\[
\omega\mapsto u(t,x;\omega)=\RS\big(\Sol(u_0,\hat{Z}(\omega)\big)(t,x)
\]
is locally $\Hi$-differentiable according to Definition~\ref{def:MallDiff}, on $\Upsilon_t$ and its derivative is given by $v^h=\langle Du,h\rangle_\Hi$. 
\end{theorem}
\begin{proof} 
We have already proved that for any $\omega\in\Upsilon_t$, on one side by Proposition~\ref{prop:differentiabilityabs}, the map $h\mapsto u(t,x;\omega+h)$ is Gateaux-differentiable at $h=0$, on the other thanks to the lower-semicontinuity of $T_\infty$, there exists $q(\omega)>0$ such that for every $f\in B_{2}(0,q(\omega))$, $\omega+ f\in\Upsilon_t$, where $B_2(0,q(\omega))$ is the ball centered at $0$ of radius $q(\omega)$ in $L^2(\T^2)$. Let us fix $\omega\in\Upsilon_t$ and the corresponding $B_{2}(0,q(\omega))$. We will now show that $f\mapsto u(t,x;\omega+f)$ is Gateaux differentiable on $B_{2}(0,q(\omega))$ and that the Gateaux differential is continuous. 

Let us begin with the first. Consider $f\in B_{2}(0,q(\omega))$, set $z=(t,x)$ and notice that, for $\delta$ small enough, 
\begin{align*}
u(z;\omega+f+&\delta h)-u(z;\omega+f)-\delta v^h(z;\omega+f)\\
&=\RS(\Sol(u_0,\hat{Z}(\omega+f+\delta h)))(z)-\RS(\Sol(u_0,\hat{Z}(\omega+f)))(z)-\delta\RS^{e_h}(\Sol^D(h,\hat{Z}(\omega+f)))(z)\\
&=\RS_f(\Sol(u_0,T_f\hat{Z}(\omega+\delta h)))(z)-\RS_f(\Sol(u_0,T_f\hat{Z}(\omega)))(z)-\delta\RS_f^{e_h}(\Sol^D(h,T_f\hat{Z}(\omega)))(z)\\
&=\RS_f^{e_h}\big(U_\delta^H(\omega)\big)(z)-\RS_f^{e_h}\big(U(\omega)\big)(z)-\delta\RS_f^{e_h}\big(V^h(\omega)\big)(z)
\end{align*}
where the previous passages are justified by the facts that both $\omega$ and $\omega+f\in\Upsilon_t$, and Lemma~\ref{lemma:consistency2}. At this point we can argue as in Proposition~\ref{prop:differentiabilityabs}, i.e. applying Propositions 3.28 and 6.9 in~\cite{Hai} and conclude via Proposition~\ref{prop:differentiabilityabs}. Indeed,~\eqref{eq:GatDiff} holds for \textit{any} admissible model on $\TS_g$ and, by Proposition~\ref{prop:TransModel}, $T_f\hat{Z}$ is indeed one. 

For the second part, notice that the Gateaux differential is given by 
\[
B_{2}(0,q(\omega))\ni f\longmapsto v^\cdot(t,x;\omega+f)=\langle Du,\cdot\rangle_\Hi(t,x;\omega+f)\in\mathcal{L}(\Hi,\R)
\]
where $\mathcal{L}(\Hi,\R)$ is the set of linear bounded operator from $\Hi$ to $\R$, that, thanks to Riesz representation theorem, can be identified by $\Hi$ itself. Let $f\in B_2(0,q(\omega))$ and $f_n\in B_2(0,q(\omega))$ be a sequence  converging to $f$ in $\Hi$, then, thanks to Theorem~\ref{thm:solDergPAM} and Proposition~\ref{prop:ContShiftModel},
\begin{multline*}
\big\|Du(t,x;\omega+f_n)-Du(t,x;\omega+f)\big\|_{\Hi}=\sup_{\|h\|_{L^2}=1}\big|v^h(t,x;\omega+f_n)-v^h(t,x;\omega+f)\big|\\
\lesssim\sup_{h:\|h\|_{L^2}=1}\left\|\RS^{e_h}\Sol^D(h,\hat{Z}(\omega+f_n))-\RS^{e_h}\Sol^D(h,\hat{Z}(\omega+f))\right\|_{\CC^{\alpha_{\min}+2,\eta}}\\
=\sup_{h:\|h\|_{L^2}=1}\left\|\RS_{f_n}^{e_h}\Sol^D(h,T_{f_n}\hat{Z}(\omega))-\RS_f^{e_h}\Sol^D(h,T_f\hat{Z}(\omega))\right\|_{\CC^{\alpha_{\min}+2,\eta}}
\end{multline*}
and now, thanks to the local Lipschitz continuity of the reconstruction, $\RS$, the extension and translation operators, $E$ and $T$, and the solution map for the abstract tangent equation, $\Sol^D$, we can conclude that the last term converges to $0$ as $n$ tends to $\infty$ uniformly over $\|h\|_{L^2}\leq 1$, which in turn completes the proof.

\end{proof}

\begin{remark}
With the help of the bounds obtained in Proposition \ref{prop:differentiabilityabs}, one could actually obtain a stronger statement, namely that $u$ is (locally) $\mathcal{H}$-differentiable on $\Upsilon_t$ as a $\CC^{2\alpha_{\min} +2, \eta}([0,t]\times \mathbb{T}^2)$-valued random variable.
 \end{remark}

\begin{remark}
We point out that through the arguments in the present section it is in principle possible to obtain higher order (local) $\Hi$-differentiability of the solution map. 
\end{remark}

\section{Existence of density for the value at a fixed point}\label{sec:Dens}

The results of the previous section guarantee that, for $(t,x)\in\R^+\times\T^2$, the solution $u(t,x;\omega)$ of gPAM determined in Theorem~\ref{thm:solgPAM}, is Malliavin differentiable at least on those points in which it does not explode, namely when $\omega\in\{t<T_\infty\}$. We now want to show that, as a random variable, conditioned on the previous set, it admits a density with respect to the Lebesgue measure. To this purpose we aim at exploiting the Bouleau and Hirsch's criterion whose application has though to be carefully handled. Indeed, if on one side one has to prove non-degeneracy of the Malliavin derivative, which is \textit{per se} everything but obvious, on the other we have an extra difficulty, coming from the fact that $u(t,x;\omega)$ is only locally $\Hi$-differentiable on $\{t<T_\infty\}$ and the latter does not have \textit{a priori} full measure. We will deal with these two issues separately. For the first, we will derive a strong maximum principle for a rather general class of linear parabolic PDEs, which will prove to be extremely useful in our context but whose interest goes way beyond it. For the second, we will suitably approximate (in two different ways) our solution with $\CC^1_{\Hi}$ random variables matching the assumptions of Theorem~\ref{thm:BH}. 

\subsection{A Mueller-type strong maximum principle}\label{subsection:Pos}

As a motivation for the following proposition, consider the homogenous version of the renormalized tangent equation~\eqref{eq:renDergPAM}, that is 
\begin{equation}\label{eq:renDergPAMhom}
\partial_t \tilde{v}_\eps^{\hom}=\Delta \tilde{v}_\eps^{\hom}+\tilde{v}_\eps^{\hom}\Big(g'(\tilde u_\eps)\xi_\eps-C\Big( (g'(\tilde u_\eps))^2+g''(\tilde u_\eps)g(\tilde u_\eps)\Big)\Big),\quad \tilde{v}_\eps^{\hom}(0,\cdot)=v_0^{\hom}(\cdot)
\end{equation}
Remark that, given fixed initial data $v_0^{\hom}$, using the same techniques as in Proposition~\ref{prop:solDergPAM} and Theorem~\ref{thm:solDergPAM}, it is possible to show that $ \tilde{v}_\eps^{\hom}$ converges (locally uniformly) in probability to some limit $v^{\hom}$,  given as reconstruction of the abstract solution, with respect to the model $\hat{Z}$ (see Theorem~\ref{thm:solgPAM}), to 
\begin{equation}\label{eq:linear}
V^{\hom} = \PK\left(\tilde{\Xi}V^{\hom} \right) + \mathscr{T}_\gamma \bar{K} v^{\hom}_0
\end{equation}
where $\tilde{\Xi}$ is a suitable modelled distribution (in the previous case, $G_\gamma'(U){\color{blue}\Xi}$). As a consequence of the (weak) maximum principle for the approximate equations, cf. Section~\ref{subsec:WMP}, it clearly holds that $v^{\hom} =  \mathcal{R}V^{\hom} \ge 0$ for initial data $v^{\hom}_0 \ge 0$. In the next proposition, we show that this latter property is all we need in order to guarantee that the reconstruction of the solution to an equation of the form~\eqref{eq:linear}, satisfies a strong maximum principle. 

\begin{theorem} \label{thm:Positive}
Let $\alpha_{\min}\in(-\frac{4}{3},-1)$, $\gamma\in(|\alpha_{\min}|,\frac{4}{3})$, $\eta\in[0,\alpha_{\min}+2)$, $Z=(\Pi,\Gamma)$ be an admissible model on $\TS_g$ and $T>0$. Given $\tilde{\Xi} \in \mathcal{D}^{\gamma,\eta}(\Gamma)$ on $[0,T]\times\T^2$, and $v^{\hom}_0 \in C^\eta$, consider the abstract fixed point equation~\eqref{eq:linear} and let $V^{\hom}\in\D^{\gamma,\eta}(\Gamma)$ be its solution. Assume a weak maximum principle of the form $v^{\hom} := \RS V^{\hom} \ge 0$ on $[0,T]\times\T^2$ whenever $v^{\hom}_0 \ge 0$. Then a strong maximum principle holds,
in the sense that if $v^{\hom}_0 \ge 0$ but not identically equal to $0$, then $v^{\hom}$ is strictly positive at times $t \in(0,T]$.
\end{theorem}

\begin{notation*}
We now introduce a notation that will be exploited \textit{only} in the following proof. For $\gamma$, $\eta\in\R$ and $t>0$, we will write $\VERT\cdot\VERT_{\gamma,\eta;t}$ for the usual norm on the space of symmetric modelled distributions (see Remark~\ref{remark:periodicMD}), $\D^{\gamma,\eta}$, but where the supremum in~\eqref{eq:ModDistNorm} is taken over $(0,t]\times\T^2$. 
\end{notation*}

\begin{proof}
W.l.o.g. we will take $T=1$. As pointed out before, by the very same arguments exploited in the proof of Proposition \ref{prop:solDergPAM}, we know that \eqref{eq:linear} admits a unique solution in $V^{\hom}\in\D^{\gamma,\eta}$, which, by linearity, satisfies 
\[
\VERT V^{\hom}\VERT_{\gamma,\eta;1}\leq C \|v^{\hom}_0\|_\eta
\]
where $C$ is a constant depending continuously on $\VERT Z\VERT_\gamma$ and $\VERT\tilde{\Xi}\VERT_{\gamma,\eta;1}$. We now set $W = \PK(\tilde{\Xi}V^{\hom} )$ and $w=\RS W$. Then, by Proposition 6.12 in~\cite{Hai}, $\tilde{\Xi}V^{\hom}\in\D^{\gamma+\alpha_{\min},\eta+\alpha_{\min}}$, hence thanks to~\eqref{bound:Schauder}, there exists $\theta>0$ and a constant $C_w>0$ such that
\[
\VERT W\VERT_{\gamma,\eta;t} \leq C_w  t^\theta,\qquad\text{for all}\,\,t\in(0,1],
\]
where, this time, $C_w$ depends continuously on the norms $\VERT Z\VERT$, $\VERT\tilde{\Xi}\VERT_{\gamma,\eta;1}$ and $\|v^{\hom}_0\|_\eta$. Since, by definition and Proposition 3.28 in~\cite{Hai}, $W(t,x) = w(t,x){\color{blue}1} +...$, omitting terms of strictly positive homogeneity, it is clear from~\eqref{eq:ModDistNorm}, 
that a bound analogous to the previous holds for $w$, namely
\[
\left|w(t,x)\right|\leq C_w t^{\theta},\qquad\text{for all}\,\,t\in(0,1]\,\,\text{and}\,\,x\in\T^2.
\]
Now fix $\delta >0$, and assume $u_0$ non negative and $v^{\hom}_0\geq 1$ on $B(x,\delta)\subset\R^2$, the ball of radius $\delta$ centered at $x$.
We first claim that by properties of the heat kernel, (see proof below) for each $\rho>0$, there exists $t_\rho>0$ s.t.
\begin{equation} \label{eq:1}
v^{\hom}_0 \geq 1\quad \mbox{ on }B(x,\delta) \;\Longrightarrow\; (\bar{K}v^{\hom}_0)(t,\cdot) \geq \frac{1}{4} \quad \mbox{ on }\,\,B(x,\delta+t\rho), \qquad\text{for all}\,\, t \leq t_\rho.
\end{equation}
Upon taking $h \leq t_\rho$ small enough so that $C_w h^{\theta} \leq \frac{1}{8}$, one has
$$v^{\hom}_0 \geq 1 \mbox{ on }B(x,\delta) \;\Longrightarrow\; v(h,\cdot) \geq \frac{1}{8} \quad\mbox{ on }\,B(x,\delta+h\rho).$$
One can then propagate the bound using linearity of the equation (and consequently of $w$ itself) to obtain that $v(1,\cdot) \geq (\frac{1}{8})^{1/h} >0$ on $B(x,\rho)$. Since $\rho$ was arbitrary this proves the claim. Also note that here we strongly use the fact that, by construction, we can always take the same value of $C_w$ when we iterate the argument over different time-steps.

We now turn to the proof of the claim~\eqref{eq:1}. W.l.o.g. take $x=0$, and consider a generic point $y \in B(0,\delta+t\rho)$, written as $y = (\delta + t \rho) u$, where $|u| \le1$. Then for $Z$ standard ($d$-dimensional) Gaussian
\begin{eqnarray*}
\bar{K}(1_{B(0,\delta)})(t,y) = P(|\sqrt{t}Z +y|\leq \delta) = P(Z \in B(\frac{y}{\sqrt{t}},\frac{\delta}{\sqrt{t}})) 
\end{eqnarray*}
Then note that $B(\frac{y}{\sqrt{t}},\frac{\delta}{\sqrt{t}})$ is a ball with radius going to $\infty$ as $t \to 0$ and containing the point $\sqrt{t} u$ (as its closest point from origin). In particular when $t\to 0$, it eventually contains all points in a half-space, so that one gets $\frac{1}{2}$ in the limit, hence the proof of the claim (and consequently of the proposition) is concluded. 
\end{proof}

\begin{remark}
Even if the previous proposition was formulated in the specific context under study, its proof has very little to do with the specifics of our regularity structure, hence the same argument can be straightforwardly applied to directly get a strong maximum principle (or equivalently, strict positivity of solutions) for any linear heat equation for which the theory applies. In particular it holds for the linear multiplicative stochastic heat equation in dimension $d=1$ (cf. \cite{HP}) where we recover Mueller's work, \cite{Mueller}, and to the linear PAM equation in dimensions $d=2,3$ for which the result appears to be new.
\end{remark}

\subsection{Density for value at a fixed point}

Let us fix $t>0$, $x \in \T^2$, and consider the random variable $F = u(t,x) \1_{\{t<T_\infty\}}$. We will show that its restriction to $\{t<T_\infty\}$ admits a density w.r.t. the Lebesgue measure, but first we need a technical lemma to approximate $\1_{\{t<T_\infty\}}$ by a sequence of $\mathcal{H}$-differentiable random variables.

\begin{lemma} \label{lem:Xn}
Fix $t\geq 0$. Then there exists a sequence $(X_n)_{n\geq 0}$ with $X_n F \in C^1_{\mathcal{H}-loc}$ such that $$X_n \leq \1_{\{t<T_\infty\}} \mbox{ and } \cup_{n \geq 0} \left\{X_n = 1, DX_n = 0\right\} = \left\{t<T_\infty\right\} \mbox{(up to a $\mathbb{P}$-null set)}.$$
\end{lemma}

\begin{proof} See Appendix~\ref{app:comp}.
\end{proof}

\begin{theorem}\label{thm:density}
In the setting of Theorem~\ref{thm:solgPAM}, let $u$ be the limit of the solutions $u_\eps$ to the renormalized gPAM equation. Assume furthermore that $g\geq 0$, and $g(u_0)$ is not identically $0$. Then for any $t>0$ and $x \in \T^2$, the law of $u(t,x)$ conditionally on $\{t<T_\infty\}$ is absolutely continuous with respect to Lebesgue measure.
\end{theorem}

\begin{proof}
First note that on $\{t<T_\infty\}$, one has $\|DF\|\neq 0$. Indeed, according to Section \ref{sec:Mall} it suffices to find one $h\in\Hi$ such that $v^h(t,x) \neq 0$, where $v^h(t,x)$ is the derivative of $u(t,x)$ with respect to the noise determined in Theorems~\ref{thm:solDergPAM} and~\ref{thm:MallDiff}. Hence, we restrict $h\in\Hi\cap\CC^{\eta}$, for $\eta\geq 0$. 

In the setting of Theorem~\ref{thm:Positive}, we define the 2-parameter semigroup $P_{0,t}$ as
\[
P_{0,t}[f]:=\RS V^{\hom}(t,\cdot)
\]
where $V^{\hom}$ is the solution to~\eqref{eq:linear}, with time-$0$ initial condition $f\in\CC^\eta$, with respect to the model $\hat{Z}$ (see Theorem~\ref{thm:solgPAM}) and similarly $P_{s,t}$ when starting at times $s\leq t$. Upon choosing $\tilde{\Xi}=G_\gamma'(U_s){\color{blue}\Xi}$, where $U_s(r,x)=U(s+r,x)$ and $U$ is the solution to the abstract counterpart of gPAM, i.e. equation~\eqref{eq:gPAMabs}, with respect to $\hat{Z}$, one can show that
\[
v^h(t,\cdot)=\int_0^t P_{s,t}[g(u_s)h]\dd s+ P_{0,t} [v^h (0, \cdot)]  = \int_0^t P_{s,t}[g(u_s)h]\dd s
\] 
with $u_s:=\RS U_s$ and noting that $v^h (0, \cdot) \equiv 0$. Indeed, the previous representation is obvious for $\tilde{v}^h_\eps$ (defined as in Theorem~\ref{thm:solgPAM}) and one can pass to the limit since the convergence of the left hand side is guaranteed by Theorem~\ref{thm:solDergPAM} while the one of the right hand side follows by the continuity of the 2-parameter semigroup in its argument, uniformly over $s\leq t$. But now we can choose $h$ such that $g(u) h$ is nonnegative and not everywhere $0$ (actually $h=1$ suffices). Since, by Proposition~\ref{prop:WMA} we already know that the homogeneous equation satisfies a weak maximum principle, Theorem \ref{thm:Positive} implies that $P_{s,t}[g(u_s)h]>0$ for $s$ in a set of positive measure.

Now, let $X_n$ be the sequence of random variables defined in Lemma \ref{lem:Xn}, then, for every set $E\subset \R$ of Lebesgue measure $0$,
\begin{align*}
 \mathbb{P}\left( F \in E,t< T_\infty \right) \leq \sum_{n=0}^{\infty}  \mathbb{P}\left( F X_n  \in E, X_n=1, DX_n=0  \right) \leq \sum_{n \geq 0}  \mathbb{P}\left( FX_n \in E, \|D(FX_n)\| \neq 0 \right) = 0.
\end{align*}
where the last equality follows by Theorem~\ref{thm:BH}, and the proof is concluded.
\end{proof}

\begin{remark}
In exactly the same way, one can show that for all $t>0$ and measure $\mu$ supported in $(0,t] \times \T^2$, the law of $\int u d\mu$ conditionally on $\{ t < T_\infty\}$ admits a density.
\end{remark}

\begin{remark} The reader should note that the proof of Theorem~\ref{thm:density}, relying on Lemma~\ref{lem:Xn}, uses indeed few specific properties of gPAM. In particular, it should be possible to adapt the argument
here to other singular PDEs, for which only local existence results are available, which in turn underlies the importance of Lemma~\ref{lem:Xn}. That said, in Section ~\ref{subsec:GEgPAM}, we gave a global existence condition specific to the structure of gPAM,  which allows for the  following alternative (but specific to gPAM) argument.
Take $g_n$ be (sufficiently) smooth, compactly supported in $[-n-1,n+1]$ and such that $g_n\equiv g$ on $[-n,n]$, where $g$ satisfies the assumptions of Theorem~\ref{thm:density}. As a consequence of Proposition~\ref{prop:GEgPAM}, we know that the solution $u^n$ to (g$_n$PAM) is globally well-posed in time, hence the results in Section~\ref{sec:MallDiff} directly imply that, for every $(t,x)\in(0,+\infty)\times\T^2$, $u^n(t,x)\in\CC^1_\Hi$ according to Definition~\ref{def:MallDiff} (see Remark~\ref{rmk:MallDiff}), which in turn guarantees that $u^n(t,x)\in\DD^{1,2}_\loc$ (see Proposition 4.1.3 in~\cite{Nua}). Therefore, the first part of the proof of Theorem~\ref{thm:density} implies that the assumptions of the Bouleau and Hirsch's criterion, Theorem~\ref{thm:BH}, are satisfied and $u^n(t,x)$ has a density with respect to the Lebesgue measure.

Now, let $u$ the solution to (gPAM), $T_\infty$ its explosion time as defined in Proposition~\ref{prop:solgPAM} and $F=u(t,x)\1_{\{t<T_\infty\}}$. Then, it is immediate to verify (e.g. looking at the approximating equations) that $u^n\equiv u$ on $\{|u|< n \}$ and consequently $\{t<T_\infty\}\subset \bigcup_n \{|u(t,x)|\leq n\}$. Therefore, for every $E\subset \R$ of Lebesgue measure $0$ we have
\[
\PR\left( F \in E,t< T_\infty \right) \leq \sum_{n=0}^{\infty}  \PR\left( u(t,x)\in E, |u(t,x)|\leq n\right) \leq \sum_{n \geq 0}  \PR\left( u^n(t,x)\in E \right) = 0.
\]
which concludes the argument.
\newline
\end{remark}

\appendix

\section{Wavelets and Translation}\label{app:comp}

Let us introduce a few notations. We always work on $\R^d$ (in fact, we will only need $d=2$). For any $n \in \N$ we let $\Lambda_n = \left\{(2^{-n}k_1,\ldots, 2^{- n}k_d), \;k=(k_1,\ldots,k_d)\in \Z^d\right\}$.  Given a function $\varphi$ and $x \in \Lambda_n$, we denote $$\varphi^n_x = 2^{n \frac{d}{2}} \varphi\left(2^{-n} \left(\cdot-x\right)\right)$$
(the rescaling is such that $L^2$-norm is preserved. We also fix a real number $r>0$ (which we will take large enough later). Wavelet analysis \cite{Dau,Meyer} then provides us with a function $\varphi$ and a finite set $\Psi = \left\{ \psi \in \Psi\right\}$ such that :
\begin{itemize}
\item $\varphi$, and all $\psi$ $\in$ $\Psi$ are in $\mathcal{C}^r$ and have compact support,
\item all $\psi$ $\in$ $\Psi$ have vanishing moments up of order $\left\lfloor r\right\rfloor$,
\item For each $n \geq 0$, the family
$$ \left\{\varphi^n_x, \; x \in \Lambda_n\right\} \cup \left\{ \psi^m_y, \; m \geq n, y \in \Lambda_m, \psi \in \Psi\right\}$$
is an ortonormal basis of $L^2(\mathbb{R}^d)$.
\end{itemize}
Let us remark that to save space we will often omit the summation over $\Psi$ and write $\psi$ for any element of $\Psi$, so that for instance we will write $\sum_{x \in \Lambda_m} F(\psi^m_x)$ for $\sum_{\psi\in\Psi} \sum_{x \in \Lambda_m}F(\psi^m_x)$

Let us also recall that for $\beta$ $\in (-r,r)$, it is well-known (e.g. \cite{Meyer}) that one can define the usual fractional Sobolev spaces $H^\beta$ via a norm on the wavelet coefficients. Since we will work only with functions on $\T^d$ (identified with $1$-periodic functions on $\R^d$), these norms can be written as
$$\left\|f\right\|_{H^\beta(\T^d)}^2 := \sum_{x \in \Lambda_0 \cap D} \left\langle f, \varphi^0_x\right\rangle^2 + \sum_{m\geq 0} \sum_{y \in \Lambda_m\cap D} \left\langle f, \psi^m_y\right\rangle^2 2^{2m\beta},$$
where $D$ is a large enough compact subset of $\R^d$.

\begin{lemma} \label{lem:WavDecorr}
For all $n\leq m \leq p$ and $x,y,z$ $\in$ $\Lambda_n \times \Lambda_m \times \Lambda_p$, one has
\begin{equation}
\left|\left\langle \psi^n_x \psi^m_y , \psi^p_z \right\rangle \right|:= \left| \int_{\R^d} \psi^n_x  \psi^m_y  \psi^p_z \right| \; \lesssim \; 2^{\frac{n d}{2}} 2^{-r'(p-m)},
\label{eq:lemWav}
\end{equation}
where $r'=\lfloor \frac{r}{2}\rfloor + 1 + \frac{d}{2}$.
The same inequality holds if $\psi^n_x$ or $\psi^m_y$ are replaced by $\varphi^n_x$, $\varphi^m_y$.
\end{lemma}

\begin{proof} By scaling it is enough to consider the case $n=0$. 
Then for any polynomials $P,Q$ with $PQ$ of degree less than $r$,
\begin{eqnarray*}
\left\langle \psi^0_x  \psi^m_y , \psi^p_z \right\rangle &=& \left\langle (\psi^0_x - P ) \psi^m_y , \psi^p_z \right\rangle  + \left\langle P  (\psi^m_y - Q), \psi^p_z \right\rangle + \left\langle P Q, \psi^p_z \right\rangle.
\end{eqnarray*}
The last term is equal to $0$ by the properties of $\psi$. Now taking for $P$ (resp. $Q$) the Taylor expansion of order $k=\lfloor \frac{r}{2}\rfloor$ for $\psi^0_x$ (resp. $\psi^m_y$) at $z$ , we have (denoting $I_p$ the support of $\psi^p_z$)
\begin{eqnarray*}
\left\langle P ( \psi^m_y - Q) , \psi^p_z \right\rangle &\leq& \left\|P 1_{I_p}\right\|_\infty \left\|(\psi^m_y - Q)1_{I_p}\right\|_2 \left\|\psi^p_z\right\|_2 \\
&\lesssim& \left\|(\psi^m_y - Q)1_{I_p}\right\|_\infty \left\|1_{I_p}\right\|_2 \;\lesssim\; \|\psi^m_y\|_{C^{k+1}} {\rm diam}(I_p)^{k+1} 2^{- \frac{pd}{2}} \\
&\lesssim& 2^{(k+1+\frac{d}{2}) (m-p)},
\end{eqnarray*}
and similarly $\left\langle (\psi^0_x - P) \psi^m_y , \psi^p_z \right\rangle$ $\lesssim$ $2^{ \frac{md}{2}}2^{-(k+1+\frac{d}{2})p}$.
\end{proof}

\begin{lemma} \label{lem:newlemApp}
Let $Z=(\Pi,\Gamma)\in\M(\TS_g)$ and $Z^{e_h}=(\Pi^{e_h},\Gamma^{e_h})$ be the extension of $Z$ on $\TS_g^H$ defined in proof of Proposition~\ref{prop:ContShiftModel}. Then, for every $\tau\in \W^H$
\begin{equation}\label{bound:App}
 \left\langle \Pi_x^{e_h} \tau, \varphi^\lambda_x \right\rangle \lesssim \lambda^{|\tau|}
\end{equation}
locally uniformly over $x\in\R^3$ and uniformly over $\varphi\in\B_1^2$. 
\end{lemma}

\begin{proof} 
We begin with two important observations. Since $\xi$ and $h$ only depend on the space coordinate $x$, we can simply remove the time coordinate. Indeed, it is immediate to check that for each symbol $\tau \in T_g^H$, $\Pi^{e_h}_{(t,x)} \tau (s,y)= \Pi^{e_h}_x \tau(y)$ does not depend neither on $t$ nor on $s$ (see also Section 10.4 in~\cite{Hai}). Hence, we will take $\varphi\in\B_1^2$ taking values in $\R^2$ and rescaled as $\varphi^\lambda_x(y) = \lambda^{-2} \varphi(\lambda^{-1}(y-x))$. Moreover, we are in the setting of Remark~\ref{remark:periodicM}, in other words, we are considering only models adapted to the action of translation, which means that the ``locally uniformly in $x$" appearing in the statement can be replaced by ``for all $x\in\T^2$". 

Now, for those elements $\tau\in\W^H\setminus\W$ we have nothing to prove since, by construction, $\Pi^{e_h}_x\tau=\Pi_x\tau$, and we know that $\Pi$ satisfies the correct analytical bounds ($Z\in\M(\TS_g)$). On the other hand, to ensure that~\eqref{bound:App} holds for all $\tau \in \left\{{\color{blue}H}, {\color{blue}\mathcal{I}(\Xi)H},  {\color{blue}\mathcal{I} (H)\Xi},  {\color{blue}\mathcal{I}(H)H}\right\}$, it suffices to check it on the wavelet basis functions $\varphi^0_x$, $\psi^n_y$, i.e.
\begin{equation} \label{eq:BndPih}
\left\langle \Pi^{e_h}_x \tau, \varphi^0_x \right\rangle \lesssim 1,  \;\;\;\left\langle \Pi_y^{e_h} \tau, \psi^n_y \right\rangle \lesssim 2^{-n(|\tau| + \frac{d}{2})}, \;\;\; x \in \Lambda_0,  \psi \in \Psi, n \geq 0, y \in \Lambda_n.
\end{equation}
The bound \eqref{eq:BndPih} for $\tau={\color{blue}H}$ is immediate by Sobolev embedding, indeed, $h \in L^2 \subset C^{\alpha}$ since $\alpha \leq - \frac{d}{2}$.

We now focus on the symbol ${\color{blue}\mathcal{I}(\Xi)H}$. Let us write 
\begin{eqnarray*}
\left\langle {\Pi}^{e_h}_x{\color{blue}\mathcal{I}(\Xi)H}, \psi^n_x \right \rangle &=& \sum_{y,z \in \Lambda_n} \left\langle \Pi_x(\mathcal{I}(\Xi)), \varphi^n_y \right \rangle \left \langle h, \varphi^n_z \right \rangle  \left\langle \varphi^n_y \varphi^n_z, \psi^n_x \right \rangle \\
&&+ \sum_{y \in \Lambda_n} \sum_{m \geq n, z \in \Lambda_m}  \left\langle \Pi_x{\color{blue}\mathcal{I}(\Xi)}, \varphi^n_y \right \rangle \left \langle  h, \psi^m_z \right \rangle \left \langle  \varphi^n_y \psi^m_z, \psi^n_x \right \rangle \\
&& + \sum_{z \in \Lambda_n} \sum_{p \geq n, y \in \Lambda_p}  \left\langle h, \varphi^n_z \right \rangle \left \langle \Pi_x{\color{blue}\mathcal{I}(\Xi)}, \psi^p_y \right \rangle \left \langle  \varphi^n_z \psi^p_y, \psi^n_x \right \rangle \\
&& + \sum_{p \geq n, y \in \Lambda_p} \sum_{m \geq n, z \in \Lambda_m}  \left\langle \Pi_x{\color{blue}\mathcal{I}(\Xi)}, \psi^p_y \right \rangle \left \langle h, \psi^m_z \right \rangle \left \langle \psi^p_y \psi^m_z, \psi^n_x \right \rangle \\
&=&: S_1 + S_2 + S_3 + S_4.
\end{eqnarray*}
Before focusing on each of the terms above, recall that, since $\Pi$ is a model, $\left\langle \Pi_x{\color{blue}\mathcal{I}(\Xi)}, \psi^m_y \right \rangle$ admits the two following bounds
\begin{align}
&\left\langle \Pi_x{\color{blue}\mathcal{I}(\Xi)}, \psi^p_y \right \rangle = \left\langle \Pi_y \Gamma_{yx}{\color{blue}\mathcal{I}(\Xi)}, \psi^p_y \right \rangle \lesssim 2^{-\frac{pd}{2}} \sum_{\beta<2+\alpha, \beta \in A} |x-y|^{2+\alpha-\beta} 2^{-p\beta}, \label{eq:BndAlpha}\\
&\left\langle \Pi_x{\color{blue}\mathcal{I}(\Xi)}, \psi^p_y \right \rangle =\left\langle \Pi_y \Gamma_{yz}\Gamma_{zx}{\color{blue}\mathcal{I}(\Xi)}, \psi^p_y \right \rangle \lesssim 2^{-\frac{pd}{2}} \sum_{\gamma< \beta< \alpha+2}2^{-p\beta}|y-z|^{2+\alpha-\beta} |z-x|^{\beta-\gamma}\label{eq:BndAlpha1}
\end{align}
Moreover, since the wavelet basis functions form an orthonormal basis of $L^2$ and $h\in L^2(\T^2)$ one has
\begin{equation}
\sum_{y \in \Lambda_n \cap D} \left\langle h, \varphi^n_y \right \rangle^2 + \sum_{m \geq n, z \in \Lambda_m \cap D} \left\langle h, \psi^m_z \right \rangle^2 \lesssim \|h\|^2,
\label{eq:Bnd-h}
\end{equation}
uniformly over all compact sets $D$ of diameter less than a fixed constant.

We are now ready to show that $S_1$, $S_2$, $S_3$ and $S_4$ satisfy the correct bounds. 
\[
S_1 =\sum_{y,z \in \Lambda_n} \left\langle \Pi_x{\color{blue}\mathcal{I}(\Xi)}, \varphi^n_y \right \rangle \left \langle h, \varphi^n_z \right \rangle  \left\langle \varphi^n_y \varphi^n_z, \psi^n_x \right \rangle 
\lesssim \sum_\beta  2^{-n(2+\alpha-\beta)} 2^{n(-\frac{d}{2} -\beta)}2^{n \frac{d}{2}} \\
\lesssim 2^{-n(2+\alpha)}
\]
where the first inequality follows by the bounds \eqref{eq:BndAlpha}, \eqref{eq:Bnd-h}, Lemma \ref{lem:WavDecorr} and the fact that the sum over $y,z$ can be removed since for each $x$ there are $O(1)$ $y,z$ in $\Lambda_n$ such that $\psi^n_x, \varphi^m_y, \varphi^p_z$ have overlapping support and, in this case, $|x-y| \lesssim 2^{-n}$. Analogous arguments and Cauchy-Schwarz inequality imply
\begin{eqnarray*}
S_2&=& \sum_{y \in \Lambda_n} \sum_{m \geq n, z \in \Lambda_m}  \left\langle \Pi_x{\color{blue}\mathcal{I}(\Xi)}, \varphi^n_y \right \rangle \left \langle  h, \psi^m_z \right \rangle \left \langle  \varphi^n_y \psi^m_z, \psi^n_x \right \rangle \\
&\lesssim& \sum_\beta 2^{-n(2+\alpha-\beta)} 2^{n(-\frac{d}{2} -\beta)} \sum_{m \geq n, z \in \Lambda_m} \left \langle  h, \psi^m_z \right \rangle 2^{n \frac{d}{2}} 2^{-r'(m-n)} \1_{\{|x-z|\lesssim 2^{-n}\}} \\
&\lesssim& 2^{-n(2+\alpha)} \left(\sum_{m \geq n, z \in \Lambda_m,|x-z|\lesssim 2^{-n}} \left \langle  h, \psi^m_z \right \rangle^2 \right)^{1/2} \left( \sum_{m \geq n} 2^{(d-2r')(m-n)}\right)^{1/2} \lesssim2^{-n(2+\alpha)},
\end{eqnarray*}
Now let us treat $S_3$. 
\begin{align*}
S_3&:= \sum_{z \in \Lambda_n} \sum_{p \geq n, y \in \Lambda_p}  \left\langle h, \varphi^n_z \right \rangle \left \langle \Pi_x{\color{blue}\mathcal{I}(\Xi)}, \psi^p_y \right \rangle \left \langle  \varphi^n_z \psi^p_y, \psi^n_x \right \rangle \lesssim \sum_{p\geq n} 2^{-\frac{pd}{2}} 2^{d(p-n)} \sum_{\beta<\alpha+2} 2^{-n(\alpha+2-\beta)} 2^{\frac{nd}{2}} 2^{-r'(p-n)}\\
&= 2^{-n(\alpha+2)} \sum_{\beta<\alpha+2,p\geq n} 2^{-(r'+\beta-\frac{d}{2})(p-n)}\lesssim 2^{-n(\alpha+2)} \sum_{p\geq n} 2^{-(r'+\alpha-\frac{d}{2})(p-n)}\lesssim 2^{-n(\alpha+2)}
\end{align*}
where the first inequality follows by the bound \eqref{eq:BndAlpha}, Lemma~\ref{lem:WavDecorr}, the fact that for a given $x\in \Lambda_n$ there exist $O(2^{d(p-n)})$ $y\in\Lambda_p$ such that $\psi^n_x$ and $\psi^p_y$ have overlapping support and, for those $y$, $|x-y|\lesssim 2^{-n}$. For the last two bounds we recall that there exist a finite number of $\beta<\alpha+2$ and $r'>-\alpha+\frac{d}{2}$.

In order to deal with $S_4$, we distinguish now two cases: $m\geq p$ and $p\geq m$. At first, we point out that since $\psi^n_x$ and $\psi^m_z$ have overlapping support only if $|z-x|\lesssim 2^{-n}$, we have
\[
\sum_{\gamma<\beta}|z-x|^{\beta-\gamma}\lesssim \sum_{\gamma<\beta} 2^{-n(\beta-\gamma)}\lesssim 1
\]
therefore~\eqref{eq:BndAlpha1} can be bounded by
\begin{equation}\label{eq:BndAlpha2}
2^{-\frac{pd}{2}} \sum_{\beta< \alpha+2}2^{-p\beta}|y-z|^{2+\alpha-\beta} \lesssim 
\begin{cases}
2^{-\frac{pd}{2}-p(\alpha+2)}, &\text{if $m\geq p$}\\
2^{-\frac{pd}{2}} \sum_{\beta< \alpha+2}2^{-p\beta}2^{-m(2+\alpha-\beta)}, &\text{if $p\geq m$}
\end{cases}
\end{equation}
indeed $\psi^p_y$ and $\psi^m_z$ have overlapping support, when $m\geq p$, only if $|z-y|\lesssim 2^{-p}$ while, when $p\geq m$, only if $|z-y|\lesssim 2^{-m}$. 

Case 1: $n\leq p\leq m$. Now, by~\eqref{eq:BndAlpha2}, Lemma~\ref{lem:WavDecorr} and the fact that given $z\in\Lambda_m$ there are $O(2^{d(m-p)})$ $y\in\Lambda_p$, we have
\begin{align*}
\sum_{m\geq n,z\in\Lambda_m}&\left\langle h, \psi^m_z \right \rangle\sum_{\substack{m\geq p\geq n\\y\in\Lambda_p}} \left \langle \Pi_x{\color{blue}\mathcal{I}(\Xi)}, \psi^p_y \right \rangle \left \langle  \psi^m_z \psi^p_y, \psi^n_x \right \rangle\\
&\lesssim 2^{\frac{nd}{2}}\sum_{m\geq n,z\in\Lambda_m} \left\langle h, \psi^m_z \right \rangle \1_{\{|z-x|\lesssim 2^{-n}\}} 2^{-m(r'-d)}\sum_{m\geq p\geq n} 2^{p(r'-(\alpha+2)-\frac{3}{2}d)}\\
&\lesssim 2^{\frac{nd}{2}}\sum_{m\geq n,z\in\Lambda_m} \left\langle h, \psi^m_z \right \rangle \1_{\{|z-x|\lesssim 2^{-n}\}} 2^{-m(r'-d)} 2^{m(r'-(\alpha+2)-\frac{3}{2}d)}\\
&\lesssim 2^{\frac{nd}{2}}\left(\sum_{m\geq n,z\in\Lambda_m,|x-z|\lesssim 2^{-n}} |\left\langle h, \psi^m_z \right \rangle|^2\right)^{\frac{1}{2}}\left(\sum_{m\geq n,z\in\Lambda_m} \1_{\{|z-x|\lesssim 2^{-n}\}} 2^{-2m(\alpha+2+\frac{d}{2})}\right)^{\frac{1}{2}}\\
&\lesssim 2^{\frac{nd}{2}} \left(\sum_{m\geq n} 2^{d(m-n)}2^{-2m(\alpha+2+\frac{d}{2})}\right)^{\frac{1}{2}}\lesssim 2^{-n(\alpha+2)}
\end{align*}
where the second inequality comes from the fact that we can take $r'>\alpha+2+\frac{3}{2}d$ and the third is a direct consequence of Cauchy-Schwarz.

Case 2 : $n\leq m \leq p$. As before, thanks to~\eqref{eq:BndAlpha2}, we have
\begin{align*}
\sum_{m\geq n,z\in\Lambda_m}&\left\langle h, \psi^m_z \right \rangle\sum_{p\geq m, y\in\Lambda_p} \left \langle \Pi_x{\color{blue}\mathcal{I}(\Xi)}, \psi^p_y \right \rangle \left \langle  \psi^m_z \psi^p_y, \psi^n_x \right \rangle\\
&\lesssim \sum_{m\geq n,z\in\Lambda_m}\left\langle h, \psi^m_z \right \rangle\1_{\{|z-x|\lesssim 2^{-n}\}} \sum_{p\geq m} \sum_{\beta<\alpha+2} 2^{-\frac{pd}{2}} 2^{-m(\alpha+2-\beta)} 2^{-p\beta} 2^{\frac{nd}{2}} 2^{-r'(p-m)}\\
&\lesssim 2^{\frac{nd}{2}}\sum_{m\geq n,z\in\Lambda_m}\left\langle h, \psi^m_z \right \rangle\1_{\{|z-x|\lesssim 2^{-n}\}} 2^{-m(\alpha+2)} 2^{-\frac{md}{2}} \sum_{p\geq m} 2^{-(r'+\alpha-\frac{d}{2})(p-m)}\\
&\lesssim 2^{\frac{nd}{2}}\left(\sum_{m\geq n,z\in\Lambda_m,|x-z|\lesssim 2^{-n}} |\left\langle h, \psi^m_z \right \rangle|^2\right)^{\frac{1}{2}}\left(\sum_{m\geq n,z\in\Lambda_m} \1_{\{|z-x|\lesssim 2^{-n}\}} 2^{-2m(\alpha+2+\frac{d}{2})}\right)^{\frac{1}{2}}\lesssim 2^{-n(\alpha+2)}
\end{align*}
where the sum in the second inequality converges since $r'+\alpha-\frac{d}{2}>0$ and the latter is obtained as in Case 1. 
\newline

\noindent In the end,  we have
\[
\left\langle {\Pi}^{e_h}_x{\color{blue}\mathcal{I}(\Xi)H}, \psi^n_x \right \rangle\lesssim 2^{-n(\alpha+2)}\|h\|_{L^2}
\]
which (since $\alpha + 2 \geq 2\alpha +2 + \frac{d}{2}$) concludes the proof of \eqref{eq:BndPih} for $\tau={\color{blue}\mathcal{I}(\Xi)H}$.
\newline

Let us now focus on ${\color{blue}\I(H)\Xi}$. Note that $${\Pi}^{e_h}_x{\color{blue}\mathcal{I} H}(y) = (N \ast h)(y) -  (N \ast h)(x),$$
where $N(z) = \int_0^\infty K(t,z) dt$ (recall that $K$ is taken of compact support so that this integral is finite). Since $N = \bar{N} - R$ with $R$ smooth and $\bar{N}$ the usual Green function of the Laplacian, by classical estimates, 
\begin{equation} \label{eq:SobolevEst}
\left\| N \ast h \right\|_{H^{2}(\T^d)} \lesssim \left\|  h \right\|_{L^2(\T^d)}. 
\end{equation}
Therefore (recalling that all $\psi \in \Psi$ have zero average), for each compact $D$,
\begin{equation} \label{eq:PiIH1}
\sum_{m\geq 0, z\in\Lambda_m \cap D}2^{4m}\left|\left\langle{\Pi}^{e_h}_z{\color{blue}\I(H)},\psi_z^m\right\rangle\right|^2 \;\lesssim \;  \;\|h\|_{L^2}^2
\end{equation}
and for each $\delta$ $\in$ $(0,1)$, 
\begin{equation}  \label{eq:PiIH2}
\sup_{n \geq 0, x \in \Lambda_n} 2^{n(\delta + \frac{d}{2})} \left \langle{\Pi}^{e_h}_x{\color{blue}\I(H)}, \varphi^n_x \right \rangle \lesssim  \left\|N \ast h\right\|_{C^\delta} \lesssim \left\|N \ast h\right\|_{H^2} \lesssim \left\|h\right\|_{L^2} 
\end{equation}
by Sobolev embedding. Moreover, since ${\color{blue}\Xi}\in T_\alpha$ and $\alpha$ is the lowest homogeneity of our regularity structure, by definition of model we have
\[
\left\langle \Pi_x{\color{blue}\Xi}, \psi^p_y \right \rangle\lesssim 2^{-\frac{pd}{2}}2^{-n\alpha} 
\]
At this point we have all the elements we need in order to proceed with the actual bound. We can write
\[
\left\langle{\Pi}^{e_h}_x{\color{blue}\mathcal{I}(H)\Xi},\psi_x^n\right\rangle:=S_1+S_2+S_3+S_4
\] 
where $S_i$'s are the same sums as before. Let us now deal with each of the terms separately.
\begin{align*}
S_1&=\sum_{y,z \in \Lambda_n} \left\langle \Pi_x{\color{blue}\Xi}, \varphi^n_y \right \rangle \left \langle{\Pi}^{e_h}_x{\color{blue}\I(H)}, \varphi^n_z \right \rangle  \left\langle \varphi^n_y \varphi^n_z, \psi^n_x \right \rangle \lesssim \sum_{z\in\Lambda_n}2^{-n\alpha-\frac{nd}{2}}\left \langle{\Pi}^{e_h}_x{\color{blue}\I(H)}, \varphi^n_z \right \rangle 2^{\frac{nd}{2}} \1_{\{|z-x|\lesssim 2^{-n}\}}\end{align*}
using the fact that for a given $x\in\Lambda_n$ there exist $O(1)$ $y,z\in\Lambda_n$ such that $\psi^n_x, \varphi^m_y, \varphi^p_z$ have overlapping support. Now note that for $|x-z|\lesssim 2^{-n}$,
\begin{align*}\left \langle{\Pi}^{e_h}_x{\color{blue}\I(H)}, \varphi^n_z \right \rangle - \left \langle{\Pi}^{e_h}_z{\color{blue}\I(H)}, \varphi^n_z \right \rangle &= \left \langle (N\ast h)(z)-(N\ast h)(x), \varphi^n_z \right \rangle\\ &\lesssim 2^{-n \delta} \|N\ast h\|_{C^\delta}  \left\|\varphi^n_z\right\|_{L^1} \lesssim 2^{-n(\delta + \frac{d}{2})} \|h\|_{L^2}\end{align*}
Hence using also \eqref{eq:PiIH2} and choosing $\delta > 2+\alpha$ we obtain
\begin{align*}
S_1 \lesssim 2^{nd} 2^{-n \alpha} 2^{-n(\delta+d/2)}  \leq 2^{-n(2+2\alpha+ d/2)}. 
\end{align*}
Then
\begin{align*}
S_2&= \sum_{y \in \Lambda_n} \sum_{m \geq n, z \in \Lambda_m}  \left\langle \Pi_x{\color{blue}\Xi}, \varphi^n_y \right \rangle \left \langle   N\ast h, \psi^m_z \right \rangle \left \langle  \varphi^n_y \psi^m_z, \psi^n_x \right \rangle \lesssim 2^{-n\alpha} \sum_{\substack{m\geq n, z\in\Lambda_m\\|z-x|\lesssim 2^{-n}}}\left \langle N\ast h, \psi^m_z \right \rangle 2^{-r(m-n)} \\
&\lesssim 2^{-n\alpha}\left(\sum_{\substack{m\geq n, z\in\Lambda_m\\|z-x|\lesssim 2^{-n}}} 2^{4m}\left|\left \langle N\ast h, \psi^m_z \right\rangle\right|^2\right)^{\frac{1}{2}}\left(\sum_{m\geq n} 2^{-4m} 2^{-(2r-d)(m-n)}\right)^{\frac{1}{2}}\lesssim 2^{-n(\alpha+2)}
\end{align*}
where the last line is justified by Cauchy-Schwarz, the fact that for a given $x\in\Lambda_n$ there exist $O(2^{d(m-n)})$ $z\in\Lambda_m$ such that $\varphi^m_y$ and $\varphi^p_z$ have overlapping support and the fact that we can take $2r>d$.
\begin{align*}
S_3&:= \sum_{z \in \Lambda_n} \sum_{p \geq n, y \in \Lambda_p}  \left\langle N\ast h, \varphi^n_z \right \rangle \left \langle \Pi_x{\color{blue}\Xi}, \psi^p_y \right \rangle \left \langle  \varphi^n_z \psi^p_y, \psi^n_x \right \rangle\lesssim \sum_{\substack{z \in \Lambda_n\\|z-x|\lesssim 2^{-n}}} \left\langle N\ast h, \varphi^n_z \right \rangle 2^{n(r-\frac{d}{2})} \sum_{p\geq n} 2^{-p(r+\alpha-\frac{d}{2})}\\
&\lesssim\sum_{\substack{z \in \Lambda_n\\|z-x|\lesssim 2^{-n}}} \left\langle N\ast h, \varphi^n_z \right \rangle 2^{-n\alpha}\lesssim 2^{-n(\alpha+2)}\left(\sum_{z\in\Lambda_n,|z-x|\lesssim 2^{-n}}2^{4n}\left|\left \langle N\ast h, \varphi^n_z \right \rangle\right|^2\right)^{\frac{1}{2}}\lesssim 2^{-n(\alpha+2)}
\end{align*}
For $S_4$, which is given by
\[
S_4=\sum_{p \geq n, y \in \Lambda_p} \sum_{m \geq n, z \in \Lambda_m}  \left\langle \Pi_x{\color{blue}\Xi}, \psi^p_y \right \rangle \left \langle N\ast h, \psi^m_z \right \rangle \left \langle \psi^p_y \psi^m_z, \psi^n_x \right \rangle,
\]
we will split the sum considering first the case $m\geq p$ and the case $p\geq m$ then. 

Case 1: $n\leq p\leq m$.
\begin{align*}
\sum_{m \geq n, z \in \Lambda_m}&\left \langle N\ast h, \psi^m_z \right \rangle\sum_{m\geq p\geq n,y\in\Lambda_p}\left\langle \Pi_x{\color{blue}\Xi}, \psi^p_y \right \rangle\left \langle \psi^p_y \psi^m_z, \psi^n_x \right \rangle\\
&\lesssim 2^{\frac{nd}{2}}\sum_{\substack{m \geq n, z \in \Lambda_m\\|z-x|\lesssim 2^{-n}}}\left \langle N\ast h, \psi^m_z \right \rangle\sum_{m\geq p\geq n} 2^{d(m-p)}2^{-p\alpha-\frac{pd}{2}-r(m-p)}\\
&=2^{\frac{nd}{2}}\sum_{\substack{m \geq n, z \in \Lambda_m\\|z-x|\lesssim 2^{-n}}}\left \langle N\ast h, \psi^m_z \right \rangle 2^{-m(r-d)}\sum_{m\geq p\geq n} 2^{p(r-\alpha-\frac{3}{2}d)}\lesssim 2^{\frac{nd}{2}}\sum_{\substack{m \geq n, z \in \Lambda_m\\|z-x|\lesssim 2^{-n}}}\left \langle N\ast h, \psi^m_z \right \rangle 2^{-m(\alpha+\frac{d}{2})}\\
&\lesssim 2^{\frac{nd}{2}}\left(\sum_{\substack{m \geq n, z \in \Lambda_m\\|z-x|\lesssim 2^{-n}}} 2^{4m}\left|\left \langle N\ast h, \psi^m_z \right\rangle\right|^2\right)^{\frac{1}{2}}\left(\sum_{m\geq n} 2^{d(m-n)}2^{-2m(\alpha+2)} 2^{-md}\right)^{\frac{1}{2}}\lesssim 2^{-n(\alpha+2)}
\end{align*}
Case 2: $n\leq m\leq p$.
\begin{align*}
\sum_{m \geq n, z \in \Lambda_m}&\left \langle N\ast h, \psi^m_z \right \rangle\sum_{p\geq m,y\in\Lambda_p}\left\langle \Pi_x{\color{blue}\Xi}, \psi^p_y \right \rangle\left \langle \psi^p_y \psi^m_z, \psi^n_x \right \rangle\\
&\lesssim 2^{\frac{nd}{2}}\sum_{\substack{m \geq n, z \in \Lambda_m\\|z-x|\lesssim 2^{-n}}}\left \langle N\ast h, \psi^m_z \right \rangle 2^{m(r-d)}\sum_{p\geq m} 2^{-p(r+\alpha-\frac{d}{2})}\lesssim 2^{\frac{nd}{2}}\sum_{\substack{m \geq n, z \in \Lambda_m\\|z-x|\lesssim 2^{-n}}}\left \langle N\ast h, \psi^m_z \right \rangle 2^{-m(\alpha+\frac{d}{2})}\\
&\lesssim 2^{\frac{nd}{2}}\left(\sum_{\substack{m \geq n, z \in \Lambda_m\\|z-x|\lesssim 2^{-n}}} 2^{4m}\left|\left \langle N\ast h, \psi^m_z \right\rangle\right|^2\right)^{\frac{1}{2}}\left(\sum_{m\geq n} 2^{d(m-n)}2^{-2m(\alpha+2)} 2^{-md}\right)^{\frac{1}{2}}\lesssim 2^{-n(\alpha+2)}
\end{align*}
where all the passages can be justified exploiting the same arguments we carried out above. Hence, we can conclude that
\[
\left\langle{\Pi}^{e_h}_x{\color{blue}\mathcal{I}(H)\Xi},\psi_x^n\right\rangle\lesssim 2^{-n(2 \alpha+2-\frac{d}{2})}.
\]
The proof of \eqref{eq:BndPih} for $\tau={\color{blue}\mathcal{I}(H)H}$ follows from the exact same argument as for ${\color{blue}\mathcal{I}(H)\Xi}$, using that $h$ $\in$ $C^\alpha$.

\end{proof}

\begin{proof}[Proof of Lemma~\ref{lem:Xn}] Given two models $\Pi$ and $\Pi'\in\M(\TS_g)$  define
$$\left\llbracket \Pi - \Pi'\right\rrbracket^2 := \sup_{(\tau,n,x) \in I} \left\langle \Pi_x \tau - \Pi'_x \tau, \varphi_x^n\right\rangle^2 2^{2n (|\tau| + \frac{d}{2})},$$
where $I = \left\{{\color{blue}\Xi}, {\color{blue}\mathcal{I} (\Xi) \Xi}\right\} \times \left\{ (n,x), \; n \geq 0, x \in \Lambda_n\right\}$. Then $\llbracket \cdot\rrbracket$ induces on $\mathcal{M}$ the same topology as $\VERT\cdot\VERT$.

For a fixed $\Pi^0$, we first claim that $\omega \mapsto \left\llbracket \Pi(\omega) - \Pi^0\right\rrbracket^2$ is $C^1_{\mathcal{H}}$, $\mathbb{P}$-a.e.. Indeed, first note that for a fixed $(\tau,n,x)$,  $\omega \mapsto \left\langle \Pi_x(\omega) \tau - \Pi^0_x \tau, \varphi_x^n\right\rangle$ is in $C^1_{\mathcal{H}}$. This is easy to see, for instance  one can check directly with the help of the computations in the proof of Lemma \ref{lem:newlemApp} that it has a $\mathcal{H}$-derivative, given by $\left\langle \Pi_x(\omega) \tau^H, \varphi_x^n\right\rangle$, where $ \tau^H={\color{blue}H}$ if $\tau = {\color{blue}\Xi}$, and $ \tau^H={\color{blue}\mathcal{I} (\Xi)  H} +{\color{blue}\mathcal{I} (H)  \Xi}$ if $\tau = {\color{blue}\I(\Xi)\Xi}$. 
Now taking a smaller $I$ if necessary, one can assume that the $\left\langle \Pi_x \tau - \Pi^0_x \tau, \varphi_x^n\right\rangle^2 2^{2n (|\tau| + \frac{d}{2})}$ are pairwise distinct random variables. Since they are elements of a fixed chaos, this actually implies that 
$$(\tau,n,x) \neq (\tau',n',x') \;\Rightarrow \; \mathbb{P} \left( \left\langle \Pi_x \tau - \Pi^0_x \tau, \varphi_x^n\right\rangle^2 2^{2n (|\tau| + \frac{d}{2})} = \left\langle \Pi_{x'} \tau' - \Pi^0_{x'} \tau', \varphi_{x'}^{n'}\right\rangle^2  2^{2n' (|\tau'| + \frac{d}{2})}\right) = 0.$$
In addition $ \left\langle \Pi_x \tau - \Pi^0_x \tau, \varphi_x^n\right\rangle 2^{n (|\tau| + \frac{d}{2})}$ goes to $0$ as $n\to \infty$ for models in the closure of smooth functions, so that the supremum in the definition of $\left\llbracket \Pi(\omega) - \Pi^0\right\rrbracket$ is $\mathbb{P}$-a.e. attained at a single $(\tau,x,n)$. Recalling that if $\ell^\infty_0$ is the set of sequences going to $0$ as $n \to \infty$, the map $(u_n)_{n\geq 0} \in \ell^\infty_0 \mapsto (\sup_{n \geq 0} u_n)$ is Fr\'echet differentiable at each sequence $(u_n)$ attaining its supremum at a single point, this proves the claim.

Now fix a smooth function $\psi$ on $[0,+\infty)$ such that $\psi \equiv 1$ on $[0,1]$ and $\psi\equiv 0$ on $[2,+\infty)$. Then, for fixed $\Pi^0\in\M(\TS_g)$, $\delta^0>0$, define a function $F_{\Pi^0,\delta^0}$ on $\M(\TS_g)$ by 
\[
F_{\Pi^0,\delta^0}(\Pi) = \psi\left(\frac{1}{\delta^0}\llbracket \Pi- \Pi^0\rrbracket \right).
\]
Then take a sequence $(\Pi^n)$ which is dense in $\{t<T_\infty\}$. Since $\{t<T_\infty\}$ is open, there exists $\delta^n>0$ such that $\left\{ \Pi\in\M(\TS_g) : \;  \left\llbracket \Pi - \Pi^n\right\rrbracket \leq 3 \delta_n \right\} \subset \{t<T_\infty\}$. Then $X_n(\omega) = F_{\Pi^n,\delta^n}(\hat{\Pi}(\omega))$ satisfies the required conditions. Indeed, it is clear by Theorem \ref{thm:MallDiff} that $FX_n$ $\in$ $C^1_{\mathcal{H}-loc}$ on $\{t<T_\infty\}$, with $D(FX_n) = F DX_n + X_n DF$. On the other hand, from the properties of $\psi$ it is clear that if $\omega \notin \{t<T_\infty\}$, then $F_{\Pi^n,\delta^n}(\Pi') = 0$ for $\llbracket\Pi'-\Pi^n\rrbracket \leq \delta_n$. In particular, for $\|h\|_{\Hi}$ small enough, $X_n(\omega+h) =0$, so that $FX_n$ is also $\mathcal{H}$-differentiable at $\omega$, with $D(FX_n)=0$. \hfill
\end{proof}

\section{Admissible Models and Consistency}\label{remark}

We collect in this Appendix the proof of Remark~\ref{remark:1} and Lemma~\ref{lemma:consistency}.

\begin{proof}[Proof of Remark~\ref{remark:1}]
We want to show that the second analytical bound in~\eqref{eq:ModelAna} automatically hold, given that the pair $(\Pi,\Gamma)$ satisfies the first as well as the equalities~\eqref{eq:Amodel} and~\eqref{eq:AmodelAlg1}. Notice that in the specific context of $\TS_g$ we have an explicit expression for $\Gamma_{xy}\tau$ for all $\tau\in T_g$. While for $\tau={\color{blue}\Xi},{\color{blue}1},{\color{blue}X_i}$ it is oblvious, for the others we see that
\begin{align*}
\Gamma_{xy}{\color{blue}\I(\Xi)\Xi}&={\color{blue}\I(\Xi)\Xi}+(f_y(\J(\Xi))-f_x(\J(\Xi))){\color{blue}\Xi}\\
\Gamma_{xy}{\color{blue}X_i\Xi}&={\color{blue}X\Xi}+(f_y(X_i)-f_x(X_i)){\color{blue}\Xi}\\
\Gamma_{xy}{\color{blue}\I(\Xi)}&={\color{blue}\I(\Xi)}+(f_y(\J(\Xi))-f_x(\J(\Xi))){\color{blue}1}
\end{align*}
Now, $\|\Gamma_{xy}{\color{blue}X_i\Xi}\|_\alpha\lesssim \|x-y\|$ is an obvious consequence of~\eqref{eq:AmodelPol}, while, for the others, showing that $\|\Gamma_{xy}\tau\|_\beta\lesssim \|x-y\|^{|\tau|-\beta}$, boils down to prove 
\[
|f_y(\J(\Xi))-f_x(\J(\Xi))|=\left|\langle\Pi_x{\color{blue}\Xi},K(x-\cdot)\rangle- \langle\Pi_y{\color{blue}\Xi},K(y-\cdot)\rangle\right|\lesssim \|x-y\|^{\alpha+2}
\]
But this bound is a bound on the map $\Pi$ itself and can be easily obtained upon using the decomposition of the kernel, splitting the cases $\|x-y\|\lesssim 2^{-1}$ and $\|x-y\|>2^{-1}$ and applying Proposition A.1 in~\cite{Hai} and the first analytical bound in~\eqref{eq:ModelAna}.
\end{proof}

In the proof of lemma~\ref{lemma:consistency}, we will make use of the actual definition of the reconstruction map (see Theorem 3.10 in~\cite{Hai} and, in particular, the bound (3.3)) and of the abstract heat kernel (see equations (5.11), (5.15), (5.16) and (7.7) in~\cite{Hai}), but since this is the only point in which we will actually need them, we refrain in this context from thoroughly explaining their structure and address the interested reader to the quoted reference. 

\begin{proof}[Proof of Lemma~\ref{lemma:consistency}]
The first point in the previous is a direct consequence of the properties of $\tau_H$. More precisely, it is due to the fact that according to Remark~\ref{rem:Invariance}, $\tau_H$ leaves the homogeneity invariant, and that, by construction, is linear and multiplicative. 

For the second, the image of $\TS_g$ through the canonical immersion map is a sector of $\TS_g^H$. Since, the extended model coincides with the original one on $\TS_g$, the bound (3.3) in the Reconstruction theorem in~\cite{Hai} guarantees that also the two reconstruction maps coincide on the elements of $\D^\gamma$ for any $\gamma>0$, and consequently, thanks to Proposition 6.9 in~\cite{Hai}, on the ones of $\D^{\gamma,\eta}$ for $\eta\in(-2,\gamma]$. Analogously, one can prove that $\RS^{e_h} \tau_H(U^h)=\RS^h U^h$, for any $U^h\in\D^{\gamma,\eta}(\Gamma^h)$. Indeed, by definition $\Pi^h=\Pi^{e_h}\tau_H$ and, once again, the bound (3.3) in~\cite{Hai} leads to the conclusion. 

The last point is essentially a consequence of the previous. On the one side, by construction, $\tau_H$ and $\I$ commute, on the other $\tau_H$ leaves the polynomials invariant, hence the only thing to check is that the coefficients of the polynomials in $\tilde{\PK}^H\tau_H(U^h)$ and the ones of $\tilde{\PK}^h(U^h)$ coincide (for the first equality exactly the same argument applies). As can be seen from equations (5.11), (5.15), (5.16) and (7.7) in~\cite{Hai}, for the first they are of the following form
\[
\langle \Pi_x^{e_h}\tau_H(\tau),D^{(k)}K(x-\cdot)\rangle,\,\,\langle \RS^{e_h}\tau_H(U^h)- \Pi_x^{e_h}\tau_H(U^h)(x),D^{(k)}K(x-\cdot)\rangle,\,\,\langle \RS^{e_h}\tau_H(U^h),D^{(k)}R(x-\cdot)\rangle
\]
and, since $\Pi_x^{e_h}\tau_H=\Pi_x^h$ and $\RS^{e_h} \tau_H(U^h)=\RS^h U^h$, they must coincide with the ones for $\tilde{\PK}^h(U^h)$, thus concluding the proof. 
\end{proof}

\end{document}